\pdfoutput=1
\documentclass[10pt]{article}

\usepackage{natbib}                   %
\usepackage{fullpage}                   %
\usepackage{authblk}                    %
\usepackage{url}
\usepackage{booktabs,longtable,array,ragged2e}
\usepackage{xparse}

\usepackage{amsmath}
\usepackage{amssymb}
\usepackage{mathtools}
\usepackage{algorithm}
\usepackage{algorithmic}
\usepackage{amsthm}
\usepackage{enumitem}
\usepackage{pdfrender}
\usepackage{graphicx}
\usepackage{subcaption}
\usepackage{booktabs}
\usepackage{float}
\usepackage{microtype}
\usepackage{placeins}
\usepackage{xcolor}
\usepackage{pgfplots}
\usepgfplotslibrary{groupplots}
\usetikzlibrary{calc}
\pgfplotsset{compat=1.18}
\usepackage{booktabs}

\newtheorem{theorem}{Theorem}[section]
\newtheorem{lemma}[theorem]{Lemma}
\newtheorem{proposition}[theorem]{Proposition}
\newtheorem{corollary}[theorem]{Corollary}

\theoremstyle{definition}
\newtheorem{assumption}[theorem]{Assumption}

\theoremstyle{remark}

\newcommand{\R}{\mathbb{R}}
\newcommand{\E}{\mathbb{E}}

\newcommand{\rbb}{\mathbb{R}}  %
\newcommand{\ebb}{\mathbb{E}}  %

\newcommand{\scal}{\mathcal{S}}

\newcommand{\ip}[2]{\langle#1, #2 \rangle}

\newcommand{\bigO}{\raisebox{-0.1ex}{$O$}}
\newcommand{\bigOtilde}{\raisebox{-0.15ex}{$\widetilde{O}$}}

\usepackage{hyperref}                   %
\definecolor{darkblue}{rgb}{0.0, 0.0, 0.5}  %

\hypersetup{
    colorlinks=true,
    linkcolor=darkblue,
    anchorcolor=darkblue,
    citecolor=darkblue,
    filecolor=darkblue,
    menucolor=darkblue,
    urlcolor=darkblue
}

\title{Adaptive Federated Optimization \\ Across Smooth and Nonsmooth Regimes}
\author[1,$\ast$]{Zimeng Wang}
\author[2,$\ast$]{Ali Dadras}
\author[3]{Alp Yurtsever}
\affil[1,2,3]{Department of Mathematics and Mathematical Statistics, Umeå University, Sweden}
\affil[1]{\texttt{zimeng.wang@umu.se}}
\affil[2]{\texttt{ali.dadras@umu.se}}
\affil[3]{\texttt{alp.yurtsever@umu.se}}
\date{}

\begin{document}

\maketitle
{\renewcommand{\thefootnote}{\fnsymbol{footnote}}\footnotetext[1]{Equal contribution.}}

\begin{abstract}
AdaGrad-style methods can simultaneously adapt to both the order and constants of smoothness in centralized optimization.
However, existing extensions to federated optimization with heterogeneous client objectives and multiple local steps require the smoothness order to be specified in advance.
We address this limitation with AdaFed, a federated method that adapts across several regimes under a single choice of parameters.
In the deterministic setting, AdaFed achieves an $\bigO(1/T)$ convergence rate for both smooth nonconvex and smooth convex objectives, and an $\bigOtilde(1/\sqrt{T})$ rate for Lipschitz nonsmooth convex objectives.
Under a light-tailed stochastic-gradient noise assumption, these guarantees extend to the stochastic setting with additional logarithmic factors and noise-dependent terms.
Experiments across the three problem classes further support our theoretical findings.
\end{abstract}

\section{Introduction}
\label{sec:introduction}

Federated learning (FL) aims to train a shared model from data distributed across many clients without requiring the clients to transmit their raw data. The resulting optimization problem is often formulated as
\begin{equation}
    \min_{x \in \R^d} F(x):= \frac{1}{m}\sum_{i=1}^m F_i(x).
    \label{eq:FL-prob}
\end{equation}
Here $F_i:\R^d\to\R$ is the local loss associated with the $i$-th client, which in stochastic settings can be expressed as:
\[
F_i(x)=\E_{\xi \sim \mathcal{D}_i}[\ell_i(x;\xi)],
\] 
and the data distributions $\mathcal{D}_i$ can differ across clients. To reduce communication cost when $m$ is large, classic algorithms such as FedAvg \citep{mcmahan2017communication,li2020fedavg} allow each participating client to perform multiple local stochastic-gradient steps before its update is aggregated at the server. When the difference among the data distributions $\{\mathcal{D}_i\}_{i=1}^m$ is significant, methods such as SCAFFOLD \citep{karimireddy2020scaffold} and FedVARP \citep{jhunjhunwala2022fedvarp} address the local drift arising from multiple local steps through control variates or server-side variance reduction. However, both theoretical guarantees and practical behavior of these methods depend on step-size schedules whose effective choices vary with the problem class and its related parameters (e.g., smoothness, gradient noise, heterogeneity, etc.).

In the centralized setting, adaptive gradient methods offer a natural way to overcome this reliance on prior problem knowledge. The pioneering method AdaGrad \citep{duchi2011adaptive} adjusts its step sizes adaptively using accumulated gradient information. Beyond empirical success, AdaGrad-type adaptive methods \citep{levy2017online,levy2018online,kavis2019unixgrad,ene2021adaptive} also exhibit an important theoretical advantage: a single update rule can achieve \emph{universal} (in the sense of \citet{nesterov2015universal}) convergence by attaining the rates associated with smooth and nonsmooth objectives without knowing in advance the problem class or smoothness parameter.

However, transferring this universal convergence property to FL settings is highly nontrivial, where algorithms couple client-side local updates with server-side aggregation. 
Server-level adaptation alone is insufficient to guarantee universal convergence because the main optimization steps occur locally on clients, and their resulting updates directly shape server aggregation.
Client-level adaptation can also fail because individual client gradients need not vanish even at a stationary point of $F$. Hence naively accumulating the client gradients during local updates may cause the step size to continuously decay.
Several adaptive FL methods have been proposed in the literature. Specifically, server-level methods include FedOpt \citep{reddi2021adaptive} and FedAMS/FedCAMS \citep{wang2022communication}, client-level methods include $\Delta$-SGD \citep{kim2024adaptive} and FedSPS \citep{mukherjee2023locally}, and methods operating at both levels include Mime \citep{karimireddy2021mime}, FAFED \citep{wu2023faster}, and FedAda$^2$ \citep{lee2026efficient}. 
Despite these efforts, the convergence results established in these works are tailored to specific problem classes, and some guarantees still require knowledge of smoothness parameters. Consequently, none of these works addresses the question central to this work:

\emph{Can a single adaptive FL method, without prior knowledge of the problem class or related parameters, automatically attain the convergence rates associated with smooth nonconvex, smooth convex, and Lipschitz convex objectives?}

\subsection{Our contributions}
We answer this question affirmatively with a new AdaGrad-type FL method. Our contributions are summarized as follows.
\begin{itemize}[leftmargin=1.5em, topsep=0.5em, itemsep=0.25em]
    \item \textbf{Algorithm.} We develop a new adaptive FL method for solving~\eqref{eq:FL-prob} that accommodates heterogeneous client objectives, stochastic local gradients, and multiple local updates per communication round. Its key feature is a scaled AdaGrad-type step size that is shared across clients and fixed throughout each round, and is updated after server aggregation based on the server displacement.

    \item \textbf{Universal convergence.}  The same algorithm automatically attains appropriate convergence rates for smooth nonconvex, smooth convex, and Lipschitz convex objectives, without prior knowledge of the problem class or associated parameters. These rates match those of classical AdaGrad in centralized settings up to logarithmic factors, and benefit jointly from the number of clients and multiple local updates (see Table~\ref{tab:results} for a summary). To the best of our knowledge, this is the first universal convergence guarantee in FL that covers all three problem classes with a single algorithm and without problem-dependent tuning.

    \item \textbf{Further extensions.} We extend the algorithm to independent, nonuniform partial participation and establish universal convergence with exact local gradients across the same three problem classes; see Appendix~\ref{app:partial-participation}. We also prove that, with exact local gradients, the full-participation algorithm automatically adapts to H\"older smoothness of any order $\alpha\!\in\![0,1]$ for convex objectives, attaining the corresponding $\smash{\bigOtilde(T^{-(1+\alpha)/2})}$ rate; see Appendix~\ref{app:holder-convex}.
\end{itemize}

\begin{table}[ht!]
\centering
\caption{Summary of convergence results for Algorithm~\ref{alg:FL-adagrad}. $\bigOtilde(\cdot)$ hides logarithmic factors.}
\label{tab:results}
\renewcommand{\arraystretch}{1.1}
\setlength{\tabcolsep}{5pt}
\begin{tabular}{@{}ccc l l l@{}}
\hline
\multicolumn{3}{c}{\textbf{Problem class}} & \textbf{Gradient type} & \textbf{Assumptions} & \textbf{Convergence} \\
\cline{1-3}
\textbf{Smooth} & \textbf{Lipschitz} & \textbf{Convex} & & & \\
\hline\\[-1.0em]
$\checkmark$ & $\times$ & $\times$ & Exact & Assm.~\ref{ass:lower-bounded} & $\bigO\big(\frac{1}{T}\big)$ \\[0.15em]
$\checkmark$ & $\times$ & $\times$ & Stochastic & Assms.~\ref{ass:lower-bounded},~\ref{ass:stochastic-noise} & $\bigOtilde\big(\frac1T+\frac{\sigma}{\sqrt{mKT}}\big)$ \\[0.5em]
$\checkmark$ & $\times$ & $\checkmark$ & Exact & Assm.~\ref{ass:lower-bounded} & $\bigO\big(\frac{1}{T}\big)$ \\[0.15em]
$\checkmark$ & $\times$ & $\checkmark$ & Stochastic & Assms.~\ref{ass:lower-bounded},~\ref{ass:stochastic-noise} & $\bigOtilde\big(\frac1{T}+\frac{\sigma}{\sqrt{mKT}}\big)$ \\[0.5em]
$\times$ & $\checkmark$ & $\checkmark$ & Exact & Assm.~\ref{ass:lower-bounded} & $\bigOtilde\big(\frac{1}{\sqrt{T}}\big)$ \\[0.15em]
$\times$ & $\checkmark$ & $\checkmark$ & Stochastic & Assms.~\ref{ass:lower-bounded},~\ref{ass:stochastic-noise} & $\bigOtilde\big(\frac{1}{\sqrt{T}}+\frac{\sigma}{\sqrt{mKT}}\big)$ \\[0.3em]
\hline
\end{tabular}
\end{table}

Compared to centralized analyses of AdaGrad~\citep{duchi2011adaptive,ward2020adagrad,wang2023adagrad}, FL creates two additional difficulties: local drift from multiple client updates and delayed accumulation due to unavailable aggregated gradient information. To handle these challenges, we link local drift to the server displacement, thereby reducing the analysis to server updates. We then bound this displacement almost surely by the maximal noise, whose moments admit logarithmic bounds, and use the result with a shifted summation lemma to control delayed accumulation.
Notably, our analysis removes bounded-domain or bounded-iterate assumptions, which are commonly used in AdaGrad-type analyses~\citep{levy2017online,levy2018online,kavis2019unixgrad,ene2021adaptive}.

The remainder of the paper is organized as follows. Section~\ref{sec:related-work} reviews related work. Section~\ref{sec:algorithm-design} presents the algorithm. Section~\ref{sec:convergence} gives the main convergence results and a proof sketch. Section~\ref{sec:numerical-experiments} reports the numerical experiments and Section~\ref{sec:conclusion} concludes the paper. All technical proofs are deferred to Appendix~\ref{app:preliminary}--\ref{app:lipschitz-convex}. Extensions to partial participation and H\"older smoothness are established in Appendix~\ref{app:partial-participation} and~\ref{app:holder-convex}, respectively. 
\section{Related Work}
\label{sec:related-work}

\noindent\textbf{Universal adaptive optimization.}
AdaGrad \citep{duchi2011adaptive} introduced adaptive step sizes based on accumulated gradients, with popular variants including Adam \citep{kingma2015adam} and AMSGrad \citep{reddi2018convergence}. Scalar AdaGrad steps have since been analyzed for stochastic and nonconvex optimization under progressively weaker assumptions. \citet{kavis2022high,ward2020adagrad} provide high-probability guarantees under uniformly bounded gradients, which is later extended in \citet{attia2023sgd,faw2022power,wang2023adagrad} with relaxed assumptions. For convex optimization, AcceleGrad \citep{levy2018online}, universal accelerated methods \citep{ene2021adaptive,kavis2019unixgrad,levy2018online} improve over AdaGrad by achieving near-optimal rates simultaneously for convex smooth and nonsmooth objectives. Nesterov's universal method \citep{nesterov2015universal} instead uses line search. Further extensions treat inexact stochastic oracles \citep{rodomanov2024universality} and H\"older smoothness \citep{orabona2023normalized}. More closely related, \citet{wang2026universal} replace gradient accumulation with gradient-mapping accumulation and obtain a method for smooth nonconvex, smooth convex, and nonsmooth convex composite problems.

\vspace{1em}

\noindent\textbf{Adaptive federated optimization.}
Existing adaptive FL methods differ mainly in where the adaptive step size is maintained. At the server, FedOpt \citep{reddi2021adaptive} applies AdaGrad, Adam, or Yogi to aggregated client displacements, while FedAMS/FedCAMS \citep{wang2022communication} add AMSGrad-style stabilization. Other server-side variants include FedUR \citep{zhang2023fedur}, FedDuA \citep{takakura2025feddua}, FedAgg \citet{yuan2025aggregated}, and the asynchronous method FADAS \citep{wang2024fadas}. At the client level, Local AdaAlter \citep{xie2020local} synchronizes AdaGrad denominators, whereas $\Delta$-SGD \citep{kim2024adaptive} follows a different type of adaptivity \citep{malitsky2020adaptive} to adapt each client's step size to its local smoothness. Other local adaptive methods include FedSPS \citep{mukherjee2023locally} and FedLADA \citep{sun2023efficient}. Methods acting at both levels include FAFED \citep{wu2023faster} and FedAda$^2$ \citep{lee2026efficient}, which coordinate server statistics, preconditioners, or global corrections with local adaptation. However, their analyses largely focus on one specific problem class (mostly smooth nonconvex or smooth convex), and some results require smoothness-dependent step-size conditions (e.g., FedOpt \citep{reddi2021adaptive} and FAFED \citep{wu2023faster}). In contrast, our method uses one unchanged rule to adapt across three problem classes, without knowledge of problem-related parameters.
A more detailed summary of existing adaptive FL methods is provided in Appendix~\ref{app:adaptive-fl-assumptions}.

\section{Algorithm Design}
\label{sec:algorithm-design}

\noindent\textbf{Notation.}
We use $\|\cdot\|$ for the Euclidean norm and denote $[n]:=\{1, 2, \dots, n\}$. %
For a stochastic algorithm, we denote by $\ebb_t[\cdot]:=\ebb\left[\cdot \mid \mathcal{F}_{t}\right]$ the expectation conditioned on $\mathcal{F}_{t}$, which contains all the information available at step $t$. Depending on the differentiability of $F$, we use $\nabla F$ to denote the gradient or a subgradient of $F$.

\vspace{1em}

\noindent\textbf{Algorithm overview.} We present the adaptive federated learning method in Algorithm~\ref{alg:FL-adagrad}. It initializes the server model at $x_1$ and the shared step size at $\eta_1=\eta/\gamma$. At each round $t$, the server broadcasts $(x_t,\eta_t,\tau_t)$ to all $m$ clients. Each client $i$ starts from $x_{i,0}^t=x_t$, performs $K$ local (stochastic) gradient\,/subgradient steps with step size $\tau_t\eta_t$, and returns its local displacement $\smash{\Delta_i^t=x_{i,K}^t-x_t}$ to server. The server then computes the new model and uses it to update the step size.

\begin{algorithm}[t]
    \caption{Adaptive Federated Learning Method (AdaFed)}\label{alg:FL-adagrad}
    \begin{algorithmic}
    \REQUIRE{} initial model $x_1\in\R^d$; local-step count $K\geq1$; server iterations $T\geq1$; prescribed local scales $\tau_t>0$; parameters $\gamma,\eta>0$.
    \STATE{} Initialize the step size $\eta_1:=\eta/\gamma$.
    \FOR{$t = 1, 2, \dots, T$}
    \STATE{} \textbf{Server} broadcasts $(x_t,\eta_t,\tau_t)$. %
    \FOR{each \textbf{client} $i\in[m]$ \textbf{in parallel}}
    \STATE{} Initialize $x_{i,0}^t:=x_t$.
    \FOR{$k=0,\dots,K-1$}
    \STATE{} Update $x_{i,k+1}^t:=x_{i,k}^t-\tau_t\eta_t g_{i,k}^t$ with $g_{i,k}^t\approx\nabla F_i(x_{i,k}^t)$.
    \ENDFOR{}
    \STATE{} Set $\Delta_i^t:=x_{i,K}^t-x_t$ and send $\Delta_i^t$ to the server.
    \ENDFOR{}
    \STATE{} \textbf{Server} updates the model and the step size:
    \begin{align*}
        x_{t+1}:=x_t+\frac1{mK\tau_t}\sum_{i=1}^m\Delta_i^t,\qquad \frac{1}{\eta_{t+1}^2}:=\frac{1}{\eta_t^2}\left(1+\frac{\|x_{t+1}-x_t\|^2}{\eta^2}\right).
    \end{align*}
    \ENDFOR{}
    \STATE{} \textbf{Output:} $\bar{x}_T$ chosen uniformly at random from $\{x_t\}_{t=1}^T$.
    \end{algorithmic}
\end{algorithm}

Note that the scales $\{\tau_t\}$ are introduced to mitigate local drift due to multiple local updates. %
At the server, the factor $1/\tau_t$ in the aggregation cancels this scaling, so that $\tau_t$ does not directly affect the server-level update. 
Indeed, let $q_t:=1/\eta_t$, and define
\begin{equation*}
    \bar g_t^{(i)}:=\frac1K\sum_{k=0}^{K-1}g_{i,k}^t,\qquad \bar g_t:=\frac1m\sum_{i=1}^m\bar g_t^{(i)}.
\end{equation*}
Then, since $\Delta_i^t=-(K\tau_t/q_t)\bar g_t^{(i)}$, the server update in Algorithm~\ref{alg:FL-adagrad} can be written as
\begin{equation}
\tag{Core-FP}
x_{t+1}=x_t-\frac{\bar g_t}{q_t},\qquad q_{t+1}^2=q_t^2+\frac{\|\bar g_t\|^2}{\eta^2}.
\label{eq:nocv-outer-method}
\end{equation}

\vspace{1em}

\noindent\textbf{Motivation for the adaptive step size.} 
Individual client gradients generally do not vanish at a stationary point of the global objective. As a result, accumulating their squared norms can force the adaptive step size to keep decaying even as the global gradient vanishes, preventing the method from exploiting global smoothness. We therefore construct the adaptive accumulator from the aggregated server direction $\bar g_t$. In the special case $K=1$ with exact local gradients, $\bar{g}_t=\nabla{F}(x_t)$, and the recursion in \eqref{eq:nocv-outer-method} reduces directly to a centralized AdaGrad-type update on the global objective. With multiple local updates, however, $\bar{g}_t$ consists of gradients evaluated along different client trajectories and is no longer the global gradient at $x_t$. Additionally, the adaptive accumulator is delayed since the server displacement $x_{t+1}-x_t$ is unavailable during clients' local updates in round $t$. Establishing that this adaptive mechanism retains universal convergence despite these discrepancies is one of the main technical challenges of our analysis.

\section{Convergence Analysis}
\label{sec:convergence}
In this section, we present the convergence results of Algorithm~\ref{alg:FL-adagrad}.  We start by making the following standard assumption on problem~\eqref{eq:FL-prob}.
\begin{assumption}\label{ass:lower-bounded}
$F$ is lower bounded, i.e, $F^\star:=\inf_{x\in\rbb^d} F(x)>-\infty$. For convex $F$, we further assume that the minimum is attained at some point $x^\star\in\rbb^d$.
\end{assumption}
Let $\mathcal F_t$ contain the history before round $t$, including $(x_t,q_t)$. For client $i$, let $\mathcal F_{t,k}^{(i)}$ contain $\mathcal F_t$ and its first $k$ oracle calls in round $t$, so $x_{i,k}^t$ is $\mathcal F_{t,k}^{(i)}$-measurable. Define the local stochastic noise 
\[
\xi_{i,k}^t:=g_{i,k}^t-\nabla F_i(x_{i,k}^t).
\]
For the stochastic setting, we make the following assumption.
\begin{assumption}
\label{ass:stochastic-noise}
The noise variables $\{\xi_{i,k}^t\}$ are conditionally independent and satisfy
\begin{equation*}
    \E[\xi_{i,k}^t\mid\mathcal F_{t,k}^{(i)}]=0,\qquad \E\left[\exp\left(\frac{\|\xi_{i,k}^t\|^2}{\sigma^2}\right)\middle|\mathcal F_{t,k}^{(i)}\right]\leq\exp(1)
\end{equation*}
for all $t\ge 1$, $i\in[m]$, and $k\in\{0,\ldots,K-1\}$.
\end{assumption}
Assumption~\ref{ass:stochastic-noise} is stated for $\sigma>0$. When $\sigma=0$, it means $\xi_{i,k}^t=0$ almost surely. This assumption is also used in~\citep{li2019convergence,nemirovski2009robust}. It allows unbounded noise but is stronger than bounded conditional variance. Indeed, under Assumption~\ref{ass:stochastic-noise}, the convexity of $\exp(\cdot)$ implies
\begin{equation}
    \E[\|\xi_{i,k}^t\|^2\mid\mathcal F_{t,k}^{(i)}]\leq\sigma^2.
    \label{eq:sg-implied-variance}
\end{equation}
However, the bounded variance assumption alone is insufficient to guarantee convergence of Algorithm~\ref{alg:FL-adagrad}, as indicated by the following proposition. The proof is deferred to Appendix~\ref{app:proof-couterexample}.
\begin{proposition}[Bounded variance is insufficient]\label{prop:fs-variance-obstruction}
Fix any $m,K\geq1$. Let $\{\tau_t\}$ be chosen arbitrarily and set $\gamma=\eta=1$. Consider the following convex, $1$-smooth, and $1$-Lipschitz client objectives:
\[
    F_i(x)=F(x):=\sqrt{1+x^2}-1, \quad i=1,\ldots,m.
\]
If $x_1=0$, then for any $T$ divisible by 4, there exist independent noises satisfying $\E[\xi_{i,k}^t\mid\mathcal F_{t,k}^{(i)}]=0$ and $\E[|\xi_{i,k}^t|^2\mid\mathcal F_{t,k}^{(i)}]\leq1$ such that
\[
    \E[|F'(\bar x_T)|^2]\geq\frac1{16384mK^2},\qquad \E[F(\bar x_T)-F^\star]\geq\frac{\sqrt T}{2048mK^2}.
\]
\end{proposition}
The counterexample exploits rare but large noise realizations permitted under bounded variance. A single such realization can move the server iterate far from the minimizer, and also produce a large server displacement before it is incorporated into the adaptive accumulator. The resulting increase in the accumulator then makes subsequent steps too small for the iterates to recover within the remaining rounds. %

\subsection{Main results}
We are now ready to state our main convergence results, which establish universal convergence for three distinct problem classes: smooth nonconvex, smooth convex, and Lipschitz convex. For ease of presentation, we define the following quantities:
\begin{align}
    &\mathcal U_T:=1+\sigma\sqrt{\frac{T}{mK}}+\sigma^2\bigl(1+\log(mKT)\bigr)\log\bigl(1+\sigma^2(1+\log(mKT))\bigr),\notag\\
    &\mathcal C_T:=\bigl(1+\sigma^2(1+\log(mKT))\bigr)\bigl(1+\log\mathcal U_T\bigr),\label{def:mathcal_C_T}\\
    &\mathcal V_T:=\left(1+\sigma^2(1+\log(mKT))\right)\left(1+\log\big(\sqrt{T}+\sigma\sqrt{T/(mK)}\big)\right)+(1+\sigma)\log T.\label{def:mathcal_V_T}
\end{align}
We use $\lesssim$ to hide constants depending only on fixed quantities (e.g., $L, G, \gamma, \eta$) hereafter, i.e., $A\lesssim B$ if there exists a constant $c>0$ such that $A\leq c B$. We also use $\bigOtilde(\cdot)$ to hide polylogarithmic factors.

\begin{theorem}[Universal convergence of Algorithm~\ref{alg:FL-adagrad}]\label{thm:unified-convergence}
Let Assumptions~\ref{ass:lower-bounded} and~\ref{ass:stochastic-noise} hold. Let $\{x_t\}_{t\geq1}$ be the iterates generated by Algorithm~\ref{alg:FL-adagrad} with arbitrary $\gamma,\eta>0$ and $\tau_t$ chosen as
\begin{equation}\label{eq:partial-tau}
\tau_t=1~\text{if}~K=1; \qquad \tau_t\leq\frac{1}{(K-1)t}~\text{if}~K\geq2.
\end{equation} 
Then, for any $T\geq1$, we have the following results:

\vspace{0.5em}

\noindent\textbf{(i) Smooth nonconvex.} If each $F_i$ is $L$-smooth, then
\begin{equation*}
    \E[\|\nabla F(\bar x_T)\|^2]\lesssim\frac{\mathcal C_T^2}{T}+\frac{\sigma\mathcal C_T}{\sqrt{mKT}}.
\end{equation*}

\noindent\textbf{(ii) Smooth convex.} If each $F_i$ is $L$-smooth and convex, then
\begin{equation*}
    \E[F(\bar x_T)-F^\star]\lesssim\frac{\mathcal C_T^2}{T}+\frac{\sigma\mathcal C_T}{\sqrt{mKT}}.
\end{equation*}

\noindent\textbf{(iii) Lipschitz convex.} If each $F_i$ is convex and $G$-Lipschitz, then
\begin{equation*}
    \E[F(\bar x_T)-F^\star]\lesssim\left(\frac1{\sqrt T}+\frac{\sigma}{\sqrt{mKT}}\right)\mathcal V_T.
\end{equation*}
\end{theorem}
Case~(i) is proved as Theorem~\ref{thm:sg-smooth-nonconvex} in Appendix~\ref{app:smooth-nonconvex}, case~(ii) as Theorem~\ref{thm:sg-smooth-convex} in Appendix~\ref{app:smooth-convex}, and case~(iii) as Theorem~\ref{thm:sg-lipschitz-convex} in Appendix~\ref{app:lipschitz-convex}. 
In fact, for the smooth convex case and $K\geq2$, we can use the more relaxed choice $\tau_t\leq\frac{\tau}{(K-1)t^p}$ for any $p>1/2$ and $\tau>0$ without degrading the convergence rate, see Theorem~\ref{thm:sg-smooth-convex} for details.

To provide more explicit convergence rates in the stochastic setting, note that $\mathcal C_T\lesssim \log^2(mKT)$ and $\mathcal V_T\lesssim \log^2(mKT)$. Consequently, the three convergence rates in Theorem~\ref{thm:unified-convergence} can be simplified to
\begin{equation}\label{eq:thm-rates}
\bigOtilde\!\left(\frac{1}{T}+\frac{\sigma}{\sqrt{mKT}}\right),\quad \bigOtilde\!\left(\frac{1}{T}+\frac{\sigma}{\sqrt{mKT}}\right), \quad \bigOtilde\!\left(\frac{1}{\sqrt T}+\frac{\sigma}{\sqrt{mKT}}\right),
\end{equation}
respectively.
In every setting, the rate matches that of classical AdaGrad in the non-federated case (up to logarithmic factors captured by $\mathcal C_T$ and $\mathcal V_T$), while the $\sigma/\sqrt{mKT}$ term shows that the convergence benefits jointly from the number of clients $m$ and local updates $K$.

\subsection{Proof sketch}
Here we provide a proof sketch for the smooth nonconvex case, which also serves as the basis of the analysis for convex cases. Besides stochasticity, the main challenges in the analysis arise from \emph{local drift} and \emph{delayed accumulator}, which are absent in the classical AdaGrad analysis.

Note that $\bar g_t$ averages gradients along moving local paths, whereas stationarity is measured by $g_t:=\nabla F(x_t)$. To estimate their difference, we decompose 
$
\bar g_t=g_t+\hat b_t+\bar\xi_t,
$
where $\bar\xi_t$ denotes average of the stochastic errors and $\hat b_t$ is the local gradient drift. Then it follows that
\begin{equation}
    \E_t[\|\bar g_t\|^2]\leq2\|g_t\|^2+2\E_t[\|\hat b_t\|^2]+\frac{2\sigma^2}{mK}.
    \label{eq:gt_sq-main}
\end{equation}
To further bound $\E_t[\|\hat b_t\|^2]$, we unroll the local updates and obtain the local distance $\E_t[\|x_{i,k}^t-x_t\|^2]\lesssim(k\tau_t/q_t)^2(\|\nabla F_i(x_t)\|+\sigma)^2$, which together with $L$-smoothness then implies
\begin{equation*}
    \E_t[\|\hat b_t\|^2]\leq\frac{L^2}{mK}\sum_{i=1}^m\sum_{k=0}^{K-1}\E_t[\|x_{i,k}^t-x_t\|^2]\lesssim\frac{\tau_t^2}{mq_t^2}\sum_{i=1}^m\bigl(\|\nabla F_i(x_t)\|+\sigma\bigr)^2.
\end{equation*}
To proceed,  we use smoothness to write $\|\nabla F_i(x_t)\|\leq\|\nabla F_i(x_1)\|+L\|x_t-x_1\|$ and then bound sum of $\tau_t^2\|x_t-x_1\|^2/q_t$. Set $r_t:=x_{t+1}-x_t$ and $\mathcal R_T:=\sum_{t=1}^T\|r_t\|^2$. Since $x_t-x_1=\sum_{j<t}r_j$, Hardy's inequality gives $ \sum_{t=1}^T\tau_t^2\|x_t-x_1\|^2/q_t^2\lesssim\mathcal R_T$. Consequently, we obtain
\begin{equation}\label{eq:bt-sum-main}
    \sum_{t=1}^T\E_t[\|\hat b_t\|^2]\lesssim1+\mathcal R_T.\\[-0.25em]
\end{equation}
Note that when $K=1$, $\hat b_t=0$ exactly.

Since~\eqref{eq:bt-sum-main} converts the local gradient drift into server displacement, we may now focus on the server-level analysis, where the natural next step is to analyze $\mathcal R_T$. However, delayed accumulator directly introduces a significant challenge: $q_t$ does not contain $\bar g_t$, so $\|r_t\|=\|\bar g_t\|/q_t$ is not automatically bounded, which prevents us to apply an AdaGrad-type summation property (cf. Lemma~\ref{lem:summation}). To address this issue, we use the decomposition $\bar g_t=\bar g_{t-1}+(g_t-g_{t-1})+(b_t-b_{t-1})$ to write
\[
    \|r_t\|=\frac{\|\bar g_t\|}{q_t}\leq\frac{\|\bar g_{t-1}\|}{q_t}+\frac{\|g_t-g_{t-1}\|}{q_t}+\frac{\|b_t\|+\|b_{t-1}\|}{q_t},
\] 
where $b_t=\hat b_t+\bar \xi_t$. Then the first term is bounded trivially due to $q_t\ge \|\bar g_{t-1}\|/\eta$. Moreover, smoothness and~\eqref{eq:nocv-outer-method} also gives $\|g_t-g_{t-1}\|/q_t\leq L\|x_t-x_{t-1}\|/q_t=L\|\bar g_{t-1}\|/(q_{t-1}q_t)\leq L\eta/q_1$. To bound the last term, let $M_T:=\max_{t\leq T,i,k}\|\xi_{i,k}^t\|$ be the largest noise, then by bounding $\hat b_t$ similarly as before and noting that $\bar \xi_t\leq M_T$, we can show $\|b_t\|\lesssim 1+M_T$. Combining these bounds together gives
\begin{equation*}
    \|r_t\|\lesssim 1+M_T.
\end{equation*}
This means $\|r_t\|$ is controlled \emph{almost surely} for all $t\leq T$, which allows us to apply the summation property in Lemma~\ref{lem:summation} to obtain
\begin{equation*}
    \mathcal R_T\lesssim\bigl(1+M_T^2\bigr)\log\frac{q_{T+1}}{q_1}.
\end{equation*}
Assumption~\ref{ass:stochastic-noise} implies $(\E[M_T^{2p}])^{1/p}=\bigOtilde(\sigma^2)$ for $p\geq1$, which further gives $\E[\mathcal R_T]=\bigOtilde\left(\log\E[q_{T+1}]\right)=\bigOtilde(1)$. This property matches the classic AdaGrad up to logarithmic factors and reveals the reason why Algorithm~\ref{alg:FL-adagrad} preserves the rate of AdaGrad. In contrast, bounded variance assumption~\eqref{eq:sg-implied-variance} may instead lead to a weaker linear reliance $\E[M_T^{2}]\lesssim\sigma^2 T$.

Let $\mathcal G_T:=\sum_{t=1}^T\|g_t\|^2/q_t$ and $\zeta_t:=\bar g_t-\E_t[\bar g_t]$. Smoothness and decomposition $\bar g_t=g_t+\E_t[\hat b_t]+\zeta_t$  yield
\begin{equation*}
    F(x_{t+1})\leq F(x_t)-\frac{\|g_t\|^2}{2q_t}+\frac{\E_t[\|\hat b_t\|^2]}{2q_t}-\frac{\langle g_t,\zeta_t\rangle}{q_t}+\frac{L}{2}\|r_t\|^2.
\end{equation*}
Summing this inequality and plugging~\eqref{eq:bt-sum-main}, we obtain $\E[\mathcal G_T]\lesssim 1+\E[\mathcal R_T]$ since $\zeta_t$ is unbiased. To bound $\E[q_{T+1}]$, we use $q_{T+1}^2=q_1^2+\sum_{t=1}^T\|\bar g_t\|^2/\eta^2$ and apply~\eqref{eq:gt_sq-main} and~\eqref{eq:bt-sum-main} to obtain
\begin{align*}
    \E[q_{T+1}]\lesssim1+\sigma\sqrt{\frac{T}{mK}}+\E[\mathcal G_T]+\E[\mathcal R_T],
\end{align*}
where we also use $\E\big[(\sum_{t=1}^T\|g_t\|^2)^{1/2}\big]\leq\E[\sqrt{q_{T+1}\mathcal G_T}]\leq\sqrt{\E[q_{T+1}]}\sqrt{\E[\mathcal G_T]}$ to solve a self-bounding inequality for $\E[q_{T+1}]$.
Then substituting the established bounds on $\E[\mathcal G_T]$ and $\E[\mathcal R_T]$, we derive $\E[q_{T+1}]\lesssim\mathcal U_T$, which further gives $\E[\mathcal R_T]\lesssim\mathcal C_T$ and $\E[\mathcal G_T]\lesssim\mathcal C_T$.

It remains to bound the second moments $\E[\mathcal G_T^2]$ and $\E[q_{T+1}^2]$. Using the 4-th momentum $\E[M_T^{4}]$ and similar arguments for estimating $\E[\mathcal R_T]$, we can show $\sqrt{\E[\mathcal R_T^2]}\lesssim\mathcal C_T$, which further implies $\sqrt{\E[\mathcal G_T^2]}\lesssim\mathcal C_T$. Then by similarly using $q_{T+1}^2=q_1^2+\sum_{t=1}^T\|\bar g_t\|^2/\eta^2$ and applying~\eqref{eq:gt_sq-main} and~\eqref{eq:bt-sum-main}, we obtain
\[
    \E[q_{T+1}^2]\lesssim 1+\frac{T\sigma^2}{mK}+\E[\mathcal G_T^2]+\E[\mathcal R_T]\lesssim\mathcal C_T^2+\frac{T\sigma^2}{mK}.
\]
Finally, $q_t\leq q_{T+1}$ and Cauchy--Schwarz imply
\begin{equation*}
    \E[\|\nabla F(\bar x_T)\|^2]=\frac1T\sum_{t=1}^T\E[\|g_t\|^2]\leq\frac{1}{T}\sqrt{\E[q_{T+1}^2]}\sqrt{\E[\mathcal G_T^2]}\lesssim\frac{\mathcal C_T^2}{T}+\frac{\sigma\mathcal C_T}{\sqrt{mKT}}.
\end{equation*} %
\section{Numerical Experiments}
\label{sec:numerical-experiments}

In this section, we evaluate the empirical performance of Algorithm~\ref{alg:FL-adagrad} with the primary goal of examining its adaptation to the underlying problem class. In particular, our experiments test whether a common parameterization adapts across the smooth nonconvex, smooth convex, and nonsmooth convex regimes, and whether increasing the number of local steps improves performance under stochastic gradients, as predicted by Theorem~\ref{thm:unified-convergence}. The implementation of AdaFed is available at \url{https://github.com/DadrasAli/AdaFed}.

For all experiments, each local objective is a logistic loss with a regime-specific penalty. Specifically, if client $i$ holds $n_i$ samples $(a_{ij},y_{ij})$ with $y_{ij}\in\{-1,+1\}$, we set
\begin{equation}
  F_i(x)=\frac{1}{n_i}\sum_{j=1}^{n_i}
  \log\!\left(1+\exp(-y_{ij}a_{ij}^{\top}x)\right) + \lambda r(x).
  \label{eq:experiment-objective}
\end{equation}
We set $\lambda\!=\!0.01$ throughout, with 
$r(x)\!=\!\|x\|_1$ for the nonsmooth convex case, 
$\frac12\|x\|_2^2$ for the smooth convex case, and 
$\sum_j x_j^2/(1+x_j^2)$ in the smooth nonconvex case.

We use 32,000 examples from the LIBSVM \citep{chang2011libsvm} IJCNN1 binary-classification dataset ($d=22$), partitioned equally among $m=32$ clients. To induce strong data heterogeneity, we group examples with similar per-example gradients at a centralized reference solution into the same client using balanced $k$-means clustering, so that each client holds 1,000 examples
and the local objectives differ substantially across clients. Additional datasets and experimental details are provided in Appendix~\ref{app:additional-experiments}.

For evaluation, we report $\|\nabla F(x)\|^2$ for the smooth nonconvex
problem and the objective residual $F(x)-F^\star$ for the convex problems.
The reference value $F^\star$ is computed to high accuracy using a
centralized solver (damped Newton for the smooth convex problem and FISTA
followed by active-set Newton for the nonsmooth convex problem).
Following Theorem~\ref{thm:unified-convergence}, the expected value for the uniform-random output metric is computed by averaging the corresponding metric along the
trajectory. We additionally report the last-iterate metric for reference.

\textbf{Adaptation across problem classes and local steps.}~
We first examine whether Algorithm~\ref{alg:FL-adagrad} adapts to the underlying problem class under a common parameterization and how its behavior changes with the number of local steps.
Figure~\ref{fig:rates-ijcnn-b1-pgf} shows the results for
$K\in\{1,10,100\}$, with exact local gradients in the top row and single-sample stochastic gradients in the bottom row.
Under exact gradients, the two smooth regimes exhibit the faster $1/T$ behavior, whereas the nonsmooth convex regime exhibits the slower $\smash{1/\sqrt{T}}$ behavior, consistently across all three choices of $K$.
Under stochastic gradients, increasing $K$ delays the transition from the faster $1/T$ behavior to the noise-dominated $\smash{1/\sqrt{KT}}$ regime in the two smooth settings. 
In the nonsmooth convex setting, increasing $K$ instead improves the constant while preserving the $\smash{1/\sqrt{T}}$ behavior.
These observations are consistent with the dependence on the problem class, $T$, and $K$ predicted by Theorem~\ref{thm:unified-convergence}.

\begin{figure}[t]
\centering
\includegraphics[width=\linewidth]{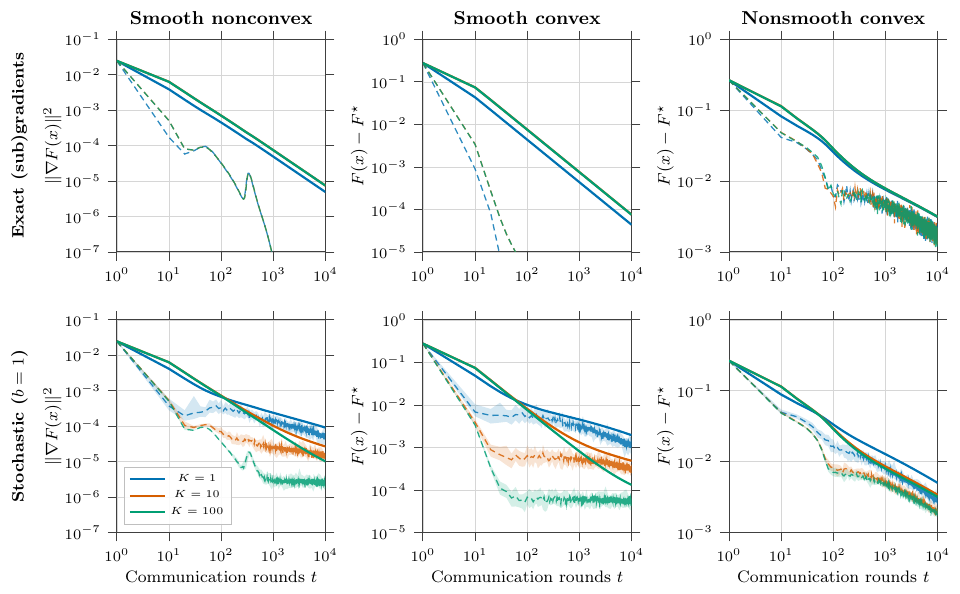}
\caption{Performance of AdaFed on IJCNN1 across the three problem classes. Solid and dashed curves show the uniform-random and last-iterate metrics, respectively. Stochastic results show the mean over twenty seeds, with shaded regions indicating the upward and downward RMS deviations.}
\label{fig:rates-ijcnn-b1-pgf}
\end{figure}

\textbf{Comparison with adaptive FL methods.}~
Figure~\ref{fig:comparison-selected-ijcnn1-32-pgf} compares Algorithm~\ref{alg:FL-adagrad} with representative methods using server-side, client-side, and both server- and client-side adaptation:
FedAdagrad~\citep{reddi2021adaptive}, FedLADA~\citep{sun2023efficient}, and FAFED~\citep{wu2023faster}, respectively.

We use $K=10$ local steps with full client participation and consider both exact local gradients and stochastic local gradients (minibatch size $b=32$). 
Each method's main parameters are selected from short tuning runs, then fixed for the 10,000-round comparison (see Appendix~\ref{app:comparison-protocol} for the tuning details). 

Across all three problem classes, Algorithm~\ref{alg:FL-adagrad} is competitive with the baselines while consistently following the rate behavior associated with the corresponding regime. 
In contrast, the baselines stagnate in some settings. 
This distinction is particularly clear where the last-iterate metrics of all three baselines reach certain accuracy floors under both exact and stochastic gradients, while Algorithm~\ref{alg:FL-adagrad} continues to decrease. 
These results further illustrate the cross-regime adaptation predicted by Theorem~\ref{thm:unified-convergence}.

\textbf{Additional experiments.}~
Appendix~\ref{app:additional-experiments} provides additional experiments on synthetic problems, larger minibatches, parameter sensitivity for $\eta$ and $\gamma$, and partial participation. 
It also extends the comparison to FedAdam, FedUR, Local AdaAlter, FedSPS, and FedAda2, covering a broader range of adaptive FL methods.

\begin{figure}[t]
\centering
\includegraphics[width=\linewidth]{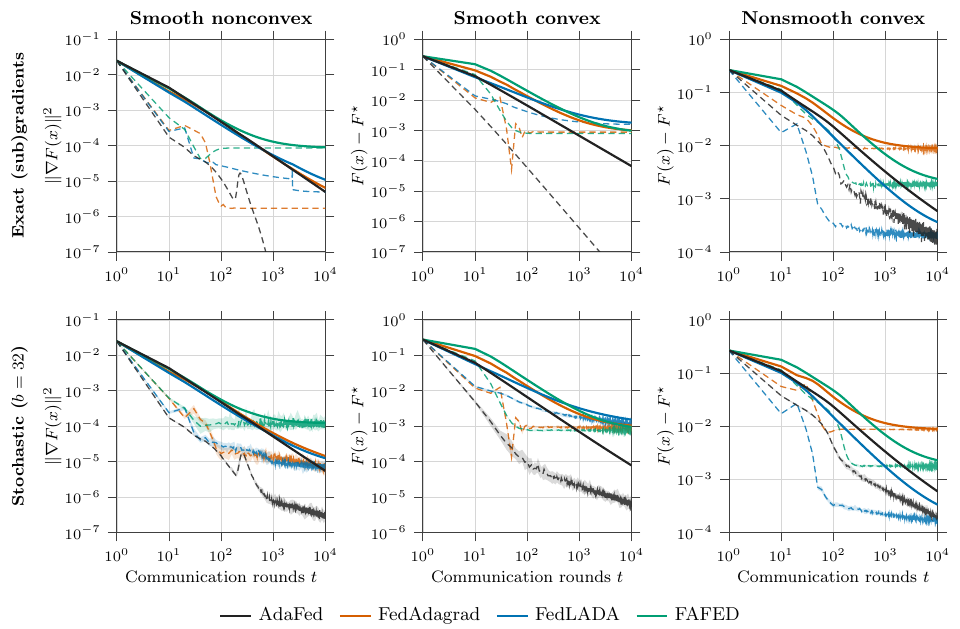}
\caption{Performance of AdaFed compared with
representative adaptive FL methods on IJCNN1 with $K=10$ local steps. 
FedAdagrad, FedLADA, and FAFED represent server-, client-, and both-level
adaptation, respectively. 
Solid and dashed curves show the uniform-random and last-iterate metrics, respectively. 
Stochastic results show the mean over twenty seeds, with shaded regions indicating the upward and downward RMS deviations.}
\label{fig:comparison-selected-ijcnn1-32-pgf}
\end{figure}

\FloatBarrier
\section{Conclusion}
\label{sec:conclusion}

Our results show that universal convergence can be achieved in federated learning (FL) by using AdaGrad-type step sizes, despite heterogeneous client objectives and multiple local updates. The key is to accumulate the server displacements, which reflect global rather than client-level stationarity from a consensus viewpoint of the FL problem. Additionally, the scaled local step controls client drift without altering the aggregated server update. From the analytical perspective, relating the local drift to server displacements and controlling the delayed accumulator yields the essential argument without bounded-iterate assumptions, which results in convergence rates comparable to centralized AdaGrad across the three problem classes.
Our experiments support these theoretical findings, showing the predicted cross-regime behavior under a common parameterization and competitive performance against existing adaptive FL methods. 
We further provide extensions to nonuniform partial participation and H\"older-smoothness in the exact local-gradient setting in the appendix.

Several directions remain open. A natural next step is to extend both the nonuniform partial-participation and H\"older-smooth analyses to stochastic local gradients. This may be possible using moment estimates similar to those in the full-participation analysis, but the details are nontrivial and require careful investigation. Another question is whether the logarithmic factors caused by delayed accumulation can be removed through a sharper analysis or a modified update rule. Finally, combining the proposed method with acceleration may lead to improved convergence rates while preserving universality for convex problems.

\section*{Acknowledgments}
This work was supported by the Swedish Research Council under registration number 2023-05476. This work was partially supported by the Wallenberg AI, Autonomous Systems and Software Program (WASP) funded by the Knut and Alice Wallenberg Foundation. Computational resources were provided on the Berzelius system funded by the Knut and Alice Wallenberg Foundation and operated by NAISS.

\bibliographystyle{abbrvnat}
\bibliography{references}

\begin{thebibliography}{44}
\providecommand{\natexlab}[1]{#1}
\providecommand{\url}[1]{\texttt{#1}}
\expandafter\ifx\csname urlstyle\endcsname\relax
  \providecommand{\doi}[1]{doi: #1}\else
  \providecommand{\doi}{doi: \begingroup \urlstyle{rm}\Url}\fi

\bibitem[Attia and Koren(2023)]{attia2023sgd}
A.~Attia and T.~Koren.
\newblock {SGD} with {AdaGrad} stepsizes: Full adaptivity with high probability
  to unknown parameters, unbounded gradients and affine variance.
\newblock In \emph{Proceedings of the 40th International Conference on Machine
  Learning}, 2023.

\bibitem[Chang and Lin(2011)]{chang2011libsvm}
C.-C. Chang and C.-J. Lin.
\newblock Libsvm: A library for support vector machines.
\newblock \emph{ACM transactions on intelligent systems and technology (TIST)},
  2\penalty0 (3):\penalty0 1--27, 2011.

\bibitem[Chen et~al.(2025)Chen, Wang, Zhang, Han, Liu, and
  Brinton]{chen2025gradient}
E.~Chen, S.~Wang, J.~Zhang, D.-J. Han, C.~Liu, and C.~Brinton.
\newblock Gradient correction in federated learning with adaptive optimization.
\newblock \emph{arXiv preprint arXiv:2502.02727}, 2025.

\bibitem[Chezhegov et~al.(2024)Chezhegov, Skorik, Khachaturov, Shalagin,
  Avetisyan, Tak{\'a}{\v{c}}, Kholodov, and Beznosikov]{chezhegov2024local}
S.~Chezhegov, S.~Skorik, N.~Khachaturov, D.~Shalagin, A.~Avetisyan,
  M.~Tak{\'a}{\v{c}}, Y.~Kholodov, and A.~Beznosikov.
\newblock Local methods with adaptivity via scaling.
\newblock \emph{arXiv preprint arXiv:2406.00846}, 2024.

\bibitem[Duchi et~al.(2011)Duchi, Hazan, and Singer]{duchi2011adaptive}
J.~Duchi, E.~Hazan, and Y.~Singer.
\newblock Adaptive subgradient methods for online learning and stochastic
  optimization.
\newblock \emph{Journal of Machine Learning Research}, 12:\penalty0 2121--2159,
  2011.

\bibitem[Durrett(2019)]{durrett2019probability}
R.~Durrett.
\newblock \emph{Probability: Theory and Examples}.
\newblock Number~49 in Cambridge Series in Statistical and Probabilistic
  Mathematics. Cambridge University Press, Cambridge, 5 edition, 2019.

\bibitem[Ene et~al.(2021)Ene, Nguyen, and Vladu]{ene2021adaptive}
A.~Ene, H.~L. Nguyen, and A.~Vladu.
\newblock Adaptive gradient methods for constrained convex optimization and
  variational inequalities.
\newblock In \emph{Proceedings of the AAAI Conference on Artificial
  Intelligence}, 2021.

\bibitem[Faw et~al.(2022)Faw, Tziotis, Caramanis, Mokhtari, Shakkottai, and
  Ward]{faw2022power}
M.~Faw, I.~Tziotis, C.~Caramanis, A.~Mokhtari, S.~Shakkottai, and R.~Ward.
\newblock The power of adaptivity in {SGD}: Self-tuning step sizes with
  unbounded gradients and affine variance.
\newblock In \emph{Proceedings of the Thirty-Fifth Conference on Learning
  Theory}, 2022.

\bibitem[Jhunjhunwala et~al.(2022)Jhunjhunwala, Sharma, Nagarkatti, and
  Joshi]{jhunjhunwala2022fedvarp}
D.~Jhunjhunwala, P.~Sharma, A.~Nagarkatti, and G.~Joshi.
\newblock {FedVARP}: Tackling the variance due to partial client participation
  in federated learning.
\newblock In \emph{Proceedings of the Thirty-Eighth Conference on Uncertainty
  in Artificial Intelligence}, 2022.

\bibitem[Karimireddy et~al.(2020)Karimireddy, Kale, Mohri, Reddi, Stich, and
  Suresh]{karimireddy2020scaffold}
S.~P. Karimireddy, S.~Kale, M.~Mohri, S.~Reddi, S.~Stich, and A.~T. Suresh.
\newblock {SCAFFOLD}: Stochastic controlled averaging for federated learning.
\newblock In \emph{Proceedings of the 37th International Conference on Machine
  Learning}, 2020.

\bibitem[Karimireddy et~al.(2021)Karimireddy, Jaggi, Kale, Mohri, Reddi, Stich,
  and Suresh]{karimireddy2021mime}
S.~P. Karimireddy, M.~Jaggi, S.~Kale, M.~Mohri, S.~J. Reddi, S.~U. Stich, and
  A.~T. Suresh.
\newblock {Mime}: Mimicking centralized stochastic algorithms in federated
  learning.
\newblock In \emph{International Conference on Learning Representations}, 2021.

\bibitem[Kavis et~al.(2019)Kavis, Levy, Bach, and Cevher]{kavis2019unixgrad}
A.~Kavis, K.~Y. Levy, F.~Bach, and V.~Cevher.
\newblock {UniXGrad}: A universal, adaptive algorithm with optimal guarantees
  for constrained optimization.
\newblock In \emph{Advances in Neural Information Processing Systems}, 2019.

\bibitem[Kavis et~al.(2022)Kavis, Levy, and Cevher]{kavis2022high}
A.~Kavis, K.~Y. Levy, and V.~Cevher.
\newblock High probability bounds for a class of nonconvex algorithms with
  {AdaGrad} stepsize.
\newblock In \emph{International Conference on Learning Representations}, 2022.

\bibitem[Kim et~al.(2024)Kim, Toghani, Uribe, and Kyrillidis]{kim2024adaptive}
J.~L. Kim, M.~T. Toghani, C.~Uribe, and A.~Kyrillidis.
\newblock Adaptive federated learning with auto-tuned clients.
\newblock In \emph{International Conference on Learning Representations}, 2024.

\bibitem[Kingma and Ba(2015)]{kingma2015adam}
D.~P. Kingma and J.~Ba.
\newblock Adam: A method for stochastic optimization.
\newblock In \emph{International Conference on Learning Representations}, 2015.

\bibitem[Lee et~al.(2026)Lee, Sharma, Zaheer, and Li]{lee2026efficient}
S.~H. Lee, S.~Sharma, M.~Zaheer, and T.~Li.
\newblock Efficient adaptive federated optimization.
\newblock In \emph{Advances in Neural Information Processing Systems}, 2026.

\bibitem[Levy(2017)]{levy2017online}
K.~Levy.
\newblock Online to offline conversions, universality and adaptive minibatch
  sizes.
\newblock In \emph{Advances in Neural Information Processing Systems}, 2017.

\bibitem[Levy et~al.(2018)Levy, Yurtsever, and Cevher]{levy2018online}
K.~Y. Levy, A.~Yurtsever, and V.~Cevher.
\newblock Online adaptive methods, universality and acceleration.
\newblock In \emph{Advances in Neural Information Processing Systems}, 2018.

\bibitem[Li et~al.(2024)Li, Huang, and Huang]{li2024fedda}
J.~Li, F.~Huang, and H.~Huang.
\newblock Fedda: Faster adaptive gradient methods for federated constrained
  optimization.
\newblock In \emph{The Twelfth International Conference on Learning
  Representations}, 2024.

\bibitem[Li and Orabona(2019)]{li2019convergence}
X.~Li and F.~Orabona.
\newblock On the convergence of stochastic gradient descent with adaptive
  stepsizes.
\newblock In \emph{Proceedings of the Twenty-Second International Conference on
  Artificial Intelligence and Statistics}, 2019.

\bibitem[Li et~al.(2020)Li, Huang, Yang, Wang, and Zhang]{li2020fedavg}
X.~Li, K.~Huang, W.~Yang, S.~Wang, and Z.~Zhang.
\newblock On the convergence of {FedAvg} on non-{IID} data.
\newblock In \emph{International Conference on Learning Representations}, 2020.

\bibitem[Liu et~al.(2026)Liu, Panda, Pandey, Cook, Slota, Wang, Chen, and
  Xu]{liu2026lodadac}
W.~Liu, A.~Panda, U.~Pandey, H.~Cook, G.~M. Slota, N.~Wang, J.~Chen, and Y.~Xu.
\newblock Lodadac: a unified local training-based decentralized framework with
  adaptive gradients and compressed communication.
\newblock \emph{arXiv preprint arXiv:2604.09970}, 2026.

\bibitem[Malitsky and Mishchenko(2020)]{malitsky2020adaptive}
Y.~Malitsky and K.~Mishchenko.
\newblock Adaptive gradient descent without descent.
\newblock In \emph{Proceedings of the 37th International Conference on Machine
  Learning}, 2020.

\bibitem[McMahan et~al.(2017)McMahan, Moore, Ramage, Hampson, and Ag{\"u}era~y
  Arcas]{mcmahan2017communication}
H.~B. McMahan, E.~Moore, D.~Ramage, S.~Hampson, and B.~Ag{\"u}era~y Arcas.
\newblock Communication-efficient learning of deep networks from decentralized
  data.
\newblock In \emph{Proceedings of the 20th International Conference on
  Artificial Intelligence and Statistics}, 2017.

\bibitem[Mukherjee et~al.(2023)Mukherjee, Loizou, and
  Stich]{mukherjee2023locally}
S.~Mukherjee, N.~Loizou, and S.~U. Stich.
\newblock Locally adaptive federated learning.
\newblock \emph{arXiv preprint arXiv:2307.06306}, 2023.

\bibitem[Nemirovski et~al.(2009)Nemirovski, Juditsky, Lan, and
  Shapiro]{nemirovski2009robust}
A.~Nemirovski, A.~Juditsky, G.~Lan, and A.~Shapiro.
\newblock Robust stochastic approximation approach to stochastic programming.
\newblock \emph{SIAM Journal on optimization}, 19\penalty0 (4):\penalty0
  1574--1609, 2009.

\bibitem[Nesterov(2015)]{nesterov2015universal}
Y.~Nesterov.
\newblock Universal gradient methods for convex optimization problems.
\newblock \emph{Mathematical Programming}, 152\penalty0 (1--2):\penalty0
  381--404, 2015.
\newblock \doi{10.1007/s10107-014-0790-0}.

\bibitem[Orabona(2023)]{orabona2023normalized}
F.~Orabona.
\newblock Normalized gradients for all.
\newblock \emph{arXiv preprint arXiv:2308.05621}, 2023.

\bibitem[Reddi et~al.(2021)Reddi, Charles, Zaheer, Garrett, Rush,
  Kone{\v{c}}n{\'y}, Kumar, and McMahan]{reddi2021adaptive}
S.~Reddi, Z.~Charles, M.~Zaheer, Z.~Garrett, K.~Rush, J.~Kone{\v{c}}n{\'y},
  S.~Kumar, and H.~B. McMahan.
\newblock Adaptive federated optimization.
\newblock In \emph{International Conference on Learning Representations}, 2021.

\bibitem[Reddi et~al.(2018)Reddi, Kale, and Kumar]{reddi2018convergence}
S.~J. Reddi, S.~Kale, and S.~Kumar.
\newblock On the convergence of {Adam} and beyond.
\newblock In \emph{International Conference on Learning Representations}, 2018.

\bibitem[Rodomanov et~al.(2024)Rodomanov, Jiang, and
  Stich]{rodomanov2024universality}
A.~Rodomanov, X.~Jiang, and S.~Stich.
\newblock Universality of {AdaGrad} stepsizes for stochastic optimization:
  Inexact oracle, acceleration and variance reduction.
\newblock In \emph{Advances in Neural Information Processing Systems}, 2024.

\bibitem[Sun et~al.(2023)Sun, Shen, Sun, Ding, and Tao]{sun2023efficient}
Y.~Sun, L.~Shen, H.~Sun, L.~Ding, and D.~Tao.
\newblock Efficient federated learning via local adaptive amended optimizer
  with linear speedup.
\newblock \emph{IEEE Transactions on Pattern Analysis and Machine
  Intelligence}, 45\penalty0 (12):\penalty0 14453--14464, 2023.

\bibitem[Takakura et~al.(2025)Takakura, Liew, and Hasegawa]{takakura2025feddua}
S.~Takakura, S.~P. Liew, and S.~Hasegawa.
\newblock {FedDuA}: Doubly adaptive federated learning.
\newblock \emph{arXiv preprint arXiv:2505.11126}, 2025.

\bibitem[Wang et~al.(2023)Wang, Zhang, Ma, and Chen]{wang2023adagrad}
B.~Wang, H.~Zhang, Z.~Ma, and W.~Chen.
\newblock Convergence of {AdaGrad} for non-convex objectives: Simple proofs and
  relaxed assumptions.
\newblock In \emph{Proceedings of Thirty Sixth Conference on Learning Theory},
  2023.

\bibitem[Wang et~al.(2025)Wang, Chen, Pappas, Yang, Quek, and
  Poor]{wang2025adaptive}
C.~Wang, Z.~Chen, N.~Pappas, H.~H. Yang, T.~Q. Quek, and H.~V. Poor.
\newblock Adaptive federated learning over the air.
\newblock \emph{IEEE Transactions on Signal Processing}, 73:\penalty0
  3187--3202, 2025.

\bibitem[Wang et~al.(2021)Wang, Xu, Garrett, Charles, Liu, and
  Joshi]{wang2021local}
J.~Wang, Z.~Xu, Z.~Garrett, Z.~Charles, L.~Liu, and G.~Joshi.
\newblock Local adaptivity in federated learning: Convergence and consistency.
\newblock \emph{arXiv preprint arXiv:2106.02305}, 2021.

\bibitem[Wang et~al.(2022)Wang, Lin, and Chen]{wang2022communication}
Y.~Wang, L.~Lin, and J.~Chen.
\newblock Communication-efficient adaptive federated learning.
\newblock In \emph{Proceedings of the 39th International Conference on Machine
  Learning}, 2022.

\bibitem[Wang et~al.(2024)Wang, Wang, Lu, and Chen]{wang2024fadas}
Y.~Wang, S.~Wang, S.~Lu, and J.~Chen.
\newblock {FADAS}: Towards federated adaptive asynchronous optimization.
\newblock In \emph{Proceedings of the 41st International Conference on Machine
  Learning}, 2024.

\bibitem[Wang and Yurtsever(2026)]{wang2026universal}
Z.~Wang and A.~Yurtsever.
\newblock Universal adaptive proximal gradient methods via gradient mapping
  accumulation.
\newblock \emph{arXiv preprint arXiv:2605.05944}, 2026.

\bibitem[Ward et~al.(2020)Ward, Wu, and Bottou]{ward2020adagrad}
R.~Ward, X.~Wu, and L.~Bottou.
\newblock {AdaGrad} stepsizes: Sharp convergence over nonconvex landscapes.
\newblock \emph{Journal of Machine Learning Research}, 21\penalty0
  (219):\penalty0 1--30, 2020.

\bibitem[Wu et~al.(2023)Wu, Huang, Hu, and Huang]{wu2023faster}
X.~Wu, F.~Huang, Z.~Hu, and H.~Huang.
\newblock Faster adaptive federated learning.
\newblock In \emph{Proceedings of the {AAAI} Conference on Artificial
  Intelligence}, 2023.

\bibitem[Xie et~al.(2020)Xie, Koyejo, Gupta, and Lin]{xie2020local}
C.~Xie, O.~Koyejo, I.~Gupta, and H.~Lin.
\newblock {Local AdaAlter}: Communication-efficient stochastic gradient descent
  with adaptive learning rates.
\newblock In \emph{OPT 2020: 12th Annual Workshop on Optimization for Machine
  Learning}, 2020.

\bibitem[Yuan and Wang(2025)]{yuan2025aggregated}
W.~Yuan and X.~Wang.
\newblock Aggregated gradients-based adaptive learning rate design in federated
  learning.
\newblock In \emph{Proceedings of the 34th ACM International Conference on
  Information and Knowledge Management}, 2025.

\bibitem[Zhang et~al.(2023)Zhang, Zeng, and Lin]{zhang2023fedur}
H.~Zhang, K.~Zeng, and S.~Lin.
\newblock {FedUR}: Federated learning optimization through adaptive centralized
  learning optimizers.
\newblock \emph{IEEE Transactions on Signal Processing}, 71:\penalty0
  2622--2637, 2023.

\end{thebibliography}

\newpage
\appendix
\section{Preliminary}
\label{app:preliminary}
\subsection{Auxiliary Lemmas}
\begin{lemma}\label{lem:log-solution}
Let $a, x>0$ and $b\in\rbb$. If $x \leq a \log x+b$, then $x \leq 2 a \log 2a - 2a+2 b.$
\end{lemma}
\begin{proof}
Applying the basic inequality $\log u \leq u-1$ with $u=x/(2a)$ gives 
\[
\log x \leq \frac{x}{2a}+\log 2a -1,
\]
which together with $x \leq a \log x+b$ implies that
\[
x \leq \frac{x}{2}+a\log 2a -a +b.
\]
We obtain the desired result after rearranging the terms.
\end{proof}

\begin{lemma}
\label{lem:sublinear-power-bound}
Fix $\alpha\in(0,1)$ and let $a,b,x\geq0$. If $x\leq ax^\alpha+b$, then $x\leq (2a)^{1/(1-\alpha)}+2b$.
\end{lemma}
\begin{proof}
If $x\leq2b$, the result is immediate. Otherwise if $x>2b$, then the given inequality implies
\[
    \frac{x}{2}<x-b\leq ax^\alpha.
\]
Therefore, $x^{1-\alpha}\leq2a$ and hence $x\leq(2a)^{1/(1-\alpha)}$, which completes the proof.
\end{proof}

\begin{lemma}\label{lem:positive-power-log-bound}
For any $c,\eta>0$, $0<a<b$, and $x\geq0$, we have
\[
x^a-cx^b \leq\left(\eta^a+c^{-a/(b-a)}\right)\log\left(1+\frac{x^a}{\eta^a}\right).
\]
\end{lemma}
\begin{proof}
If $x^a-cx^b\leq0$, then the result is immediate.
Otherwise we have $1-cx^{b-a}>0$, which implies $x<c^{-1/(b-a)}$ and hence $x^a<c^{-a/(b-a)}$.
Using $\log(1+z)\geq z/(1+z)$ for $z\geq0$, we obtain
\[
\log\left(1+\frac{x^a}{\eta^a}\right)\geq\frac{x^a}{\eta^a+x^a}\geq\frac{x^a}{\eta^a+c^{-a/(b-a)}}\ge \frac{x^a-cx^b}{\eta^a+c^{-a/(b-a)}},
\]
which completes the proof.
\end{proof}

\begin{lemma}[\citep{wang2026universal}]\label{lem:summation}
Let ${\{a_k\}}_{k\ge 1}$ be a nonnegative sequence and define $A_{k+1} = A_1+\sum_{i=1}^k a_i$ with $A_1>0$. Then for all $t\ge 1$:
\[
2(\sqrt{A_{t+1}} - \sqrt{A_1}) \leq \sum_{k=1}^t \frac{a_k}{\sqrt{A_{k}}}.
\]
Moreover, if there exists $c > 0$ such that $a_k \leq c A_{k}$ for all $k \geq 1$, then
\begin{align*}
&\sum_{k=1}^t \frac{a_k}{\sqrt{A_{k}}} \leq (1+\sqrt{1+c})(\sqrt{A_{t+1}} - \sqrt{A_1}),\\
&\sum_{k=1}^t \frac{a_k}{A_k} \leq (1+c)\log\frac{A_{t+1}}{A_1}.
\end{align*}
\end{lemma}

\begin{lemma}
\label{lem:log-power-concavity}
For every $a\in\mathbb R$ and $p\ge 1$, the function $x\mapsto(a+\log x)^p$ is concave on $[\exp(p-1-a),\infty)$.
\end{lemma}
\begin{proof}
If $p=1$, then $a+\log x$ is concave on $(0,\infty)$. Next we assume $p>1$. For $x\geq\exp(p-1-a)$, we have $a+\log x\geq p-1>0$ and
\[
    \frac{d^2}{dx^2}(a+\log x)^p=\frac{p(a+\log x)^{p-2}(p-1-a-\log x)}{x^2}\leq0,
\]
which completes the proof.
\end{proof}

\begin{lemma}[Conditional $L^p$ triangle inequality]
\label{lem:conditional-l2-triangle}
Let $p\geq1$, let $X$ and $Y$ be random vectors such that $\E[\|X\|^p]<\infty$ and $\E[\|Y\|^p]<\infty$, and let $\mathcal G$ be a sub-$\sigma$-algebra. Then, almost surely,
\[
    \left(\E[\|X+Y\|^p\mid\mathcal G]\right)^{1/p}\leq\left(\E[\|X\|^p\mid\mathcal G]\right)^{1/p}+\left(\E[\|Y\|^p\mid\mathcal G]\right)^{1/p}.
\]
\end{lemma}
\begin{proof}
For $p=1$, the claim follows directly from $\|X+Y\|\leq\|X\|+\|Y\|$. Suppose $p>1$ and let $q=p/(p-1)$, so that $1/p+1/q=1$. The pointwise triangle inequality and conditional H\"older inequality give
\begin{align*}
    &\E[\|X+Y\|^p\mid\mathcal G]\leq\E[\|X\|\|X+Y\|^{p-1}\mid\mathcal G]+\E[\|Y\|\|X+Y\|^{p-1}\mid\mathcal G]\\
    \leq&\left(\E[\|X\|^p\mid\mathcal G]\right)^{1/p}\left(\E[\|X+Y\|^{(p-1)q}\mid\mathcal G]\right)^{1/q}
    +\left(\E[\|Y\|^p\mid\mathcal G]\right)^{1/p}\left(\E[\|X+Y\|^{(p-1)q}\mid\mathcal G]\right)^{1/q}\\
    =&\left(\left(\E[\|X\|^p\mid\mathcal G]\right)^{1/p}+\left(\E[\|Y\|^p\mid\mathcal G]\right)^{1/p}\right)\left(\E[\|X+Y\|^p\mid\mathcal G]\right)^{(p-1)/p},
\end{align*} 
which proves the result after simplification.
\end{proof}

\begin{lemma}[Doob's $L^2$ maximal inequality~{\citep[Theorem~4.4.4]{durrett2019probability}}]\label{lem:doob-l2}
Let $\{\mathcal M_t\}_{t=0}^T$ be a sequence of square-integrable martingale. Then
\[
    \E\left[\max_{0\leq t\leq T}|\mathcal M_t|^2\right]\leq4\E[|\mathcal M_T|^2].
\]
\end{lemma}

\begin{lemma}\label{lem:pp-predictable}
Let $Z_t\geq0$ be integrable and $\mathcal F_{t+1}$-measurable for all $t\geq1$. For any constant $c>0$ and $T\geq1$, we have
\[
    \E\left[\sqrt{c+\sum_{t=1}^T Z_t}\right]\leq\sqrt2\,\E\left[\sqrt{c+\sum_{t=1}^T \E_t[Z_t]}\right].
\]
\end{lemma}
\begin{proof}
Let $a_t:=\E_t[Z_t]$, $X_T:=c+\sum_{t=1}^T Z_t$, and $A_t:=c+\sum_{j=1}^t a_j$. Since $A_T\geq A_t\geq c$, we have
\[
\begin{aligned}
    \E\left[\frac{X_T}{\sqrt{A_T}}\right]&\leq\sqrt c+\sum_{t=1}^T\E\left[\frac{Z_t}{\sqrt{A_t}}\right]=\sqrt c+\sum_{t=1}^T\E\left[\frac{a_t}{\sqrt{A_t}}\right]\\
    &\leq\sqrt c+2\E\left[\sum_{t=1}^T(\sqrt{A_t}-\sqrt{A_{t-1}})\right]=\sqrt c+2\E[\sqrt{A_T}]-2\sqrt{A_0}\leq2\E[\sqrt{A_T}],
\end{aligned}
\]
where the second inequality follows from
\[
\frac{a_t}{\sqrt{A_t}}=\frac{A_t-A_{t-1}}{\sqrt{A_t}}\leq\frac{2(A_t-A_{t-1})}{\sqrt{A_t}+\sqrt{A_{t-1}}}=2(\sqrt{A_t}-\sqrt{A_{t-1}}).
\]
Therefore, Cauchy--Schwarz further yields
\[
    \E[\sqrt{X_T}]=\E\left[\sqrt{\frac{X_T}{\sqrt{A_T}}}\,A_T^{1/4}\right]\leq\sqrt{\E\left[\frac{X_T}{\sqrt{A_T}}\right]}\sqrt{\E[\sqrt{A_T}]}\leq\sqrt2\,\E[\sqrt{A_T}],
\]
which completes the proof.
\end{proof}

\subsection{Proof of Proposition~\ref{prop:fs-variance-obstruction}}\label{app:proof-couterexample}

\begin{proof}
Note that $x^\star=0$ and $F^\star=F(0)=0$. Moreover, direct calculation gives
\[
    F'(x)=\frac{x}{\sqrt{1+x^2}},\qquad F''(x)=\frac{1}{(1+x^2)^{3/2}},\qquad |F'(x)|\leq1,\qquad 0<F''(x)\leq1.
\]
Thus every $F_i=F$ is convex, $1$-smooth, and $1$-Lipschitz. Since all client objectives are identical, their heterogeneity is zero.

Fix $T$ divisible by 4, let $M:=16mK\sqrt T$ and all local noises $\{\xi_{i,k}^t\}$ be independent and have discrete distribution
\[
    \mathbb P(\xi=M)=\mathbb P(\xi=-M)=\frac1{2M^2},\qquad \mathbb P(\xi=0)=1-\frac1{M^2},
\]
which satisfies the unbiasedness and bounded variance conditions:
\[
    \E[\xi]=M\frac1{2M^2}-M\frac1{2M^2}=0,
    \qquad
    \E[\xi^2]=M^2\left(\frac1{2M^2}+\frac1{2M^2}\right)=1.
\]

Let $\mathcal E_T$ be the event: exactly one of the $mKT$ noise realizations is nonzero and this realization occurs at the last local step of some client in a round $j\in\{T/2+1,\ldots,3T/4\}$. 
Then there are $mT/4$ possible locations for the unique nonzero noise, which implies
\begin{equation}\label{eq:prob-ecal_T}
    \mathbb P(\mathcal E_T)=\frac{mT}{4M^2}\left(1-\frac1{M^2}\right)^{mKT-1}\geq\frac{mT}{4M^2}\left(1-\frac{mKT-1}{M^2}\right)\geq\frac{mT}{8M^2}=\frac1{2048mK^2}.
\end{equation}
Here we use $(1+x)^p \geq 1+px$ for $x \ge -1$ and $p \ge 1$, and
\[
    \frac{mKT-1}{M^2}\leq\frac{mKT}{256m^2K^2T}=\frac1{256mK}\leq\frac12.
\]
Now let us consider the event $\mathcal E_T$. Since $F'(x_1)=0$, all local gradients vanish except for the unique nonzero noise. Thus induction over the preceding local and server updates gives $x_t=x_{i,k}^t=0$ and $q_t=1$ until the exceptional noise occurs. In particular, it follows that $x_j=0$ and $q_j=1$. If the nonzero noise occurs at client $i$ and has sign $s\in\{-1,1\}$, then its occurrence at local step $K-1$ gives
\[
    \bar g_j^{(i)}=\frac{sM}{K},
    \qquad
    \bar g_j^{(i')}=0\quad \forall i'\neq i.
\]
Consequently, the server update~\eqref{eq:nocv-outer-method} yields
\[
    \bar g_j=\frac{sM}{mK}=16s\sqrt T,\qquad x_{j+1}=-16s\sqrt T,\qquad q_{j+1}=\sqrt{q_j^2+|\bar g_j|^2}=\sqrt{1+256T}.
\]
All later noises also vanish on $\mathcal E_T$. Hence, for every $t\geq j+1$,
\[
    |\bar g_t|\leq\frac1{mK}\sum_{i=1}^m\sum_{k=0}^{K-1}|F'(x_{i,k}^t)|\leq1,
    \qquad
    q_t\geq q_{j+1}\geq16\sqrt T.
\]
Unrolling the server update now gives, for $j+1\leq t\leq T+1$,
\[
    |x_t-x_{j+1}|\leq\sum_{\ell=j+1}^{t-1}\frac{|\bar g_\ell|}{q_\ell}\leq\frac{t-j-1}{16\sqrt T}\leq\frac{\sqrt T}{16}.
\]
It follows that
\begin{equation}\label{eq:xt-after-j}
    |x_t|\geq| x_{j+1} | - |x_t - x_{j+1}| \geq 16\sqrt T - \frac{\sqrt T}{16} = \frac{255}{16}\sqrt T \geq 8\sqrt T,
\end{equation}
which also implies that $x_t$ has the same sign as $x_{j+1}$. Since $j\leq3T/4$ and $|F'(x)|^2=x^2/(1+x^2)\geq1/2$ for $|x|\geq1$, we further have
\[
    \frac1T\sum_{t=1}^T|F'(x_t)|^2\ge\frac1T\sum_{t=j+1}^T|F'(x_t)|^2
    \geq\frac{T-j}{T}\frac12
    \geq\frac18,
\]
Moreover, since $x_t=0$ for $t\leq j$ and the remaining ones have the same sign, it follows from $j\leq3T/4$, $F^\star=F(0)=0$, and~\eqref{eq:xt-after-j} that
\begin{align*}
    \sum_{t=1}^T F(x_t) -F^\star=&\sum_{t=j+1}^T F(x_t)=\sum_{t=j+1}^T(\sqrt{1+x_t^2}-1)\ge \sum_{t=j+1}^T(|x_t|-1)\\
    \ge& \sum_{t=j+1}^T(8\sqrt{T}-1)\ge (T-j)4\sqrt{T}\ge T^{3/2}.
\end{align*}
Therefore, we conclude from~\eqref{eq:prob-ecal_T} that
\begin{align*}
    &\E[|F'(\bar x_T)|^2]=\frac1T\sum_{t=1}^T\E[|F'(x_t)|^2]\geq\frac18\mathbb P(\mathcal E_T)\geq\frac1{16384mK^2},\\
    &\E[F(\bar x_T)-F^\star]=\frac1T\sum_{t=1}^T\E[F(x_t)-F^\star]\geq\sqrt T\,\mathbb P(\mathcal E_T)\geq\frac{\sqrt T}{2048mK^2},
\end{align*}
which completes the proof.
\end{proof}

\subsection{Additional notations for analysis}
Write $g_t:=\nabla F(x_t)$ and define
\begin{align*}
    \hat b_t&:=\frac1{mK}\sum_{i=1}^m\sum_{k=0}^{K-1}\bigl(\nabla F_i(x_{i,k}^t)-\nabla F_i(x_t)\bigr),\\
    \bar\xi_t&:=\frac1{mK}\sum_{i=1}^m\sum_{k=0}^{K-1}\xi_{i,k}^t,\qquad b_t:=\bar g_t-g_t=\hat b_t+\bar\xi_t.
\end{align*}
Fix any $\tau>0$, we choose the local-path scale according to
\begin{align}
    K=1:~&\tau_t=\tau,\notag\\
    K\ge 2:~&\tau_t\leq
        \begin{cases}
            \displaystyle\frac{\tau}{(K-1)t}, & \text{smooth nonconvex or Lipschitz convex},\\[0.8em]
            \displaystyle\frac{\tau}{(K-1)t^p},~\forall p>\frac12, &\text{smooth convex}.\label{eq:tau_t-scheme}
        \end{cases}
\end{align}
In particular, we can always choose $\tau_t\leq \frac{1}{(K-1)t}$ for $K\ge 2$, which is exactly~\eqref{eq:partial-tau}.
Under~\eqref{eq:tau_t-scheme}, we have that $\{\tau_t\}_{t\ge 1}$ is upper bounded by $\tau$.
Let
\begin{equation*}
    c_K:=\frac1K\sum_{k=0}^{K-1}k^2=\frac{(K-1)(2K-1)}{6},\quad \Theta:=\left(1+\frac{L\tau}{q_1}\right)^{K-1}
\end{equation*}
and
\[
    B^2:=\frac{c_K\Theta^2 L^2}{m}\sum_{i=1}^m\|\nabla F_i(x_1)\|^2, \qquad \widetilde B^2:=\frac{c_K\Theta^2L^2}{m}\sum_{i=1}^m\bigl(\|\nabla F_i(x_1)\|+\sigma\bigr)^2.
\]
For $t\geq1$, denote the finite sums of $\tau_k$ and $\tau_k^2$ by
\[
    \mathcal T_t:=\sum_{k=1}^t\tau_k,\qquad \widetilde{\mathcal T}_t:=\sum_{k=1}^t\tau_k^2.
\]
If $K\ge 2$ and $\tau_t$ is chosen according to~\eqref{eq:tau_t-scheme}, then
\begin{equation}\label{eq:quantities-estimates}
c_K\lesssim (K-1)^2,\quad \Theta\lesssim1,\quad B\lesssim K-1, \quad \widetilde B\lesssim K-1,\quad \widetilde{\mathcal T}_t\lesssim \frac{1}{(K-1)^2},
\end{equation}
where we use $(1+1/t)^t\leq e$ to obtain $\Theta\lesssim1$. Moreover, for Lipschitz convex objectives, we have $\mathcal T_t\lesssim \log t/(K-1)$.

Let $\zeta_t:=\bar g_t-\E_t[\bar g_t]$. Set 
\[
\mathcal A_T:=\eta^2+2R_0^2+2R_1^2\sigma^2(\log(mKT)+2).
\] 
Given $T\ge 1$, let $M_T\ge 0$ be the maximum noise norm over the first $T$ rounds:
\[
    M_T:=\max_{1\leq t\leq T,\,1\leq i\leq m,\,0\leq k<K}\|\xi_{i,k}^t\|.
\]

Let $r_t:=x_{t+1}-x_t$. Define the accumulated trajectory quantities:
\[
    \mathcal R_T:=\sum_{t=1}^T\|r_t\|^2,\qquad \mathcal G_T:=\sum_{t=1}^T\frac{\|g_t\|^2}{q_t}.
\]

We focus on the nontrivial case $K\ge2$ in the proofs. The special case $K=1$ can be easily proved by noting that the coefficient of $\tau_t$ always involve $(K-1)$ (possibly through quantities $c_K$, $B$, and $\widetilde B$) in the analysis, which means that $\tau_t$ can be chosen arbitrarily for $K=1$. 

\section{Smooth nonconvex objectives}\label{app:smooth-nonconvex}
Assume throughout this section that every $F_i$ is $L$-smooth. 
From unbiasedness of $\xi_{i,k}^t$, we have $\E_t[\bar\xi_t]=0$ and hence
\begin{equation}
    \E_t[\bar g_t]=g_t+\E_t[\hat b_t].
\label{eq:sg-averaged-noise-moments}
\end{equation}
Additionally, the independence among the local noises $\{\xi_{i,k}^t\}$ gives
\begin{equation}
    \label{eq:noise-ave-bd}
    \E_t[\|\bar\xi_t\|^2]=\frac1{m^2K^2}\sum_{i=1}^m\sum_{k=0}^{K-1}\E_t[\|\xi_{i,k}^t\|^2]\leq\frac{mK\sigma^2}{m^2K^2}=\frac{\sigma^2}{mK}.
\end{equation}
Finally, since $\bar g_t=g_t+\hat b_t+\bar\xi_t$, we apply $\|a+b\|^2\leq2\|a\|^2+2\|b\|^2$ to get
\[
    \E_t[\|\bar g_t\|^2]\leq2\E_t[\|g_t+\bar\xi_t\|^2]+2\E_t[\|\hat b_t\|^2]=2\|g_t\|^2+2\E_t[\|\bar\xi_t\|^2]+2\E_t[\|\hat b_t\|^2],
\]
where the equality follows from $\E_t[\ip{g_t}{\bar\xi_t}]=\ip{g_t}{\E_t[\bar\xi_t]}=0$. This together with \eqref{eq:noise-ave-bd} gives
\begin{equation}
    \label{eq:gt_sq-exp}
    \E_t[\|\bar g_t\|^2]\leq2\|g_t\|^2+2\E_t[\|\hat b_t\|^2]+\frac{2\sigma^2}{mK}
\end{equation}

Before presenting the main theorem and its proof, we first introduce some auxiliary lemmas that provide bounds for trajectory-related quantities.
\begin{lemma}\label{lem:sg-local-moments}
Let Assumptions~\ref{ass:lower-bounded} and~\ref{ass:stochastic-noise} hold. Assume each $F_i$ is $L$-smooth. Let $\kappa_T:=\max_{t\in[T]} (t-1)\tau_t$.
Then for every $T\geq1$,
\begin{equation*}
    \sum_{t=1}^T\E_t[\|\hat b_t\|^2]\leq\frac{2\widetilde B^2\widetilde{\mathcal T}_T}{q_1^2}+\frac{8c_K\Theta^2L^4\kappa_T^2}{q_1^2}\mathcal R_T.
\end{equation*}
\end{lemma}

\begin{proof}
We start by using Lemma~\ref{lem:conditional-l2-triangle} with $\mathcal G=\mathcal F_t$ and $L$-smoothness to obtain
\begin{equation}
\begin{aligned}
    \sqrt{\E_t[\|\nabla F_i(x_{i,k}^t)\|^2]}
    &\leq\|\nabla F_i(x_t)\|+\sqrt{\E_t[\|\nabla F_i(x_{i,k}^t)-\nabla F_i(x_t)\|^2]}\\
    &\leq\|\nabla F_i(x_t)\|+L\sqrt{\E_t[\|x_{i,k}^t-x_t\|^2]}.
\end{aligned}
\label{eq:sg-local-gradient-l2}
\end{equation}
Applying Lemma~\ref{lem:conditional-l2-triangle},~\eqref{eq:sg-local-gradient-l2}, and~\eqref{eq:sg-implied-variance} gives
\[
\begin{aligned}
    \sqrt{\E_t[\|x_{i,k+1}^t-x_t\|^2]}&=\sqrt{\E_t\!\left[\|x_{i,k}^t-x_t-\tau_t(\nabla F_i(x_{i,k}^t)+\xi_{i,k}^t)/q_t\|^2\right]}\\
    &\leq \sqrt{\E_t[\|x_{i,k}^t-x_t\|^2]}+\frac{\tau_t}{q_t}\sqrt{\E_t[\|\nabla F_i(x_{i,k}^t)+\xi_{i,k}^t\|^2]}\\
    &\leq \sqrt{\E_t[\|x_{i,k}^t-x_t\|^2]}+\frac{\tau_t}{q_t}\left(\sqrt{\E_t[\|\nabla F_i(x_{i,k}^t)\|^2]}+\sqrt{\E_t[\|\xi_{i,k}^t\|^2]}\right)\\
    &\leq\left(1+\frac{L\tau_t}{q_t}\right)\sqrt{\E_t[\|x_{i,k}^t-x_t\|^2]}+\frac{\tau_t}{q_t}\bigl(\|\nabla F_i(x_t)\|+\sigma\bigr).
\end{aligned}
\]
Since $x_{i,0}^t=x_t$, applying the above inequality recursively and then using $\tau_t\leq\tau$ and $q_t\geq q_1$ yields
\begin{align}
    &\sqrt{\E_t[\|x_{i,k}^t-x_t\|^2]}\leq\frac{\tau_t}{q_t}\bigl(\|\nabla F_i(x_t)\|+\sigma\bigr)\sum_{s=0}^{k-1}\left(1+\frac{L\tau_t}{q_t}\right)^s\notag\\
    \leq&\frac{\tau_t}{q_t}\bigl(\|\nabla F_i(x_t)\|+\sigma\bigr)k\left(1+\frac{L\tau}{q_1}\right)^{k-1}\leq\frac{k\tau_t\Theta}{q_t}\bigl(\|\nabla F_i(x_t)\|+\sigma\bigr).\label{eq:sqrt-xk-xt}
\end{align}
Then it follows from the convexity of $\|\cdot\|^2$, the $L$-smoothness of $F_i$, and~\eqref{eq:sqrt-xk-xt} that
\begin{align*}
    \E_t[\|\hat b_t\|^2]\leq&\frac{1}{mK}\sum_{i=1}^m\sum_{k=0}^{K-1}\E_t[\|\nabla F_i(x_{i,k}^t)-\nabla F_i(x_t)\|^2]\leq\frac{L^2}{mK}\sum_{i=1}^m\sum_{k=0}^{K-1}\E_t[\|x_{i,k}^t-x_t\|^2]\\
    \leq&\frac{\Theta^2L^2\tau_t^2}{mq_t^2K}\sum_{i=1}^m\sum_{k=0}^{K-1}k^2\bigl(\|\nabla F_i(x_t)\|+\sigma\bigr)^2=\frac{c_K\Theta^2L^2\tau_t^2}{mq_t^2}\sum_{i=1}^m\bigl(\|\nabla F_i(x_t)\|+\sigma\bigr)^2.
\end{align*}
To proceed, applying $\|\nabla F_i(x_t)\|\leq\|\nabla F_i(x_1)\|+L\|x_t-x_1\|$ into the above inequality and using $(a+b)^2\leq2a^2+2b^2$ yield
\begin{equation}
\begin{aligned}
    \E_t[\|\hat b_t\|^2]&\leq\frac{c_K\Theta^2L^2\tau_t^2}{mq_t^2}\sum_{i=1}^m\left(\|\nabla F_i(x_1)\|+\sigma+L\|x_t-x_1\|\right)^2\\
    &\leq\frac{2\tau_t^2}{q_t^2}\left(\widetilde B^2+c_K\Theta^2L^4\|x_t-x_1\|^2\right).
\end{aligned}
\label{eq:sg-drift-around-initial-point}
\end{equation}
For $2\leq t\leq T$, note that $x_t-x_1=\sum_{j=1}^{t-1}r_j$, hence we have from $\tau_t\leq\kappa_T/(t-1)$ and $q_t\geq q_j$ for $j\leq t$ that
\begin{equation}
\begin{aligned}
    \sum_{t=1}^T\frac{\tau_t^2}{q_t^2}\|x_t-x_1\|^2&=\sum_{t=2}^T\frac{\tau_t^2}{q_t^2}\left\|\sum_{j=1}^{t-1}r_j\right\|^2\leq\kappa_T^2\sum_{t=2}^T\left(\frac1{t-1}\sum_{j=1}^{t-1}\frac{\|r_j\|}{q_j}\right)^2\\
    &\leq4\kappa_T^2\sum_{j=1}^{T-1}\frac{\|r_j\|^2}{q_j^2}\leq\frac{4\kappa_T^2}{q_1^2}\sum_{t=1}^T\|r_t\|^2,
\end{aligned}
\label{eq:sg-weighted-hardy}
\end{equation}
where the second inequality holds from Hardy's inequality $\sum_{t=1}^{T}\bar a_t^2\leq 4\sum_{t=1}^{T}a_t^2$ with $\bar a_t$ being the average of nonnegative $a_1, a_2, \ldots, a_t$.
Summing~\eqref{eq:sg-drift-around-initial-point} over $t=1,2,\ldots,T$ and applying~\eqref{eq:sg-weighted-hardy} then gives
\[
    \sum_{t=1}^T\E_t[\|\hat b_t\|^2]\leq\frac{2\widetilde B^2}{q_1^2}\sum_{t=1}^T\tau_t^2+2c_K\Theta^2L^4\sum_{t=1}^T\frac{\tau_t^2}{q_t^2}\|x_t-x_1\|^2\leq\frac{2\widetilde B^2\widetilde{\mathcal T}_T}{q_1^2}+\frac{8c_K\Theta^2L^4\kappa_T^2}{q_1^2}\sum_{t=1}^T\|r_t\|^2,
\]
which completes the proof.
\end{proof}

Before introducing the next lemma, we first bound moments $\E[M_T^{2p}]$ for any $p\geq1$. Suppose $\sigma>0$, let $N:=mKT$. Reindex the noises $\{\xi_{i,k}^t\}$ as $\xi_1,\ldots,\xi_N$, and denote
\[
    Z_j:=\frac{\|\xi_j\|^2}{\sigma^2}\geq0,\qquad A_N:=\frac1N\sum_{j=1}^N\exp(Z_j)\geq1.
\]
Then it follows from Assumption~\ref{ass:stochastic-noise} that $\E[\exp(Z_j)]\leq\exp(1)$ for every $j$, and hence $\E[A_N]\leq\exp(1)$. Moreover,
\begin{equation}\label{eq:sg-noise-max-accumulator}
    \frac{M_T^2}{\sigma^2}=\max_{1\leq j\leq N}Z_j\leq\log\left(\sum_{j=1}^N\exp(Z_j)\right)=\log N+\log A_N.
\end{equation}
For $p\geq1$, Lemma~\ref{lem:log-power-concavity} with $a=p-1+\log N$ shows that the function $x\mapsto(p-1+\log N+\log x)^p$ is concave on $[1/N,\infty)$, hence on $[1,\infty)$.
Applying Jensen's inequality on~\eqref{eq:sg-noise-max-accumulator} therefore yields
\[
    \frac{\E[M_T^{2p}]}{\sigma^{2p}}\leq\E\left[(p-1+\log N+\log A_N)^p\right]\leq(p-1+\log N+\log\E[A_N])^p\leq(\log(mKT)+p)^p,
\]
which is equivalent to
\begin{equation}
    \left(\E[M_T^{2p}]\right)^{1/p}\leq\sigma^2\bigl(\log(mKT)+p\bigr),\quad \forall p\geq 1.
    \label{eq:sg-noise-envelope}
\end{equation}
Note that~\eqref{eq:sg-noise-envelope} is a consequence of Assumption~\ref{ass:stochastic-noise}, which holds regardless of the problem settings.

\begin{lemma}\label{lem:sg-energy}
Let Assumption~\ref{ass:stochastic-noise} hold. Assume each $F_i$ is $L$-smooth.
If $\rho_\tau:=\sup_{t\ge 1}\sqrt{t-1}\,\tau_t<\infty$, define deterministic constants
\begin{align*}
    R_0&:=\max\left\{\frac{\|g_1\|}{q_1},\eta\left(1+\frac L{q_1}\right)\right\}+\frac{2B\tau+2\eta\sqrt{c_K}\,\Theta L^2\rho_\tau}{q_1^2},\\
    R_1&:=\frac2{q_1}\left(1+\frac{\sqrt{c_K}\,\Theta L\tau}{q_1}\right).
\end{align*}
Then for every $T\geq1$, we have $\max_{1\leq t\leq T}\|r_t\|\leq R_0+R_1M_T$ almost surely and
\begin{align}
    \E[\mathcal R_T]\leq&2\mathcal A_T\left(1+\log\frac{\E[q_{T+1}]}{q_1}\right),\label{eq:sg-expected-energy}\\
    \sqrt{\E[\mathcal R_T^2]}\leq&2\bigl(\mathcal A_T+4R_1^2\sigma^2\bigr)\left(4+\log\frac{\E[q_{T+1}]}{q_1}\right).
\label{eq:sg-nonconvex-energy-second-moment}
\end{align}
\end{lemma}
\begin{proof}
We first establish a deterministic bound on $\|b_t\|$. Using smoothness and definition of $M_T$, we obtain
\begin{align*}
    \|x_{i,k+1}^t-x_t\|&\leq\|x_{i,k}^t-x_t\|+\frac{\tau_t}{q_t}\bigl(\|\nabla F_i(x_{i,k}^t)\|+M_T\bigr)\\
    &\leq\|x_{i,k}^t-x_t\|+\frac{\tau_t}{q_t}\bigl(\|\nabla F_i(x_t)\|+L\|x_{i,k}^t-x_t\|+M_T\bigr)\\
    &=\left(1+\frac{L\tau_t}{q_t}\right)\|x_{i,k}^t-x_t\|+\frac{\tau_t}{q_t}\bigl(\|\nabla F_i(x_t)\|+M_T\bigr).
\end{align*}
Unrolling this recursion similarly as in the proof of Lemma~\ref{lem:sg-local-moments} gives (with $\sigma$ replaced by $M_T$), for $t\leq T$,
\begin{align*}
    &\|x_{i,k}^t-x_t\|\leq\frac{\tau_t}{q_t}\bigl(\|\nabla F_i(x_t)\|+M_T\bigr)\sum_{s=0}^{k-1}\left(1+\frac{L\tau_t}{q_t}\right)^s\leq \frac{k\tau_t\Theta}{q_t}\bigl(\|\nabla F_i(x_t)\|+M_T\bigr),\\
    &\|\hat b_t\|^2 \leq\frac{c_K\Theta^2L^2\tau_t^2}{mq_t^2}\sum_{i=1}^m\bigl(\|\nabla F_i(x_t)\|+M_T\bigr)^2.
\end{align*}
Then by applying Lemma~\ref{lem:conditional-l2-triangle} and the smoothness of $F_i$, we further have
\begin{align}
    \|\hat b_t\|&\leq\frac{\tau_tL\Theta\sqrt{c_K}}{q_t}\Big[\frac1m\sum_{i=1}^m\bigl(\|\nabla F_i(x_t)\|+M_T\bigr)^2\Big]^{1/2}\notag\\
    &\leq\frac{\tau_tL\Theta\sqrt{c_K}}{q_t}\Big[\frac1m\sum_{i=1}^m\bigl(\|\nabla F_i(x_1)\|+L\|x_t-x_1\|+M_T\bigr)^2\Big]^{1/2}\notag\\
    &\leq\frac{\tau_tL\Theta\sqrt{c_K}}{q_t}\Big[\Big(\frac1m\sum_{i=1}^m\|\nabla F_i(x_1)\|^2\Big)^{1/2}+L\|x_t-x_1\|+M_T\Big]\notag\\
    &=\frac{\tau_t}{q_t}\left(B+\sqrt{c_K}\,\Theta LM_T+\sqrt{c_K}\,\Theta L^2\|x_t-x_1\|\right).\label{eq:hat-b-norm1}
\end{align}
To bound $\|x_t-x_1\|$, we use the update rule \eqref{eq:nocv-outer-method} and $q_t\geq q_1$ to write
\[
    \|x_t-x_1\|^2=\Big\|\sum_{j=1}^{t-1}r_j\Big\|^2\leq(t-1)\sum_{j=1}^{t-1}\|r_j\|^2\leq\frac{t-1}{q_1^2}\sum_{j=1}^{t-1}\|\bar g_j\|^2=\frac{(t-1)\eta^2}{q_1^2}(q_t^2-q_1^2),
\]
which further implies
\[
    \|x_t-x_1\|\leq\frac{\eta\sqrt{t-1}}{q_1}\sqrt{q_t^2-q_1^2}\leq\frac{\eta\sqrt{t-1}}{q_1}q_t.
\]
Substituting the above inequality into \eqref{eq:hat-b-norm1} and noting that $b_t=\hat b_t+\bar\xi_t$ and $\|\bar \xi_t\|\leq M_T$ gives
\begin{equation}
    \|b_t\|\leq\|\hat b_t\|+\|\bar\xi_t\|\leq\frac{B\tau_t+\eta\sqrt{c_K}\,\Theta L^2\sqrt{t-1}\,\tau_t}{q_1}+\left(1+\frac{\sqrt{c_K}\,\Theta L\tau_t}{q_1}\right)M_T.
    \label{eq:sg-pathwise-bias}
\end{equation}

Next, we show that $\|r_t\|\leq R_0+R_1M_T$ for all $1\leq t\leq T$.
For $t=1$, using $\bar g_1=g_1+b_1$,~\eqref{eq:sg-pathwise-bias}, and $\tau_t\leq \tau$ gives 
\[
\|r_1\|\leq \frac{\|g_1\|+\|b_1\|}{q_1}\leq \frac{\|g_1\|}{q_1}+\frac{B\tau}{q_1^2}+\frac{1}{q_1}\left(1+\frac{\sqrt{c_K}\,\Theta L\tau}{q_1}\right)M_T\leq R_0+R_1M_T.
\] 
For $t\geq2$, note that $q_t^2\geq\|\bar g_{t-1}\|^2/\eta^2$ and the smoothness of $F$ implies:
\[
    \frac{\|g_t-g_{t-1}\|}{q_t}\leq\frac{L\|x_t-x_{t-1}\|}{q_t}=\frac{L\|\bar g_{t-1}\|}{q_{t-1}q_t}\leq\frac{L\eta}{q_1}.
\]
Consequently, we have from the decomposition $\bar g_t=\bar g_{t-1}+(g_t-g_{t-1})+(b_t-b_{t-1})$ and~\eqref{eq:sg-pathwise-bias} that
\begin{align*}
    \|r_t\|=\frac{\|\bar g_t\|}{q_t}\leq&\frac{\|\bar g_{t-1}\|+\|g_t-g_{t-1}\|+\|b_t\|+\|b_{t-1}\|}{q_t}\\
    \leq&\eta\left(1+\frac L{q_1}\right)+\frac{\|b_t\|+\|b_{t-1}\|}{q_1}\\
    \leq&\eta\left(1+\frac L{q_1}\right)+\frac{2B\tau+2\eta\sqrt{c_K}\,\Theta L^2\rho_\tau}{q_1^2}+\frac{2}{q_1}\left(1+\frac{\sqrt{c_K}\,\Theta L\tau}{q_1}\right)M_T\\
    \leq& R_0+R_1 M_T.
\end{align*}

We next bound the moments of $\mathcal R_T$. Since $\|r_t\|\leq R_0+R_1M_T$ for $t\leq T$ almost surely, we have
\[
    \frac{\|\bar g_t\|^2/\eta^2}{q_t^2}=\frac{\|r_t\|^2}{\eta^2}\leq\frac{(R_0+R_1M_T)^2}{\eta^2}.
\]
Therefore, we apply the log-sum in Lemma~\ref{lem:summation} with $a_t=\|r_t\|^2/\eta^2$, $A_t=q_t^2$, and $c=(R_0+R_1M_T)^2/\eta^2$ to obtain
\begin{equation}
    \label{eq:sum-r-sq}
    \mathcal R_T=\eta^2\sum_{t=1}^T\frac{\|\bar g_t\|^2/\eta^2}{q_t^2}\leq\bigl(\eta^2+(R_0+R_1M_T)^2\bigr)\log\frac{q_{T+1}^2}{q_1^2}.
\end{equation}
Taking expectations and applying Cauchy--Schwarz inequality further gives
\begin{align*}
    \E[\mathcal R_T]\leq& 2\sqrt{\E\left[\bigl(\eta^2+2R_0^2+2R_1^2M_T^2\bigr)^2\right]}\sqrt{\E\left[\left(\log\frac{q_{T+1}}{q_1}\right)^2\right]}\notag\\
    \leq& 2\Bigl(\eta^2+2R_0^2+2R_1^2\sqrt{\E[M_T^4]}\Bigr)\sqrt{\E\left[\left(1+\log\frac{q_{T+1}}{q_1}\right)^2\right]}\notag\\
    \leq& 2\Bigl[\eta^2+2R_0^2+2R_1^2\sigma^2\bigl(\log(mKT)+2\bigr)\Bigr]\left(1+\log\frac{\E[q_{T+1}]}{q_1}\right),
\end{align*}
where the second inequality uses Lemma~\ref{lem:conditional-l2-triangle} and $q_{T+1}\geq q_1$, and the last inequality follows from Lemma~\ref{lem:log-power-concavity} with $a=1$ and $p=2$, as well as~\eqref{eq:sg-noise-envelope} with $p=2$.
Squaring both sides of~\eqref{eq:sum-r-sq} gives
\[
    \mathcal R_T^2\leq4\left(\eta^2+(R_0+R_1M_T)^2\right)^2\left(\log\frac{q_{T+1}}{q_1}\right)^2.
\]
Taking expectations and applying Cauchy--Schwarz inequality $\E[XY]\leq(\E[X^2])^{1/2}(\E[Y^2])^{1/2}$ further gives
\begin{align}
    \sqrt{\E[\mathcal R_T^2]}&\leq2\left(\E\left[\left(\eta^2+(R_0+R_1M_T)^2\right)^2\left(\log\frac{q_{T+1}}{q_1}\right)^2\right]\right)^{1/2}\notag\\
    &\leq2\left(\E\left[\left(\eta^2+(R_0+R_1M_T)^2\right)^4\right]\right)^{1/4}\left(\E\left[\left(\log\frac{q_{T+1}}{q_1}\right)^4\right]\right)^{1/4}.\label{eq:sg-energy-l2-holder}
\end{align}
For the first term on the RHS of~\eqref{eq:sg-energy-l2-holder}, we have from $(a+b)^2\leq2a^2+2b^2$ and Lemma~\ref{lem:conditional-l2-triangle} with $p=4$ that
\begin{align}
    \left(\E\left[\left(\eta^2+(R_0+R_1M_T)^2\right)^4\right]\right)^{1/4}&\leq\eta^2+2R_0^2+2R_1^2\left(\E[M_T^8]\right)^{1/4}\notag\\
    &\leq\mathcal A_T+4R_1^2\sigma^2,\label{eq:sg-energy-factor-l4}
\end{align}
where the last inequality follows from~\eqref{eq:sg-noise-envelope} with $p=4$.
To bound the other term in~\eqref{eq:sg-energy-l2-holder}, we apply Lemma~\ref{lem:log-power-concavity} with $a=4$ and $p=4$ to derive
\begin{equation}\label{eq:sg-log-factor-l4}
    \E\left[\left(\log\frac{q_{T+1}}{q_1}\right)^4\right]\leq\E\left[\left(4+\log\frac{q_{T+1}}{q_1}\right)^4\right]\leq\left(4+\log\frac{\E[q_{T+1}]}{q_1}\right)^4.   
\end{equation}
Substituting~\eqref{eq:sg-energy-factor-l4} and~\eqref{eq:sg-log-factor-l4} into~\eqref{eq:sg-energy-l2-holder} yields~\eqref{eq:sg-nonconvex-energy-second-moment}.
The proof is complete.
\end{proof}

\begin{lemma}\label{lem:sg-nonconvex-second-moments}
Let Assumptions~\ref{ass:lower-bounded} and~\ref{ass:stochastic-noise} hold. Assume each $F_i$ is $L$-smooth. Let $\kappa_T:=\max_{t\in[T]} (t-1)\tau_t$. Then for every $T\geq1$,
\begin{align}
    \E[\mathcal G_T]\leq&2(F(x_1)-F^\star)+\frac{2\widetilde B^2\widetilde{\mathcal T}_T}{q_1^3}+\left(L+\frac{8c_K\Theta^2L^4\kappa_T^2}{q_1^3}\right)\E[\mathcal R_T],\label{eq:gcal-moment1}\\
    \sqrt{\E[\mathcal G_T^2]}\leq&2(F(x_1)-F^\star)+\frac{2\widetilde B^2\widetilde{\mathcal T}_T}{q_1^3}+\left(L+\frac{8c_K\Theta^2L^4\kappa_T^2}{q_1^3}\right)\sqrt{\E[\mathcal R_T^2]}\notag\\
    &+\frac{2\sigma}{\sqrt{mq_1}}\left(1+\frac{(K-1)L\tau\Theta}{2q_1}\right)\sqrt{\E[\mathcal G_T]}.\label{eq:gcal-moment2}
\end{align}
\end{lemma}
\begin{proof}
Define the noiseless local quantities at round $t$ by
\[
    \widetilde x_{i,0}^t:=x_t,\qquad \widetilde x_{i,k+1}^t:=\widetilde x_{i,k}^t-\frac{\tau_t}{q_t}\nabla F_i(\widetilde x_{i,k}^t),\qquad \widetilde g_t^{(i)}:=\frac1K\sum_{k=0}^{K-1}\nabla F_i(\widetilde x_{i,k}^t).
\]
Then it follows that
\[
    x_{i,k+1}^t-\widetilde x_{i,k+1}^t=x_{i,k}^t-\widetilde x_{i,k}^t-\frac{\tau_t}{q_t}\bigl(\nabla F_i(x_{i,k}^t)-\nabla F_i(\widetilde x_{i,k}^t)\bigr)-\frac{\tau_t}{q_t}\xi_{i,k}^t.
\]
Lemma~\ref{lem:conditional-l2-triangle} with $\mathcal G=\mathcal F_t$, $L$-smoothness, and~\eqref{eq:sg-implied-variance} therefore give
\[
\begin{aligned}
    \sqrt{\E_t[\|x_{i,k+1}^t-\widetilde x_{i,k+1}^t\|^2]}&\leq\sqrt{\E_t[\|x_{i,k}^t-\widetilde x_{i,k}^t\|^2]}
    +\frac{\tau_t}{q_t}\sqrt{\E_t[\|\nabla F_i(x_{i,k}^t)-\nabla F_i(\widetilde x_{i,k}^t)\|^2]}+\frac{\tau_t}{q_t}\sqrt{\E_t[\|\xi_{i,k}^t\|^2]}\\
    &\leq\left(1+\frac{L\tau_t}{q_t}\right)\sqrt{\E_t[\|x_{i,k}^t-\widetilde x_{i,k}^t\|^2]}+\frac{\tau_t\sigma}{q_t}.
\end{aligned}
\]
Since $x_{i,0}^t=\widetilde x_{i,0}^t=x_t$, applying the above inequality recursively gives
\begin{equation}
    \sqrt{\E_t[\|x_{i,k}^t-\widetilde x_{i,k}^t\|^2]}\leq\frac{\tau_t\sigma}{q_t}\sum_{s=0}^{k-1}\left(1+\frac{L\tau_t}{q_t}\right)^s\leq\frac{k\tau_t\Theta\sigma}{q_t}\leq\frac{k\tau\Theta\sigma}{q_1}.
    \label{eq:sqrt-xk-tilde_xk}
\end{equation}
Note that $\bar g_t^{(i)}-\widetilde g_t^{(i)}=\frac1K\sum_{k=0}^{K-1}\bigl(\nabla F_i(x_{i,k}^t)-\nabla F_i(\widetilde x_{i,k}^t)+\xi_{i,k}^t\bigr)$, we obtain by applying Lemma~\ref{lem:conditional-l2-triangle}, smoothness, and~\eqref{eq:sqrt-xk-tilde_xk} that
\[
\begin{aligned}
    \sqrt{\E_t[\|\bar g_t^{(i)}-\widetilde g_t^{(i)}\|^2]}&\leq\frac1K\sum_{k=0}^{K-1}\left(L\sqrt{\E_t[\|x_{i,k}^t-\widetilde x_{i,k}^t\|^2]}+\sqrt{\E_t[\|\xi_{i,k}^t\|^2]}\right)\\
    &\leq\frac\sigma K\sum_{k=0}^{K-1}\left(1+\frac{kL\tau\Theta}{q_1}\right)=\sigma\left(1+\frac{(K-1)L\tau\Theta}{2q_1}\right).
\end{aligned}
\]
For different clients $i\ne j$, conditional independence gives
\[
    \E_t\left[\left\langle\bar g_t^{(i)}-\E_t[\bar g_t^{(i)}],\bar g_t^{(j)}-\E_t[\bar g_t^{(j)}]\right\rangle\right]=0.
\]
Since $\widetilde g_t^{(i)}$ is $\mathcal F_t$-measurable, we therefore have
\begin{equation}
\begin{aligned}
    \E_t[\|\zeta_t\|^2]&=\frac1{m^2}\sum_{i=1}^m\E_t\left[\|\bar g_t^{(i)}-\E_t[\bar g_t^{(i)}]\|^2\right]\\
    &=\frac1{m^2}\sum_{i=1}^m\left(\E_t\left[\|\bar g_t^{(i)}-\widetilde g_t^{(i)}\|^2\right]-\left\|\E_t[\bar g_t^{(i)}]-\widetilde g_t^{(i)}\right\|^2\right)\\
    &\leq\frac1{m^2}\sum_{i=1}^m\E_t\left[\|\bar g_t^{(i)}-\widetilde g_t^{(i)}\|^2\right]\leq\frac{\sigma^2}{m}\left(1+\frac{(K-1)L\tau\Theta}{2q_1}\right)^2.
\end{aligned}
    \label{eq:sg-centered-direction-variance}
\end{equation}
Note $\bar g_t=g_t+\E_t[\hat b_t]+\zeta_t$ from~\eqref{eq:sg-averaged-noise-moments}, we have from $L$-smoothness of $F$ and $r_t=-\bar g_t/q_t$ that
\begin{align*}
    F(x_{t+1})&\leq F(x_t)+\langle g_t,r_t\rangle+\frac L2\|r_t\|^2\\
    &=F(x_t)-\frac{\|g_t\|^2}{q_t}-\frac{\langle g_t,\E_t[\hat b_t]\rangle}{q_t}-\frac{\langle g_t,\zeta_t\rangle}{q_t}+\frac L2\|r_t\|^2\\
    &\leq F(x_t)-\frac{\|g_t\|^2}{2q_t}+\frac{\E_t[\|\hat b_t\|^2]}{2q_t}-\frac{\langle g_t,\zeta_t\rangle}{q_t}+\frac L2\|r_t\|^2,
\end{align*}
where the last inequality uses Young's inequality and $\|\E_t[\hat b_t]\|^2\leq\E_t[\|\hat b_t\|^2]$. Summing the above inequality from $t=1$ to $T$ and rearranging, we get
\begin{equation}\label{eq:sg-pathwise-weighted-gradient}
    \mathcal G_T\leq2(F(x_1)-F^\star)+\sum_{t=1}^T\frac{\E_t[\|\hat b_t\|^2]}{q_t}+L\mathcal R_T-2\sum_{t=1}^T\frac{\langle g_t,\zeta_t\rangle}{q_t}.
\end{equation}
Applying Lemma~\ref{lem:sg-local-moments} into~\eqref{eq:sg-pathwise-weighted-gradient} and using $q_t\geq q_1$ then implies
\begin{equation}\label{eq:sum-core1}
    \mathcal G_T\leq2(F(x_1)-F^\star)+\frac{2\widetilde B^2\widetilde{\mathcal T}_T}{q_1^3}+\left(L+\frac{8c_K\Theta^2L^4\kappa_T^2}{q_1^3}\right)\mathcal R_T-2\sum_{t=1}^T\frac{\langle g_t,\zeta_t\rangle}{q_t}.
\end{equation}
Taking expectations and using $\E[\langle g_t,\zeta_t\rangle]=\E[\langle g_t,\E_t[\zeta_t]\rangle]=0$ for all $t$ proves~\eqref{eq:gcal-moment1}. 
Additionally, for any two rounds $s<t$, we have
\[
    \E[\langle g_s,\zeta_s\rangle\langle g_t,\zeta_t\rangle]=\E[\langle g_s,\zeta_s\rangle\langle g_t,\E_t[\zeta_t]\rangle]=0.
\]
Therefore, we obtain from~\eqref{eq:sg-centered-direction-variance} that
\begin{align}
    &\E\left[\left|2\sum_{t=1}^T\frac{\langle g_t,\zeta_t\rangle}{q_t}\right|^2\right]=4\sum_{t=1}^T\E\left[\frac{\langle g_t,\zeta_t\rangle^2}{q_t^2}\right]\leq \frac{4}{q_1}\sum_{t=1}^T\E\left[\frac{\|g_t\|^2\|\zeta_t\|^2}{q_t}\right]\notag\\
    =&\frac{4}{q_1}\sum_{t=1}^T\E\left[\frac{\|g_t\|^2}{q_t}\E_t[\|\zeta_t\|^2]\right]\leq\frac{4\sigma^2}{mq_1}\left(1+\frac{(K-1)L\tau\Theta}{2q_1}\right)^2\E[\mathcal G_T].\label{eq:noise-term-second-moment}
\end{align}
Finally, taking squares, expectations, and square roots on both sides of~\eqref{eq:sum-core1}, then applying Lemma~\ref{lem:conditional-l2-triangle} with $p=2$ and~\eqref{eq:noise-term-second-moment} yields
\[
\begin{aligned}
    \sqrt{\E[\mathcal G_T^2]}&\leq2(F(x_1)-F^\star)+\frac{2\widetilde B^2\widetilde{\mathcal T}_T}{q_1^3}
    +\left(L+\frac{8c_K\Theta^2L^4\kappa_T^2}{q_1^3}\right)\sqrt{\E[\mathcal R_T^2]}
    +\sqrt{\E\left[\left|2\sum_{t=1}^T\frac{\langle g_t,\zeta_t\rangle}{q_t}\right|^2\right]}\\
    &\leq2(F(x_1)-F^\star)+\frac{2\widetilde B^2\widetilde{\mathcal T}_T}{q_1^3}+\left(L+\frac{8c_K\Theta^2L^4\kappa_T^2}{q_1^3}\right)\sqrt{\E[\mathcal R_T^2]}
    +\frac{2\sigma}{\sqrt{mq_1}}\left(1+\frac{(K-1)L\tau\Theta}{2q_1}\right)\sqrt{\E[\mathcal G_T]},
\end{aligned}
\]
which proves~\eqref{eq:gcal-moment2}.
\end{proof}

\subsection{Main Result and Proof}
\begin{theorem}[Smooth nonconvex convergence]
\label{thm:sg-smooth-nonconvex}
Let Assumptions~\ref{ass:lower-bounded} and~\ref{ass:stochastic-noise} hold. Assume each $F_i$ is $L$-smooth. Let $\{x_t\}_{t\ge 1}$ be the iterates generated by Algorithm~\ref{alg:FL-adagrad} with $\tau_t$ chosen as in~\eqref{eq:tau_t-scheme} and arbitrary $\gamma,\eta>0$. Then for any $T\geq1$, we have
\begin{equation*}
    \E[\|\nabla F(\bar x_T)\|^2]\lesssim\frac{\mathcal C_T^2}{T}+\frac{\sigma\mathcal C_T}{\sqrt{mKT}},
\end{equation*}
where $\mathcal C_T>0$ is defined in~\eqref{def:mathcal_C_T}.
\end{theorem}

\begin{proof}[Proof of Theorem~\ref{thm:sg-smooth-nonconvex}]
With $K\ge2$, it follows from scheme~\eqref{eq:tau_t-scheme} that 
\[
\kappa_T\lesssim \frac{1}{K-1},\qquad \widetilde{\mathcal T}_T\lesssim \frac{1}{(K-1)^2},\quad \forall T\ge1.
\]
We first bound $\E[q_{T+1}]$. To this end, we apply Lemma~\ref{lem:pp-predictable} with $c=q_1^2$ and $Z_t=\|\bar g_t\|^2/\eta^2$ to obtain
\begin{align}
    \E[q_{T+1}]&\leq\sqrt2\,\E\left[\sqrt{q_1^2+\frac1{\eta^2}\sum_{t=1}^T\E_t[\|\bar g_t\|^2]}\right]\notag\\
    &\leq\sqrt2\,\E\left[\sqrt{q_1^2+\frac2{\eta^2}\sum_{t=1}^T\|g_t\|^2+\frac2{\eta^2}\sum_{t=1}^T\E_t[\|\hat b_t\|^2]+\frac{2T\sigma^2}{\eta^2mK}}\right]\notag\\
    &\leq\sqrt2q_1+\frac2\eta\E\left[\sqrt{\sum_{t=1}^T\|g_t\|^2}\right]+\frac2\eta\E\left[\sqrt{\sum_{t=1}^T\E_t[\|\hat b_t\|^2]}\right]+\frac{2\sigma}{\eta}\sqrt{\frac{T}{mK}},\label{eq:sg-metric-comparison-first}
\end{align}
where the second inequality uses~\eqref{eq:gt_sq-exp}, and the last follows from the basic inequality $\sqrt{a+b}\leq\sqrt a+\sqrt b$. 

Next, we bound the two expectation terms in~\eqref{eq:sg-metric-comparison-first}. Since $q_t\leq q_{T+1}$ for all $t\leq T+1$, we apply Cauchy--Schwarz inequality to obtain
\begin{equation}\label{eq:nc-qT-bd1}
    \E\left[\sqrt{\sum_{t=1}^T\|g_t\|^2}\right]\leq\E\left[\sqrt{q_{T+1}\sum_{t=1}^T\frac{\|g_t\|^2}{q_t}}\right]= \E\left[\sqrt{q_{T+1}\mathcal G_T}\right]\leq\sqrt{\E[q_{T+1}]}\sqrt{\E[\mathcal G_T]}.
\end{equation}
For the other term, applying Lemma~\ref{lem:sg-local-moments} and concavity of $\sqrt{\cdot}$ gives
\[
    \E\left[\sqrt{\sum_{t=1}^T\E_t[\|\hat b_t\|^2]}\right]\leq\sqrt{\frac{2\widetilde B^2\widetilde{\mathcal T}_T}{q_1^2}+\frac{8c_K\Theta^2L^4\kappa_T^2}{q_1^2}\E[\mathcal R_T]}.
\]
Substituting the above two bounds into~\eqref{eq:sg-metric-comparison-first}, and applying $2\sqrt{ab}\leq a+b$ and $\sqrt{a+b}\leq\sqrt a+\sqrt b$ yields
\begin{align*}
    \E[q_{T+1}]\leq& \sqrt2q_1+\frac2\eta\sqrt{\E[q_{T+1}]\E[\mathcal G_T]}+\frac2\eta\sqrt{\frac{2\widetilde B^2\widetilde{\mathcal T}_T}{q_1^2}+\frac{8c_K\Theta^2L^4\kappa_T^2}{q_1^2}\E[\mathcal R_T]}+\frac{2\sigma}{\eta}\sqrt{\frac{T}{mK}}\\
    \leq& \sqrt2q_1+\frac{\E[q_{T+1}]}2+\frac2{\eta^2}\E[\mathcal G_T]+\frac{2\sqrt2\widetilde B\sqrt{\widetilde{\mathcal T}_T}}{\eta q_1}+\frac{4\sqrt{2c_K}\,\Theta L^2\kappa_T}{\eta q_1}\sqrt{\E[\mathcal R_T]}+\frac{2\sigma}{\eta}\sqrt{\frac{T}{mK}},
\end{align*}
which together with $2\sqrt{\E[\mathcal R_T]}\leq1+\E[\mathcal R_T]$ further simplifies to
\begin{equation}\label{eq:q-moment1-1}
    \E[q_{T+1}]\leq2\sqrt2q_1+\frac4{\eta^2}\E[\mathcal G_T]+\frac{4\sqrt2\widetilde B\sqrt{\widetilde{\mathcal T}_T}}{\eta q_1}+\frac{4\sqrt{2c_K}\,\Theta L^2\kappa_T}{\eta q_1}\left(1+\E[\mathcal R_T]\right)+\frac{4\sigma}{\eta}\sqrt{\frac{T}{mK}},
\end{equation}
Plugging~\eqref{eq:gcal-moment1} and~\eqref{eq:sg-expected-energy} into~\eqref{eq:q-moment1-1} and rearranging the terms gives
\begin{equation}\label{eq:q-moment1-core}
    \E[q_{T+1}]\leq c_1+\frac{4\sigma}{\eta}\sqrt{\frac{T}{mK}}+2c_2\mathcal A_T\left(1+\log\frac{\E[q_{T+1}]}{q_1}\right),
\end{equation}
where the positive constants $c_1$ and $c_2$ are defined as (we use~\eqref{eq:quantities-estimates})
\begin{align*}
    c_1&:=2\sqrt2q_1+\frac{8(F(x_1)-F^\star)}{\eta^2}+\frac{8\widetilde B^2\widetilde{\mathcal T}_T}{\eta^2q_1^3}+\frac{4\sqrt2\widetilde B\sqrt{\widetilde{\mathcal T}_T}}{\eta q_1}+\frac{4\sqrt{2c_K}\,\Theta L^2\kappa_T}{\eta q_1}\lesssim 1,\\
    c_2&:=\frac4{\eta^2}\left(L+\frac{8c_K\Theta^2L^4\kappa_T^2}{q_1^3}\right)+\frac{4\sqrt{2c_K}\,\Theta L^2\kappa_T}{\eta q_1}\lesssim 1.
\end{align*}
Solving the inequality~\eqref{eq:q-moment1-core} for $\E[q_{T+1}]$ using Lemma~\ref{lem:log-solution} gives
\begin{equation}
    \E[q_{T+1}]\leq 2c_1+\frac{8\sigma}{\eta}\sqrt{\frac{T}{mK}}+4c_2\mathcal A_T\log\frac{4c_2\mathcal A_T}{q_1}.
    \label{eq:sg-closed-metric}
\end{equation}
Since $R_0\lesssim 1$ and $R_1\lesssim 1$ from~\eqref{eq:quantities-estimates}, we have $\mathcal A_T\lesssim 1+\sigma^2(1+\log(mKT))$ and hence it follows from~\eqref{eq:sg-closed-metric} that
\begin{align}
    \E[q_{T+1}]\lesssim&1+\sigma\sqrt{\frac{T}{mK}}+\big(1+\sigma^2(1+\log(mKT))\big)\log\bigl(\exp(1)+\sigma^2(1+\log(mKT))\bigr)\notag\\
    \lesssim&1+\sigma\sqrt{\frac{T}{mK}}+\sigma^2(1+\log(mKT))\log\bigl(1+\sigma^2(1+\log(mKT))\bigr)=\mathcal U_T.\label{q-moment1-final}
\end{align}

Applying the bound~\eqref{q-moment1-final} in~\eqref{eq:sg-expected-energy} then gives
\begin{equation*}
    \E[\mathcal R_T]\leq2\mathcal A_T\left(1+\log\frac{\E[q_{T+1}]}{q_1}\right)\lesssim \bigl(1+\sigma^2(1+\log(mKT))\bigr)(1+\log\mathcal U_T)=\mathcal C_T,
\end{equation*}
applying which into~\eqref{eq:sg-nonconvex-energy-second-moment} and~\eqref{eq:gcal-moment1} respectively yields
\begin{align}
    \sqrt{\E[\mathcal R_T^2]}\leq & \bigl(\mathcal A_T+4R_1^2\sigma^2\bigr)\left(4+\log\frac{\E[q_{T+1}]}{q_1}\right)\lesssim \mathcal A_T\left(1+\log\frac{\E[q_{T+1}]}{q_1}\right)\lesssim \mathcal C_T,\label{eq:nc-RT-bd}\\
    \E[\mathcal G_T]\lesssim & 1+\E[\mathcal R_T]\lesssim 1+\mathcal C_T\lesssim \mathcal C_T.\notag
\end{align}
Applying the above two bounds into~\eqref{eq:gcal-moment2} gives
\[
    \sqrt{\E[\mathcal G_T^2]}\lesssim 1+\sqrt{\E[\mathcal R_T^2]}+\frac{\sigma}{\sqrt{m}}\sqrt{\E[\mathcal G_T]}\lesssim 1+\mathcal C_T+\frac{\sigma}{\sqrt{m}}\sqrt{\mathcal C_T}\lesssim \mathcal C_T+\sigma\sqrt{\frac{\mathcal C_T}{m}}\lesssim \mathcal C_T,
\]
where the last inequality uses $\sigma^2\lesssim \mathcal C_T$ and $m>1$.

To proceed, we obtain from the monotonicity of $q_t$ and Cauchy--Schwarz inequality that
\begin{equation}\label{eq:g_sq-sum}
    \sum_{t=1}^T\E[\|g_t\|^2]\leq\E[q_{T+1}\mathcal G_T]\leq\sqrt{\E[q_{T+1}^2]}\sqrt{\E[\mathcal G_T^2]}\lesssim \mathcal C_T\sqrt{\E[q_{T+1}^2]}.
\end{equation}
It remains to control $\E[q_{T+1}^2]$. Combining~\eqref{eq:nocv-outer-method} and~\eqref{eq:gt_sq-exp} gives
\begin{equation}\label{eq:sg-q-sq-recursion}
    \E[q_{T+1}^2]=q_1^2+\frac1 {\eta^2}\sum_{t=1}^T\E[\|\bar g_t\|^2]\leq q_1^2+\frac2{\eta^2}\sum_{t=1}^T\E[\|g_t\|^2]+\frac2{\eta^2}\sum_{t=1}^T\E[\|\hat b_t\|^2]+\frac{2T\sigma^2}{\eta^2mK}.
\end{equation}
Applying Lemma~\ref{lem:sg-local-moments} and using $\E[\mathcal R_T]\lesssim\mathcal C_T$ leads to
\[
    \sum_{t=1}^T\E[\|\hat b_t\|^2]\lesssim 1+\E[\mathcal R_T]\lesssim \mathcal C_T.
\]
Substituting the above inequality and~\eqref{eq:g_sq-sum} into~\eqref{eq:sg-q-sq-recursion} and noting that $1\lesssim \sqrt{\E[q_{T+1}^2]}$ gives
\begin{equation*}
    \E[q_{T+1}^2]\lesssim 1+\mathcal C_T\sqrt{\E[q_{T+1}^2]}+\mathcal C_T+\frac{2T\sigma^2}{\eta^2mK}\lesssim \mathcal C_T\sqrt{\E[q_{T+1}^2]}+\frac{2T\sigma^2}{\eta^2mK}.
\end{equation*}
Consequently, there exist constants $C>0$ such that
\[
    \E[q_{T+1}^2]\leq C\mathcal C_T\sqrt{\E[q_{T+1}^2]}+\frac{2T\sigma^2}{\eta^2mK}\leq\frac{\E[q_{T+1}^2]}{2}+\frac{C^2}{2}\mathcal C_T^2+\frac{2T\sigma^2}{\eta^2mK},
\]
which simplifies to
\begin{equation}
    \E[q_{T+1}^2]\lesssim \mathcal C_T^2+\frac{T\sigma^2}{mK}.
\label{eq:sg-q-sq-final}
\end{equation}
Finally, substituting~\eqref{eq:sg-q-sq-final} into~\eqref{eq:g_sq-sum} and dividing by $T$ gives
\[
    \E[\|\nabla F(\bar x_T)\|^2]=\frac1T\sum_{t=1}^T\E[\|g_t\|^2]\lesssim \frac{\mathcal C_T}{T}\sqrt{\E[q_{T+1}^2]}\leq \frac{\mathcal C_T}{T}\Big(\mathcal C_T+\sigma\sqrt{\frac{T}{mK}}\Big) = \frac{\mathcal C_T^2}{T}+\frac{\sigma\mathcal C_T}{\sqrt{mKT}},
\]
which completes the proof.
\end{proof}

\begin{corollary}[Deterministic smooth nonconvex convergence]\label{cor:sg-nonconvex-deterministic}
Consider the same setting as in Theorem~\ref{thm:sg-smooth-nonconvex}, except that the local oracle is exact, i.e., $\sigma=0$. Then for any $T\geq1$, we have
\[
    \E[\|\nabla F(\bar x_T)\|^2]=\frac1T\sum_{t=1}^T\|\nabla F(x_t)\|^2\lesssim \frac{1}{T}.
\]
\end{corollary}
\begin{proof}
When $\sigma=0$, we have from the definition~\eqref{def:mathcal_C_T} that
\[
    \mathcal C_T=1+\log\mathcal U_T=1.
\]
Then the result follows immediately from Theorem~\ref{thm:sg-smooth-nonconvex} by plugging in $\mathcal C_T=1$ and $\sigma=0$.
\end{proof}

\section{Smooth convex objectives}\label{app:smooth-convex}
Assume in this subsection that every $F_i$ is convex and smooth, and let $x^\star$ minimize $F$. Thus all results from the smooth nonconvex analysis remain valid. 

For convex objectives, define
\begin{equation}
    \epsilon_t:=\frac1{mK}\sum_{i=1}^m\sum_{k=0}^{K-1}\E_t\left[F_i(x_t)-F_i(x_{i,k}^t)-\langle\nabla F_i(x_{i,k}^t),x_t-x_{i,k}^t\rangle\right]\ge 0,
    \label{eq:eps-def}
\end{equation}
where $\nabla F_i(x_{i,k}^t)\in\partial F_i(x_{i,k}^t)$ denotes a subgradient if $F_i$ is nonsmooth. We also denote 
\[
d_t:=x_t-x^\star,\qquad D_T:=\max_{1\leq t\leq T+1}\|d_t\|^2.
\]

\begin{lemma}
\label{lem:sg-smooth-convex-error}
Let Assumptions~\ref{ass:lower-bounded} and~\ref{ass:stochastic-noise} hold. Assume each $F_i$ is convex. Then, for every $t\geq1$,
\begin{equation}\label{eq:cvx-F-bound}
    F(y)\geq F(x_t)+\langle\E_t[\bar g_t],y-x_t\rangle-\epsilon_t,\quad \forall y\in\R^d.
\end{equation}
If additionally each $F_i$ is $L$-smooth, then
\begin{equation}
    \label{eq:eps-bd}
    \epsilon_t\leq\frac{c_K\Theta^2L\tau_t^2}{q_t^2}\left[\frac1m\sum_{i=1}^m\bigl(\|\nabla F_i(x^\star)\|+\sigma\bigr)^2+L^2\|d_t\|^2\right].
\end{equation}
\end{lemma}
\begin{proof}
For each client, convexity of $F_i$ at $x_{i,k}^t$ implies
\[
\begin{aligned}
    F_i(y)&\geq F_i(x_{i,k}^t)+\langle\nabla F_i(x_{i,k}^t),y-x_{i,k}^t\rangle\\
    &=F_i(x_t)+\langle\nabla F_i(x_{i,k}^t),y-x_t\rangle-\left[F_i(x_t)-F_i(x_{i,k}^t)-\langle\nabla F_i(x_{i,k}^t),x_t-x_{i,k}^t\rangle\right].
\end{aligned}
\]
Then~\eqref{eq:cvx-F-bound} follows by averaging this inequality over $i,k$ and taking $\E_t$, and noting that
\[
    \frac1{mK}\sum_{i=1}^m\sum_{k=0}^{K-1}\E_t[\nabla F_i(x_{i,k}^t)]=\frac1{mK}\sum_{i=1}^m\sum_{k=0}^{K-1}\E_t[g_{i,k}^t-\xi_{i,k}^t]=\E_t[\bar g_t].
\]

To prove~\eqref{eq:eps-bd}, we use convexity and smoothness to obtain
\[
    0\leq F_i(x_t)-F_i(x_{i,k}^t)-\langle\nabla F_i(x_{i,k}^t),x_t-x_{i,k}^t\rangle\leq\frac L2\|x_t-x_{i,k}^t\|^2,
\]
which together with~\eqref{eq:sqrt-xk-xt} and $K^{-1}\sum_{k=0}^{K-1}k^2=c_K$ yields
\begin{align*}
    0\leq\epsilon_t\leq\frac{L}{2mK}\sum_{i=1}^m\sum_{k=0}^{K-1}\E_t[\|x_t-x_{i,k}^t\|^2]\leq\frac{c_K\Theta^2L\tau_t^2}{2mq_t^2}\sum_{i=1}^m\bigl(\|\nabla F_i(x_t)\|+\sigma\bigr)^2.
\end{align*}
Then~\eqref{eq:eps-bd} follows by plugging in the following inequality due to smoothness of $F_i$ around $x^\star$ and $(a+b)^2\leq2a^2+2b^2$
\[
    \bigl(\|\nabla F_i(x_t)\|+\sigma\bigr)^2\leq2\bigl(\|\nabla F_i(x^\star)\|+\sigma\bigr)^2+2L^2\|x_t-x^\star\|^2.
\]
The proof is complete.
\end{proof}

\begin{lemma}\label{lem:sg-smooth-convex-moments}
Let Assumptions~\ref{ass:lower-bounded} and~\ref{ass:stochastic-noise} hold. Assume each $F_i$ is convex and $L$-smooth. If $\tau_t$ is chosen as in~\eqref{eq:tau_t-scheme}, then for any $T\geq1$, we have
\[
    \sqrt{\E[D_T^2]}\lesssim\mathcal C_T,
\]
where $\mathcal C_T>0$ is defined in~\eqref{def:mathcal_C_T}.
\end{lemma}
\begin{proof}
Note that $d_{t+1}=d_t-\bar g_t/q_t$ from the update rules~\eqref{eq:nocv-outer-method} and also $\bar g_t=\E_t[\bar g_t]+\zeta_t$, hence we have
\[
    \|d_{t+1}\|^2=\left\|d_t-\frac{\bar g_t}{q_t}\right\|^2=\|d_t\|^2-\frac{2}{q_t}\langle\E_t[\bar g_t],d_t\rangle-\frac{2}{q_t}\langle\zeta_t,d_t\rangle+\|r_t\|^2.
\]
Adding $2(F(x_t)-F^\star)/q_t$ to both sides of the above equality and applying $F(x_t)-F^\star\leq\langle\E_t[\bar g_t],d_t\rangle+\epsilon_t$ (which is~\eqref{eq:cvx-F-bound} with $y=x^\star$) yields
\begin{equation}\label{eq:d-recursion-1}
    \|d_{t+1}\|^2+\frac{2(F(x_t)-F^\star)}{q_t}
    \leq\|d_t\|^2+\|r_t\|^2-\frac{2}{q_t}\langle\zeta_t,d_t\rangle+\frac{2\epsilon_t}{q_t}.
\end{equation}
To bound the term $2\epsilon_t/q_t$, we apply~\eqref{eq:eps-bd} from Lemma~\ref{lem:sg-smooth-convex-error} and note $q_t\geq q_1$ to obtain
\begin{align}
    \frac{2\epsilon_t}{q_t}
    &\leq\frac{2c_K\Theta^2L\tau_t^2}{q_t^3}\left[\frac1m\sum_{i=1}^m\bigl(\|\nabla F_i(x^\star)\|+\sigma\bigr)^2+L^2\|d_t\|^2\right]\notag\\
    &\leq\frac{2c_K\Theta^2L^3\tau_t^2}{q_1^3}\|d_t\|^2+\frac{2c_K\Theta^2L\tau_t^2}{mq_1^3}\sum_{i=1}^m\bigl(\|\nabla F_i(x^\star)\|+\sigma\bigr)^2.\label{eq:eps-q-bd}
\end{align}
For ease of notation, denote the coefficients as
\[
    \alpha_t:=\frac{2c_K\Theta^2L^3\tau_t^2}{q_1^3},\qquad \beta_t:=\frac{2c_K\Theta^2L\tau_t^2}{mq_1^3}\sum_{i=1}^m\bigl(\|\nabla F_i(x^\star)\|+\sigma\bigr)^2.
\]
Then substituting~\eqref{eq:eps-q-bd} into~\eqref{eq:d-recursion-1} yields
\begin{equation}
    \|d_{t+1}\|^2+\frac{2(F(x_t)-F^\star)}{q_t}\leq(1+\alpha_t)\|d_t\|^2+\|r_t\|^2+\beta_t-\frac2{q_t}\langle\zeta_t,d_t\rangle.
\label{eq:sg-smooth-distance-recursion}
\end{equation}
To proceed, let $P_t:=\prod_{j=1}^t(1+\alpha_j)$ with $P_0:=1$, and
\[
    \mathcal M_t:=-2\sum_{s=1}^t\frac{\langle\zeta_s,d_s\rangle}{P_s q_s},\qquad 1\leq t\leq T.
\]
Since $\E_t[\zeta_t]=0$, we have
\[
    \E_t[\mathcal M_t]=\mathcal M_{t-1}-\frac{2}{P_tq_t}\bigl\langle\E_t[\zeta_t],d_t\bigr\rangle=\mathcal M_{t-1},
\]
which implies that $\{\mathcal M_t\}_{t=1}^T$ is a martingale.
Since $(1+\alpha_t)/P_t=1/P_{t-1}$, dividing both sides of~\eqref{eq:sg-smooth-distance-recursion} by $P_t$ and unrolling gives
\[
    \frac{\|d_{t+1}\|^2}{P_t}\leq\frac{\|d_t\|^2}{P_{t-1}}+\frac{\|r_t\|^2+\beta_t}{P_t}-\frac{2\langle\zeta_t,d_t\rangle}{P_tq_t}\leq\|d_1\|^2+\sum_{s=1}^t\frac{\|r_s\|^2+\beta_s}{P_s}+\mathcal M_t.
\]
Then it follows from $1\leq P_t\leq P_T$ that
\[
    D_T=\max_{1\leq t\leq T+1}\|d_t\|^2\leq P_T\left(\|d_1\|^2+\sum_{t=1}^T\beta_t+\mathcal R_T+\max_{1\leq t\leq T}|\mathcal M_t|\right),
\]
which together with Lemma~\ref{lem:conditional-l2-triangle} yields
\begin{equation}
    \label{eq:DT-moment2-1}
    \sqrt{\E[D_T^2]}\leq P_T\left(\|d_1\|^2+\sum_{t=1}^T\beta_t+\sqrt{\E[\mathcal R_T^2]}+\sqrt{\E\left[\max_{1\leq t\leq T}|\mathcal M_t|^2\right]}\right).
\end{equation}
Next, we bound the last term in~\eqref{eq:DT-moment2-1}. Since $\zeta_t/q_t=-(r_t-\E_t[r_t])$, we have
\[
    \frac{\E_t[\|\zeta_t\|^2]}{q_t^2}=\E_t[\|r_t-\E_t[r_t]\|^2]=\E_t[\|r_t\|^2]-\|\E_t[r_t]\|^2\leq\E_t[\|r_t\|^2],
\]
which together with Lemma~\ref{lem:doob-l2} and $\E[\langle g_s,\zeta_s\rangle\langle g_t,\zeta_t\rangle]=\E[\langle g_s,\zeta_s\rangle\langle g_t,\E_t[\zeta_t]\rangle]=0,~\forall s<t$ yield
\begin{align}
    \E\left[\max_{1\leq t\leq T}|\mathcal M_t|^2\right]&\leq4\E[|\mathcal M_T|^2]=16\sum_{t=1}^T\E\left[\frac{\langle\zeta_t,d_t\rangle^2}{P_t^2q_t^2}\right]\leq 16\sum_{t=1}^T\E\left[\frac{\|\zeta_t\|^2\|d_t\|^2}{q_t^2}\right]\notag\\
    &\leq16\sum_{t=1}^T\E\left[\|d_t\|^2\E_t[\|r_t\|^2]\right]=16\sum_{t=1}^T\E\left[\|d_t\|^2\|r_t\|^2\right]\notag\\
    &\leq16\E[D_T\mathcal R_T]\leq 16\sqrt{\E[D_T^2]}\sqrt{\E[\mathcal R_T^2]}.\label{eq:mcal-bd}
\end{align}
Applying~\eqref{eq:mcal-bd} into~\eqref{eq:DT-moment2-1} and using Young's inequality gives
\begin{align*}
    \sqrt{\E[D_T^2]}\leq& P_T\left(\|d_1\|^2+\sum_{t=1}^T\beta_t+\sqrt{\E[\mathcal R_T^2]}+4\sqrt{\sqrt{\E[D_T^2]}\sqrt{\E[\mathcal R_T^2]}}\right)\\
    \leq& P_T\left(\|d_1\|^2+\sum_{t=1}^T\beta_t\right)+P_T\sqrt{\E[\mathcal R_T^2]}+\frac12\sqrt{\E[D_T^2]}+8P_T^2\sqrt{\E[\mathcal R_T^2]}.
\end{align*}
Additionally, we have from $1+x\leq e^x$, and $\sum_{t=1}^T\tau_t^2=\widetilde{\mathcal T}_T$, and~\eqref{eq:quantities-estimates} that
\begin{align*}
    P_T&\leq\exp\Big(\sum_{t=1}^{T}\alpha_t\Big)=\exp\left(\frac{2c_K\Theta^2L^3}{q_1^3}\widetilde{\mathcal T}_T\right)\lesssim 1,\\
    \sum_{t=1}^T\beta_t&=\frac{2c_K\Theta^2L\widetilde{\mathcal T}_T}{mq_1^3}\sum_{i=1}^m\bigl(\|\nabla F_i(x^\star)\|+\sigma\bigr)^2\lesssim 1.
\end{align*}
Therefore, we conclude that
\begin{equation}
    \sqrt{\E[D_T^2]}\leq2P_T\left(\|d_1\|^2+\sum_{t=1}^T\beta_t\right)+(2P_T+16P_T^2)\sqrt{\E[\mathcal R_T^2]}\lesssim 1+\sqrt{\E[\mathcal R_T^2]}\lesssim \sqrt{\E[\mathcal R_T^2]}.
\label{eq:sg-smooth-max-distance-moment}
\end{equation}

It remains to establish a refined bound on $\E[\mathcal R_T^2]$. Note that we may not directly apply the result~\eqref{eq:nc-RT-bd} from the proof of Theorem~\ref{thm:sg-smooth-nonconvex}, since it relies on the stronger scheme $\tau_t\lesssim \frac{1}{(K-1)t}$. 
Instead, we start by plugging~\eqref{eq:nc-qT-bd1} into~\eqref{eq:sg-metric-comparison-first} and applying Young's inequality to obtain
    \begin{align*}
        \E[q_{T+1}]\leq&\sqrt2q_1+\frac2\eta\sqrt{\E[q_{T+1}]\E[\mathcal G_T]}+\frac2\eta\E\left[\sqrt{\sum_{t=1}^T\E_t[\|\hat b_t\|^2]}\right]+\frac{2\sigma}{\eta}\sqrt{\frac{T}{mK}}\\
        \leq& \sqrt2q_1+\frac{\E[q_{T+1}]}2+\frac2{\eta^2}\E[\mathcal G_T]+1+\frac1{\eta^2}\sum_{t=1}^T\E[\|\hat b_t\|^2]+\frac{2\sigma}{\eta}\sqrt{\frac{T}{mK}},
    \end{align*}
    which simplifies to
    \begin{equation}
        \E[q_{T+1}]\leq 2+2\sqrt{2}q_1+\frac{4}{\eta^2}\E[\mathcal G_T]+\frac2{\eta^2}\sum_{t=1}^T\E[\|\hat b_t\|^2]+\frac{4\sigma}{\eta}\sqrt{\frac{T}{mK}}.
    \label{eq:sg-qT-bd2}
    \end{equation}
    We next bound the two expectation terms on the RHS of~\eqref{eq:sg-qT-bd2}.
    To bound $\sum_{t=1}^T\E[\|\hat b_t\|^2]$, we replace $x_1$ with $x^\star$ in~\eqref{eq:sg-drift-around-initial-point} to obtain
    \begin{align}
        \sum_{t=1}^T\E[\|\hat b_t\|^2]\leq&\sum_{t=1}^T\frac{c_K\Theta^2L^2\tau_t^2}{mq_t^2}\sum_{i=1}^m\E[\left(\|\nabla F_i(x_1)\|+\sigma+L\|x_t-x^\star\|\right)^2]\notag\\
        \leq&\frac{2c_K\Theta^2L^2\widetilde{\mathcal T}_T}{q_1^2}\left[\frac1m\sum_{i=1}^m\bigl(\|\nabla F_i(x^\star)\|+\sigma\bigr)^2+L^2\sqrt{\E[D_T^2]}\right]\notag\\
        \lesssim& 1+\sqrt{\E[D_T^2]}\lesssim \sqrt{\E[\mathcal R_T^2]},\label{eq:cvx-hat_bt_sq-bd}
    \end{align}
where the second inequality holds from $(a+b)^2\leq2a^2+2b^2$ and $q_t\ge q_1$. Moreover, taking expectations on both sides of~\eqref{eq:sg-pathwise-weighted-gradient}, then using~\eqref{eq:cvx-hat_bt_sq-bd} and $\E[\mathcal R_T]\leq \sqrt{\E[\mathcal R_T^2]}$ gives
\begin{equation}
    \E[\mathcal G_T]\leq2(F(x_1)-F^\star)+\frac{1}{q_1}\sum_{t=1}^T\E[\|\hat b_t\|^2]+L\E[\mathcal R_T]\lesssim 1+\sqrt{\E[\mathcal R_T^2]}+\E[\mathcal R_T]\lesssim \sqrt{\E[\mathcal R_T^2]}.
\label{eq:cvx-GT-bd}
\end{equation}
Now substituting~\eqref{eq:cvx-hat_bt_sq-bd} and~\eqref{eq:cvx-GT-bd} into~\eqref{eq:sg-qT-bd2} gives
\begin{equation}\label{eq:cvx-qT-bd-refined}
    \E[q_{T+1}]\lesssim 1+\sqrt{\E[\mathcal R_T^2]}+\sigma\sqrt{\frac{T}{mK}}.
\end{equation}
Plugging~\eqref{eq:cvx-qT-bd-refined} into~\eqref{eq:sg-nonconvex-energy-second-moment} and noting $\mathcal A_T=\eta^2+2R_0^2+2R_1^2\sigma^2(\log(mKT)+2)\lesssim 1+\sigma^2(\log(mKT)+1)$ implies
\begin{align*}
    &\sqrt{\E[\mathcal R_T^2]}\leq2\bigl(\mathcal A_T+4R_1^2\sigma^2\bigr)\left(4+\log\frac{\E[q_{T+1}]}{q_1}\right)\\
    \lesssim&\left(1+\sigma^2(\log(mKT)+1)\right)\Big(1+\log\big(1+\sqrt{\E[\mathcal R_T^2]}+\sigma\sqrt{T/(mK)}\big)\Big).
\end{align*}
Let $y:=\sqrt{\E[\mathcal R_T^2]}$, $s_T:=\sigma\sqrt{T/(mK)}$, and $Q_T:=1+\sigma^2(\log(mKT)+1)$. Then there exists a constant $C>0$ such that
\[
    y\leq CQ_T\bigl(1+\log(1+y+s_T)\bigr).
\]
Since $1+y+s_T\leq(1+y)(1+s_T)$, we have
\[
    1+y\leq CQ_T\log(1+y)+1+CQ_T\bigl(1+\log(1+s_T)\bigr).
\]
Applying Lemma~\ref{lem:log-solution} with $x=1+y$, $a=CQ_T$, and $b=1+CQ_T(1+\log(1+s_T))$ gives
\begin{align*}
    1+y&\leq2CQ_T\log(2CQ_T)-2CQ_T+2+2CQ_T\bigl(1+\log(1+s_T)\bigr)\\
    &=2+2CQ_T\bigl(\log(2CQ_T)+\log(1+s_T)\bigr),
\end{align*}
and hence
\[
    \sqrt{\E[\mathcal R_T^2]}=y\lesssim Q_T\bigl(1+\log Q_T+\log(1+s_T)\bigr)\lesssim Q_T(1+\log\mathcal U_T)=\mathcal C_T.
\]
Finally, substituting this bound into~\eqref{eq:sg-smooth-max-distance-moment} yields the desired result.
\end{proof}

\subsection{Main Result and Proof}
\begin{theorem}[Smooth convex convergence]\label{thm:sg-smooth-convex}
Let Assumptions~\ref{ass:lower-bounded} and~\ref{ass:stochastic-noise} hold. Assume each $F_i$ is $L$-smooth and convex. Let $\{x_t\}_{t\ge 1}$ be the iterates generated by Algorithm~\ref{alg:FL-adagrad} with $\tau_t$ chosen as in~\eqref{eq:tau_t-scheme} and arbitrary $\gamma, \eta>0$. Then for any $T\geq1$, we have
\begin{equation*}
    \E[F(\bar x_T)-F^\star]\lesssim\frac{\mathcal C_T^2}{T}+\frac{\sigma\mathcal C_T}{\sqrt{mKT}},
\end{equation*}
where $\mathcal C_T>0$ is defined in~\eqref{def:mathcal_C_T}.
\end{theorem}
\begin{proof}[Proof of Theorem~\ref{thm:sg-smooth-convex}]
Let $\Delta_t:=F(x_t)-F^\star$. Rearranging~\eqref{eq:d-recursion-1} gives
\begin{equation}\label{eq:sg-smooth-convex-one-step}
    \Delta_t\leq\frac{q_t}{2}\bigl(\|d_t\|^2-\|d_{t+1}\|^2\bigr)+\frac{q_t}{2}\|r_t\|^2+\epsilon_t-\langle\zeta_t,d_t\rangle.
\end{equation}
We proceed by bounding the summation of each term on the RHS of~\eqref{eq:sg-smooth-convex-one-step} in expectation.
Moreover, we have from the monotonicity of $q_t$ that
\[
    \sum_{t=1}^Tq_t\bigl(\|d_t\|^2-\|d_{t+1}\|^2\bigr)=q_1\|d_1\|^2+\sum_{t=2}^T(q_t-q_{t-1})\|d_t\|^2-q_T\|d_{T+1}\|^2\leq q_TD_T,
\]
which together with Cauchy--Schwarz inequality and Lemma~\ref{lem:sg-smooth-convex-moments} implies
\begin{equation}\label{eq:cvx-rate-term1}
    \E\Big[\sum_{t=1}^Tq_t\bigl(\|d_t\|^2-\|d_{t+1}\|^2\bigr)\Big]\leq\sqrt{\E[q_{T+1}^2]}\sqrt{\E[D_T^2]}\lesssim\mathcal C_T\sqrt{\E[q_{T+1}^2]}.
\end{equation}
Since $(\|\bar g_t\|^2/\eta^2)/q_t^2=\|r_t\|^2/\eta^2\leq(R_0+R_1M_T)^2/\eta^2$ from Lemma~\ref{lem:sg-energy}, we obtain from Lemma~\ref{lem:summation} that
\begin{equation}\label{eq:cvx-rate-term2-1}
    \sum_{t=1}^Tq_t\|r_t\|^2=\eta^2\sum_{t=1}^T\frac{\|\bar g_t\|^2/\eta^2}{q_t}\leq\eta^2\left(1+\sqrt{1+(R_0+R_1M_T)^2/\eta^2}\right)(q_{T+1}-q_1).   
\end{equation}
Let $B_T:=1+\sqrt{1+(R_0+R_1M_T)^2/\eta^2}$, then applying $(a+b)^2\leq2(a^2+b^2)$ twice implies
\[
    B_T^2\leq 4+2(R_0+R_1M_T)^2/\eta^2\leq 4+\frac{4R_0^2}{\eta^2}+\frac{4R_1^2M_T^2}{\eta^2}.
\]
Note that $\E[M_T^2]\leq\sigma^2(\log(mKT)+1)$ from~\eqref{eq:sg-noise-envelope}.
Therefore, it follows from $R_0\lesssim 1$ and $R_1\lesssim 1$ that
\[
\sqrt{\E[B_T^2]}\leq\left(4+\frac{4R_0^2}{\eta^2}+\frac{4R_1^2\sigma^2(\log(mKT)+1)}{\eta^2}\right)^{1/2}\lesssim 1+\sigma\sqrt{\log(mKT)+1}\lesssim \mathcal C_T.
\] 
Taking expectations on both sides of~\eqref{eq:cvx-rate-term2-1} and plugging in the above bound gives
\begin{equation}\label{eq:cvx-rate-term2}
    \E\left[\sum_{t=1}^Tq_t\|r_t\|^2\right]\leq\eta^2\sqrt{\E[B_T^2]}\sqrt{\E[(q_{T+1}-q_1)^2]}\lesssim\mathcal C_T\sqrt{\E[q_{T+1}^2]}.
\end{equation}
Moreover, we have from~\eqref{eq:eps-bd}, $q_t\geq q_1$, $\E[\|d_t\|^2]\leq\E[D_T]\leq\sqrt{\E[D_T^2]}$, and~\eqref{eq:quantities-estimates} that
\begin{equation}\label{eq:cvx-rate-term3}
    \sum_{t=1}^T\E[\epsilon_t]\leq\frac{c_K\Theta^2L\widetilde{\mathcal T}_T}{q_1^2}\left[\frac1m\sum_{i=1}^m\bigl(\|\nabla F_i(x^\star)\|+\sigma\bigr)^2+L^2\sqrt{\E[D_T^2]}\right]\lesssim\mathcal C_T,    
\end{equation}
where the last inequality uses Lemma~\ref{lem:sg-smooth-convex-moments}.
Note that the last term in~\eqref{eq:sg-smooth-convex-one-step} has zero expectation since $\E[\langle\zeta_t,d_t\rangle]=\E[\langle\E_t[\zeta_t],d_t\rangle]=0$.
Taking expectations on both sides of~\eqref{eq:sg-smooth-convex-one-step} and plugging in~\eqref{eq:cvx-rate-term1},~\eqref{eq:cvx-rate-term2}, and~\eqref{eq:cvx-rate-term3} gives
\begin{align}
    \sum_{t=1}^T\E[\Delta_t]&\leq\frac12\E\left[\sum_{t=1}^Tq_t(\|d_t\|^2-\|d_{t+1}\|^2)\right]
      +\frac12\E\left[\sum_{t=1}^Tq_t\|r_t\|^2\right]
      +\sum_{t=1}^T\E[\epsilon_t]\notag\\
    &\lesssim \mathcal C_T\sqrt{\E[q_{T+1}^2]}+\mathcal C_T\lesssim \mathcal C_T\sqrt{\E[q_{T+1}^2]}.
    \label{eq:sg-smooth-unweighted-objective-sum}
\end{align}

It remains to control $\E[q_{T+1}^2]$. To this end, we obtain from the update rule for $q_{t+1}$ and~\eqref{eq:gt_sq-exp} that
\begin{align}
    \E[q_{T+1}^2]=&q_1^2+\frac1{\eta^2}\sum_{t=1}^T\E[\|\bar g_t\|^2]\notag\\
    \leq& q_1^2+\frac2{\eta^2}\sum_{t=1}^T\E[\|\nabla F(x_t)\|^2]+\frac2{\eta^2}\sum_{t=1}^T\E[\|\hat b_t\|^2]+\frac{2T\sigma^2}{\eta^2mK}\notag\\
    \leq& q_1^2+\frac{4L}{\eta^2}\sum_{t=1}^T\E[\Delta_t]+\frac2{\eta^2}\sum_{t=1}^T\E[\|\hat b_t\|^2]+\frac{2T\sigma^2}{\eta^2mK},\label{eq:cvx-qT-sq1}
\end{align}
where the last inequality uses the fact that $\|\nabla F(x_t)\|^2\leq 2L\Delta_t$ for smooth convex functions.
To bound $\sum_{t=1}^T\E[\|\hat b_t\|^2]$, we replace $x_1$ with $x^\star$ in~\eqref{eq:sg-drift-around-initial-point} to obtain
\begin{align*}
    \sum_{t=1}^T\E[\|\hat b_t\|^2]\leq&\sum_{t=1}^T\frac{c_K\Theta^2L^2\tau_t^2}{mq_t^2}\sum_{i=1}^m\left(\|\nabla F_i(x_1)\|+\sigma+L\|x_t-x^\star\|\right)^2\notag\\
    \leq&\frac{2c_K\Theta^2L^2\widetilde{\mathcal T}_T}{q_1^2}\left[\frac1m\sum_{i=1}^m\bigl(\|\nabla F_i(x^\star)\|+\sigma\bigr)^2+L^2\sqrt{\E[D_T^2]}\right]\lesssim\mathcal C_T,
\end{align*}
where the second inequality holds from $(a+b)^2\leq2a^2+2b^2$, $q_t\ge q_1$ and the last inequality uses Lemma~\ref{lem:sg-smooth-convex-moments}.
Substituting the above bound and~\eqref{eq:sg-smooth-unweighted-objective-sum} into~\eqref{eq:cvx-qT-sq1} leads to
\[
    \E[q_{T+1}^2]\lesssim 1+\mathcal C_T\sqrt{\E[q_{T+1}^2]}+\mathcal C_T +\frac{2T\sigma^2}{\eta^2mK}\lesssim \mathcal C_T\sqrt{\E[q_{T+1}^2]}+\frac{2T\sigma^2}{\eta^2mK}.
\]
Therefore, there exists a constant $C>0$ such that
\[
\E[q_{T+1}^2]\leq C\mathcal C_T\sqrt{\E[q_{T+1}^2]}+\frac{2T\sigma^2}{\eta^2mK}\leq \frac{\E[q_{T+1}^2]}{2}+\frac{C^2\mathcal C_T^2}{2}+\frac{2T\sigma^2}{\eta^2mK},
\]
which further implies
\[
    \sqrt{\E[q_{T+1}^2]}\leq\sqrt{C^2\mathcal C_T^2+\frac{4T\sigma^2}{\eta^2mK}}\lesssim\mathcal C_T+\sigma\sqrt{\frac{T}{mK}}.
\]
Plugging the above inequality into~\eqref{eq:sg-smooth-unweighted-objective-sum} gives
\[
    \sum_{t=1}^T\E[\Delta_t]\lesssim\mathcal C_T^2+\sigma\mathcal C_T\sqrt{\frac{T}{mK}}.
\]
Then the proof is complete after dividing both sides of the above inequality by $T$.
\end{proof}

\begin{corollary}[Deterministic smooth convex convergence]\label{cor:sg-smooth-convex-deterministic}
Consider the same setting as in Theorem~\ref{thm:sg-smooth-convex}, except that the local oracle is exact, i.e., $\sigma=0$. Then for any $T\geq1$, we have
\[
    \ebb[F(\bar x_T)-F^\star]=\frac1T\sum_{t=1}^T F(x_t)-F^\star\lesssim\frac1T.
\]
\end{corollary}
\begin{proof}
    The result follows immediately from Theorem~\ref{thm:sg-smooth-convex} by plugging in $\mathcal C_T=1$ and $\sigma=0$.
\end{proof}

\section{Lipschitz convex objectives}\label{app:lipschitz-convex}

Assume throughout this subsection that every $F_i$ is convex and $G$-Lipschitz, and write $\nabla F_i(x)\in\partial F_i(x)$ for a measurable subgradient selection satisfying $\|\nabla F_i(x)\|\leq G$.

Note that
\begin{equation}\label{eq:lip-ave_grad}
\Big\|\frac1{mK}\sum_{i=1}^m\sum_{k=0}^{K-1}\nabla F_i(x_{i,k}^t)\Big\|\leq\frac1{mK}\sum_{i=1}^m\sum_{k=0}^{K-1}\|\nabla F_i(x_{i,k}^t)\|\leq G.
\end{equation}
Lemma~\ref{lem:conditional-l2-triangle} and~\eqref{eq:noise-ave-bd} therefore yields
\[
    \sqrt{\E_t[\|\bar g_t\|^2]}\leq \sqrt{\E_t\Big[\Big\|\frac1{mK}\sum_{i=1}^m\sum_{k=0}^{K-1}\nabla F_i(x_{i,k}^t)\Big\|^2\Big]}+\sqrt{\E_t[\|\bar\xi_t\|^2]}\leq G+\frac{\sigma}{\sqrt{mK}},
\]
and equivalently,
\begin{equation}
    \E_t[\|\bar g_t\|^2]\leq\left(G+\frac{\sigma}{\sqrt{mK}}\right)^2.
    \label{eq:sg-lipschitz-second-moment}
\end{equation}

\begin{lemma}\label{lem:sg-lipschitz-local}
Let Assumption~\ref{ass:stochastic-noise} hold. Assume each $F_i$ is convex and $G$-Lipschitz. Let $\epsilon _t$ be defined in~\eqref{eq:eps-def} with $\nabla F_i(x_{i,k}^t)\in\partial F_i(x_{i,k}^t)$. Then, for every $t\geq1$,
\begin{equation*}
    \epsilon_t\leq\frac{(K-1)G\sqrt{G^2+\sigma^2}}{q_t}\tau_t.
\end{equation*}
\end{lemma}
\begin{proof}
We first use $G$-Lipschitz continuity and $\|\nabla F_i(x_{i,k}^t)\|\leq G$ to get
\begin{align}
    F_i(x_t)-F_i(x_{i,k}^t)-\langle \nabla F_i(x_{i,k}^t),x_t-x_{i,k}^t\rangle
    \leq G\|x_t-x_{i,k}^t\|+\|\nabla F_i(x_{i,k}^t)\|\|x_t-x_{i,k}^t\|\leq2G\|x_t-x_{i,k}^t\|.
    \label{eq:F-lip-1}
\end{align}
On the other hand, we have from~\eqref{eq:sg-implied-variance} and unbiasedness of $\xi_{i,k}^t$ that
\[
    \E[\|g_{i,k}^t\|^2]=\E[\|\nabla F_i(x_{i,k}^t)\|^2]+2\E[\langle \nabla F_i(x_{i,k}^t),\E[\xi_{i,k}^t\mid\mathcal F_{t,k}^{(i)}]\rangle]+\E[\|\xi_{i,k}^t\|^2]\leq G^2+\sigma^2.
\]
Recall that $x_{i,k}^t-x_t=-(\tau_t/q_t)\sum_{s=0}^{k-1}g_{i,s}^t$, we thus obtain
\[
    \E_t[\|x_{i,k}^t-x_t\|]\leq\frac{\tau_t}{q_t}\sum_{s=0}^{k-1}\E_t[\|g_{i,s}^t\|]\leq\frac{\tau_t}{q_t}\sum_{s=0}^{k-1}\sqrt{\E_t[\|g_{i,s}^t\|^2]}\leq\frac{k\tau_t}{q_t}\sqrt{G^2+\sigma^2},
\]
which together with~\eqref{eq:F-lip-1} gives
\[
    \epsilon_t\leq\frac{2G}{mK}\sum_{i=1}^m\sum_{k=0}^{K-1}\E_t[\|x_{i,k}^t-x_t\|]
    \leq\frac{2G\tau_t\sqrt{G^2+\sigma^2}}{q_tK}\sum_{k=0}^{K-1}k
    =\frac{(K-1)G\sqrt{G^2+\sigma^2}}{q_t}\tau_t.
\]
The proof is complete.
\end{proof}

\begin{lemma}\label{lem:sg-lipschitz-energy}
Let Assumption~\ref{ass:stochastic-noise} hold. Assume each $F_i$ is convex and $G$-Lipschitz. Then for every $T\geq1$, we have $\max_{1\leq t\leq T}\|r_t\|\leq (G+M_T)/q_1$ almost surely, and
\begin{align}
    &\E[\mathcal R_T]\leq2\left(\eta^2+\frac{2G^2}{q_1^2}+\frac{2\sigma^2(\log(mKT)+2)}{q_1^2}\right)\left(1+\log\frac{\E[q_{T+1}]}{q_1}\right),\label{eq:lip-RT-bd}\\
    &\sqrt{\E[\mathcal R_T^2]}\leq 2\left(\eta^2+\frac{2G^2}{q_1^2}+\frac{2\sigma^2(\log(mKT)+4)}{q_1^2}\right)\left(4+\log\frac{\E[q_{T+1}]}{q_1}\right).\label{eq:lip-RT_sq-bd}
\end{align}
\end{lemma}
\begin{proof}
Firstly, it follows from~\eqref{eq:lip-ave_grad} and $\|\bar\xi_t\|\leq M_T$ that
\[
    \|r_t\|=\frac{\|\bar g_t\|}{q_t}\leq\frac{G+M_T}{q_t}\leq\frac{G+M_T}{q_1},~\forall 1\leq t\leq T,
\]
which proves $\max_{1\leq t\leq T}\|r_t\|\leq (G+M_T)/q_1$ almost surely.

We proceed to bound $\E[\mathcal R_T]$. Since $(\|\bar g_t\|^2/\eta^2)/q_t^2=\|r_t\|^2/\eta^2\leq(G+M_T)^2/(\eta^2q_1^2)$, we apply Lemma~\ref{lem:summation} with $A_t=q_t^2$, $a_t=\|\bar g_t\|^2/\eta^2$, and $c=(G+M_T)^2/(\eta^2q_1^2)$ to get
\begin{equation}\label{eq:RT-bd1}
    \mathcal R_T=\sum_{t=1}^{T}\|r_t\|^2=\eta^2\sum_{t=1}^T\frac{\|\bar g_t\|^2/\eta^2}{q_t^2}\leq2\left(\eta^2+\frac{(G+M_T)^2}{q_1^2}\right)\log\frac{q_{T+1}}{q_1}.
\end{equation}
Taking expectations on both sides of~\eqref{eq:RT-bd1} and applying Cauchy--Schwarz inequality further gives
\begin{align*}
    \E[\mathcal R_T]\leq& 2\sqrt{\E\Big[\Big(\eta^2+\frac{2G^2}{q_1^2}+\frac{2M_T^2}{q_1^2}\Big)^2\Big]}\sqrt{\E\Big[\log^2\frac{q_{T+1}}{q_1}\Big]}\\
    \leq& 2\left(\eta^2+\frac{2G^2}{q_1^2}+\frac{2\sqrt{\E[M_T^4]}}{q_1^2}\right)\sqrt{\E\left[\left(1+\log\frac{q_{T+1}}{q_1}\right)^2\right]}\\
    \leq& 2\left(\eta^2+\frac{2G^2}{q_1^2}+\frac{2\sigma^2(\log(mKT)+2)}{q_1^2}\right)\left(1+\log\frac{\E[q_{T+1}]}{q_1}\right),
\end{align*}
where the second inequality uses Lemma~\ref{lem:conditional-l2-triangle} and $q_{T+1}\geq q_1$, and the last one uses Lemma~\ref{lem:log-power-concavity} with $a=1$ and $p=2$ and $\sqrt{\E[M_T^4]}\leq\sigma^2(\log(mKT)+2)$ from~\eqref{eq:sg-noise-envelope}. This proves~\eqref{eq:lip-RT-bd}.

Using similar techniques, we square both sides of~\eqref{eq:RT-bd1} and take expectations to get
\begin{align*}
    \sqrt{\E[\mathcal R_T^2]}\leq&2\left(\E\Big[\Big(\eta^2+\frac{(G+M_T)^2}{q_1^2}\Big)^2\log^2\frac{q_{T+1}}{q_1}\Big]\right)^{1/2}\\
    \leq& 2\left(\E\Big[\Big(\eta^2+\frac{2G^2}{q_1^2}+\frac{2M_T^2}{q_1^2}\Big)^4\Big]\right)^{1/4}\left(\E\Big[\log^4\frac{q_{T+1}}{q_1}\Big]\right)^{1/4}\\
    \leq& 2\left(\eta^2+\frac{2G^2}{q_1^2}+\frac{2(\E[M_T^8])^{1/4}}{q_1^2}\right)\left(\E\Big[\left(4+\log\frac{q_{T+1}}{q_1}\right)^4\Big]\right)^{1/4}\\
    \leq& 2\left(\eta^2+\frac{2G^2}{q_1^2}+\frac{2\sigma^2(\log(mKT)+4)}{q_1^2}\right)\left(4+\log\frac{\E[q_{T+1}]}{q_1}\right),
\end{align*}
where the last inequality uses Lemma~\ref{lem:log-power-concavity} with $a=4$ and $p=4$ and $(\E[M_T^8])^{1/4}\leq\sigma^2(\log(mKT)+4)$ from~\eqref{eq:sg-noise-envelope}. The proof is complete.
\end{proof}

\[
    \mathcal V_T:=\left(1+\sigma^2(\log(mKT)+1)\right)\left(1+\log\big(\sqrt{T}+\sigma\sqrt{T/(mK)}\big)\right)+(1+\sigma)\log T.
\]
\begin{lemma}
\label{lem:sg-lipschitz-moments}
Let Assumptions~\ref{ass:lower-bounded} and~\ref{ass:stochastic-noise} hold. Assume each $F_i$ is convex and $G$-Lipschitz. If $\mathcal T_T\lesssim \log(T)/(K-1)$, then for every $T\geq1$,
\begin{equation*}
    \sqrt{\E[\mathcal R_T^2]}\lesssim\mathcal V_T,\qquad\sqrt{\E[D_T^2]}\lesssim\mathcal V_T,
\end{equation*}
where $\mathcal V_T>0$ is defined in~\eqref{def:mathcal_V_T}.
\end{lemma}
\begin{proof}
We first bound $\E[q_{T+1}]$. Note from the update rule and~\eqref{eq:sg-lipschitz-second-moment} that
\begin{equation}
    \E[q_{T+1}^2]=q_1^2+\frac1{\eta^2}\sum_{t=1}^T\E[\|\bar g_t\|^2]\leq q_1^2+\frac{T}{\eta^2}\left(G+\frac{\sigma}{\sqrt{mK}}\right)^2\lesssim T\left(1+\frac{\sigma}{\sqrt{mK}}\right)^2,
\label{eq:sg-lipschitz-accumulator-second-moment}
\end{equation}
which implies
\[
    \E[q_{T+1}]\leq\sqrt{\E[q_{T+1}^2]}\lesssim \sqrt{T}\left(1+\frac{\sigma}{\sqrt{mK}}\right)= \sqrt{T}+\sigma\sqrt{\frac{T}{mK}}.
\]
Plugging the above inequality into~\eqref{eq:lip-RT_sq-bd} gives
\[
    \sqrt{\E[\mathcal R_T^2]}\lesssim \left(1+\sigma^2(\log(mKT)+1)\right)\left(1+\log\big(\sqrt{T}+\sigma\sqrt{T/(mK)}\big)\right)\lesssim\mathcal V_T.
\]

To bound $\E[D_T^2]$, we start from~\eqref{eq:d-recursion-1} to get
\begin{equation}
    \|d_{t+1}\|^2+\frac{2(F(x_t)-F^\star)}{q_t}
    \leq\|d_t\|^2+\|r_t\|^2-\frac2{q_t}\langle\zeta_t,d_t\rangle+\frac{2\epsilon_t}{q_t}.
\label{eq:sg-lipschitz-distance-recursion-base}
\end{equation}
From Lemma~\ref{lem:sg-lipschitz-local} and $q_t\geq q_1$, we obtain
\begin{equation}\label{eq:lip-beta-def}
    \frac{2\epsilon_t}{q_t}\leq\frac{2(K-1)G\sqrt{G^2+\sigma^2}}{q_t^2}\tau_t
    \leq\frac{2(K-1)G\sqrt{G^2+\sigma^2}}{q_1^2}\tau_t=:\tilde{\beta}_t.
\end{equation}

For $1\leq t\leq T$, let $\widetilde{\mathcal{M}}_t:=-2\sum_{s=1}^t\langle\zeta_s,d_s\rangle/q_s$.
Since $\E_t[\zeta_t]=0$, we have
\[
    \E_t[\widetilde{\mathcal{M}}_t]=\widetilde{\mathcal{M}}_{t-1}-\frac{2}{q_t}\bigl\langle\E_t[\zeta_t],d_t\bigr\rangle=\widetilde{\mathcal{M}}_{t-1},
\]
which shows that $\{\widetilde{\mathcal{M}}_t\}_{t=1}^T$ is a martingale.

Applying\eqref{eq:lip-beta-def} on~\eqref{eq:sg-lipschitz-distance-recursion-base} and unrolling yields
\[
    \|d_{t+1}\|^2\leq\|d_1\|^2+\sum_{s=1}^t(\|r_s\|^2+\tilde{\beta}_s)+\widetilde{\mathcal{M}}_t,
\]
and hence
\begin{equation}\label{eq:lip-DT-bd1}
    D_T\leq\|d_1\|^2+\mathcal R_T+\sum_{t=1}^T\tilde{\beta}_t+\max_{1\leq s\leq T}|\widetilde{\mathcal{M}}_s|.
\end{equation}
Using similar arguments as~\eqref{eq:mcal-bd} in the proof of Lemma~\ref{lem:sg-smooth-convex-moments}, we have
\[
    \E\left[\max_{1\leq s\leq T}|\widetilde{\mathcal{M}}_s|^2\right]\leq 16\sqrt{\E[D_T^2]}\sqrt{\E[\mathcal R_T^2]}.
\]
Using Lemma~\ref{lem:conditional-l2-triangle} and the above inequality, then it follows from~\eqref{eq:lip-DT-bd1} that
\begin{align*}
    \sqrt{\E[D_T^2]}
    &\leq\|d_1\|^2+\sum_{t=1}^T\tilde{\beta}_t+\sqrt{\E[\mathcal R_T^2]}+\sqrt{\E\left[\max_{1\leq s\leq T}|\widetilde{\mathcal{M}}_s|^2\right]}\\
    &\leq\|d_1\|^2+\sum_{t=1}^T\tilde{\beta}_t+\sqrt{\E[\mathcal R_T^2]}
      +4\sqrt{\sqrt{\E[D_T^2]}\sqrt{\E[\mathcal R_T^2]}}\\
    &\leq\|d_1\|^2+\sum_{t=1}^T\tilde{\beta}_t+\frac12\sqrt{\E[D_T^2]}+9\sqrt{\E[\mathcal R_T^2]},
\end{align*}
where the last inequality uses Young's inequality. Simplifying the above inequality and noting~\eqref{eq:lip-beta-def} gives
\begin{align*}
    \sqrt{\E[D_T^2]}\leq&2\left(\|d_1\|^2+\frac{2(K-1)G(G+\sigma)}{q_1^2}\mathcal T_T\right)+18\sqrt{\E[\mathcal R_T^2]}\\
    \lesssim&1+(1+\sigma)\log T+ \mathcal V_T\lesssim\mathcal V_T,
\end{align*}
where the second inequality uses $\mathcal T_T\lesssim \log(T)/(K-1)$. The proof is complete.
\end{proof}

\subsection{Main Result and Proof}
\begin{theorem}[Convex Lipschitz convergence]\label{thm:sg-lipschitz-convex}
Let Assumptions~\ref{ass:lower-bounded} and~\ref{ass:stochastic-noise} hold. Assume each $F_i$ is convex and $G$-Lipschitz. Let $\{x_t\}_{t\ge 1}$ be the iterates generated by Algorithm~\ref{alg:FL-adagrad} with $\tau_t$ chosen as in~\eqref{eq:tau_t-scheme} and arbitrary $\gamma,\eta>0$. Then for any $T\geq1$, we have
\[
    \E[F(\bar x_T)-F^\star]\lesssim \Big(\frac{1}{\sqrt{T}}+\frac{\sigma}{\sqrt{mKT}}\Big)\mathcal V_T,
\]
where $\mathcal V_T>0$ is defined in~\eqref{def:mathcal_V_T}.
\end{theorem}
\begin{proof}[Proof of Theorem~\ref{thm:sg-lipschitz-convex}]
Let $\Delta_t:=F(x_t)-F^\star$. Rearranging~\eqref{eq:sg-lipschitz-distance-recursion-base} gives
\begin{equation}
    \Delta_t\leq\frac{q_t}{2}\bigl(\|d_t\|^2-\|d_{t+1}\|^2\bigr)+\frac{q_t}{2}\|r_t\|^2+\epsilon_t-\langle\zeta_t,d_t\rangle.
\label{eq:sg-lipschitz-convex-one-step}
\end{equation}
We bound the sum of each term on the RHS of~\eqref{eq:sg-lipschitz-convex-one-step} in expectation. First, monotonicity of $q_t$ gives
\begin{equation}
    \sum_{t=1}^Tq_t\bigl(\|d_t\|^2-\|d_{t+1}\|^2\bigr)=q_1\|d_1\|^2+\sum_{t=2}^T(q_t-q_{t-1})\|d_t\|^2-q_T\|d_{T+1}\|^2\leq q_TD_T.
\label{eq:sg-lipschitz-convex-summation-by-parts}
\end{equation}
Therefore, Cauchy--Schwarz and Lemma~\ref{lem:sg-lipschitz-moments} yield
\begin{equation}
    \E\left[\sum_{t=1}^Tq_t\bigl(\|d_t\|^2-\|d_{t+1}\|^2\bigr)\right]\leq\E[q_{T+1}D_T]\leq\sqrt{\E[q_{T+1}^2]}\sqrt{\E[D_T^2]}\lesssim\mathcal V_T\sqrt{\E[q_{T+1}^2]}.
\label{eq:sg-lipschitz-convex-term1}
\end{equation}
On the other hand, since $(\|\bar g_t\|^2/\eta^2)/q_t^2=\|r_t\|^2/\eta^2\leq(G+M_T)^2/(\eta^2q_1^2)$ from Lemma~\ref{lem:sg-lipschitz-energy}, we apply Lemma~\ref{lem:summation} with $A_t=q_t^2$, $a_t=\|\bar g_t\|^2/\eta^2$, and $c=(G+M_T)^2/(\eta^2q_1^2)$ to get
\begin{equation}
    \sum_{t=1}^Tq_t\|r_t\|^2\leq\eta^2\left(1+\sqrt{1+\frac{(G+M_T)^2}{\eta^2q_1^2}}\right)(q_{T+1}-q_1).
\label{eq:sg-lipschitz-convex-weighted-energy}
\end{equation}
Denote $B_T:=\left(1+\sqrt{1+(G+M_T)^2/(\eta^2q_1^2)}\right)$. Then it follows from $(a+b)^2\leq 2a^2+2b^2$ and~\eqref{eq:sg-noise-envelope} with $p=1$ that
\begin{align*}
    \sqrt{\E[B_T^2]}\leq& \sqrt{4+\frac{4G^2}{\eta^2q_1^2}+\frac{4\E[M_T^2]}{\eta^2q_1^2}}\leq\left(4+\frac{4G^2}{\eta^2q_1^2}+\frac{4\sigma^2(\log(mKT)+1)}{\eta^2q_1^2}\right)^{1/2}\\
    \lesssim& \sqrt{1+\sigma^2(\log(mKT)+1)}\leq 1+\sigma\sqrt{\log(mKT)+1}\lesssim\mathcal V_T.
\end{align*}
Taking expectations on~\eqref{eq:sg-lipschitz-convex-weighted-energy}, we obtain from the above inequality and applying Cauchy--Schwarz inequality that
\begin{equation}
    \E\left[\sum_{t=1}^Tq_t\|r_t\|^2\right] \leq\eta^2\sqrt{\E[B_T^2]}\sqrt{\E[(q_{T+1}-q_1)^2]}\lesssim\mathcal V_T\sqrt{\E[q_{T+1}^2]}.
\label{eq:sg-lipschitz-convex-term2}
\end{equation}
Additionally, Lemma~\ref{lem:sg-lipschitz-local} and $\mathcal T_t\lesssim \log t/(K-1)$ gives
\begin{equation}
    \sum_{t=1}^T\E[\epsilon_t]\leq\frac{(K-1)G\sqrt{G^2+\sigma^2}}{q_1}\mathcal T_T\lesssim (1+\sigma)\log T\lesssim\mathcal V_T.
\label{eq:sg-lipschitz-convex-term3}
\end{equation}
Note that the last term in~\eqref{eq:sg-lipschitz-convex-one-step} has zero expectation since $\E[\langle\zeta_t,d_t\rangle]=\E[\langle\E_t[\zeta_t],d_t\rangle]=0$.
Now summing~\eqref{eq:sg-lipschitz-convex-one-step} over $t=1,2,\ldots,T$, taking expectations, and applying~\eqref{eq:sg-lipschitz-convex-term1},~\eqref{eq:sg-lipschitz-convex-term2},~\eqref{eq:sg-lipschitz-convex-term3} yield
\begin{align}
    \sum_{t=1}^T\E[\Delta_t]&\leq\frac12\E\left[\sum_{t=1}^Tq_t(\|d_t\|^2-\|d_{t+1}\|^2)\right]
      +\frac12\E\left[\sum_{t=1}^Tq_t\|r_t\|^2\right]
      +\sum_{t=1}^T\E[\epsilon_t]\notag\\
    &\lesssim\mathcal V_T\left(1+\sqrt{\E[q_{T+1}^2]}\right)
    \lesssim\mathcal V_T\sqrt{\E[q_{T+1}^2]}.\label{eq:sg-lipschitz-convex-objective-sum}
\end{align}
Finally, plugging the bound on $\sqrt{\E[q_{T+1}^2]}$~\eqref{eq:sg-lipschitz-accumulator-second-moment} into~\eqref{eq:sg-lipschitz-convex-objective-sum} leads to
\[
    \sum_{t=1}^T\E[\Delta_t]
    \lesssim\sqrt{T}\left(1+\frac{\sigma}{\sqrt{mK}}\right)\mathcal V_T.
\]
Then the proof is complete after dividing both sides of the above inequality by $T$.
\end{proof}

\begin{corollary}[Deterministic Lipschitz convex convergence]\label{cor:sg-lipschitz-deterministic}
Consider the same setting as in Theorem~\ref{thm:sg-lipschitz-convex}, except that the local oracle is exact, i.e., $\sigma=0$. Then, for every $T\geq1$,
\[
    \E[F(\bar x_T)-F^\star]=\frac1T\sum_{t=1}^T F(x_t)-F^\star\lesssim\frac{\log(1+T)}{\sqrt T}.
\]
\end{corollary}
\begin{proof}
    When $\sigma=0$, we obtain from the definition of $\mathcal V_T$~\eqref{def:mathcal_V_T} that
    \[
        \mathcal V_T=1+\log\sqrt{T}+\log T\lesssim \log(1+T).
    \]
    Then the result follows immediately from Theorem~\ref{thm:sg-lipschitz-convex} by plugging in $\sigma=0$.
\end{proof}

\section{Extension I: Partial Participation}\label{app:partial-participation}
\begin{algorithm}
    \caption{Adaptive Federated Learning Method (AdaFed) with Partial Participation}\label{alg:FL-adagrad-partial}
    \begin{algorithmic}
    \REQUIRE{} initial model $x_1\in\rbb^d$; local-step count $K\geq1$; server iterations $T\geq1$; inclusion probabilities $0<p_i\leq1$; prescribed local scales $\tau_t>0$; parameters $\gamma,\eta>0$.
    \STATE{} Initialize the step size $\eta_1:=\eta/\gamma$.
    \FOR{$t = 1, 2, \dots, T$}
    \STATE{} Sample active clients $\scal_t\subset\{1,\ldots,m\}$ independently with probabilities $p_1,\ldots,p_m$.
    \STATE{} \textbf{Server} sends $(x_t,\eta_t,\tau_t)$ to every \textbf{client} $i\in\scal_t$.
    \FOR{each \textbf{client} $i\in\scal_t$ \textbf{in parallel}}
    \STATE{} Initialize $x_{i,0}^t:=x_t$.
    \FOR{$k=0,\dots,K-1$}
    \STATE{} Update $x_{i,k+1}^t:=x_{i,k}^t-\tau_t\eta_t g_{i,k}^t$ with $g_{i,k}^t\approx\nabla F_i(x_{i,k}^t)$.
    \ENDFOR{}
    \STATE{} Set $\Delta_i^t:=x_{i,K}^t-x_t$ and send $\Delta_i^t$ to the server.
    \ENDFOR{}
    \STATE{} \textbf{Server} updates the model and the step size:
    \begin{align*}
        x_{t+1}:=x_t+\frac1{mK\tau_t}\sum_{i\in\scal_t}\frac{\Delta_i^t}{p_i},\qquad \frac{1}{\eta_{t+1}^2}:=\frac{1}{\eta_t^2}\left(1+\frac{\|x_{t+1}-x_t\|^2}{\eta^2}\right).
    \end{align*}
    \ENDFOR{}
    \STATE{} \textbf{Output:} $\bar{x}_T$ chosen uniformly at random from $\{x_t\}_{t=1}^T$.
    \end{algorithmic}
\end{algorithm}

Conditionally on $\mathcal F_t$, the inclusion indicators $\chi_t^{(i)}:=\mathbf1\{i\in\mathcal S_t\}$ are independent with $\E_t[\chi_t^{(i)}]=p_i\in(0,1]$. 
Since $\Delta_i^t=-K\tau_t\bar g_t^{(i)}/q_t$, the server update in Algorithm~\ref{alg:FL-adagrad-partial} can be rewritten as
\begin{equation}
    \tag{Core-PP}
    \widehat g_t:=\frac1m\sum_{i=1}^m\frac{\chi_t^{(i)}}{p_i}\bar g_t^{(i)},\qquad x_{t+1}=x_t-\frac{\widehat g_t}{q_t},\qquad q_{t+1}^2=q_t^2+\frac{\|\widehat g_t\|^2}{\eta^2}.
    \label{eq:pp-method}
\end{equation}
\stepcounter{equation}
Note that scaling $\bar g_t^{(i)}$ by $1/p_i$ for each active client removes the bias arising from unequal participation probabilities. Indeed, due to the independence of the inclusion indicators conditionally on $\mathcal F_t$, we have
\[
    \E_t[\widehat g_t]=\frac1m\sum_{i=1}^m\E_t\left[\frac{\chi_t^{(i)}}{p_i}\right]\E_t[\bar g_t^{(i)}]=\frac1m\sum_{i=1}^m\E_t[\bar g_t^{(i)}]=\E_t[\bar g_t].
\]

\subsection{Analysis Setup}
For simplicity, we analyze Algorithm~\ref{alg:FL-adagrad-partial} with exact local gradients $g_{i,k}^t=\nabla F_i(x_{i,k}^t)$ and $\tau_t$ chosen as in~\eqref{eq:partial-tau}.
Thus $\kappa_\tau:=\sup_{t\geq2}(t-1)\tau_t<\infty$ and $\widetilde{\mathcal T}_T<\infty$.
For analysis, define the counterfactual local path for every client (including inactive ones) by
\begin{equation}
    x_{i,0}^t=x_t,\qquad x_{i,k+1}^t=x_{i,k}^t-\frac{\tau_t}{q_t}g_{i,k}^t,\qquad \bar g_t^{(i)}:=\frac1K\sum_{k=0}^{K-1}g_{i,k}^t.
    \label{eq:pp-local}
\end{equation}
Recall $g_t=\nabla F(x_t)$, $b_t^{(i)}=\bar g_t^{(i)}-\nabla F_i(x_t)$, and $b_t:=\frac1m\sum_{i=1}^m b_t^{(i)}=\bar g_t-g_t$ from full-participation analysis. We also define the participation noise via
\[
    \nu_t:=\widehat g_t-\bar g_t=\frac1m\sum_{i=1}^m\left(\frac{\chi_t^{(i)}}{p_i}-1\right)\bar g_t^{(i)}.
\]
We first bound the conditional variance of $\widehat g_t$ in the following lemma.
\begin{lemma}\label{lem:participation-moments}
For all $t\geq 1$, we have $\E_t[\nu_t]=0$ and
\begin{equation*}
    \E_t[\|\nu_t\|^2]=\frac1{m^2}\sum_{i=1}^m\frac{1-p_i}{p_i}\|\bar g_t^{(i)}\|^2.
\end{equation*}
\end{lemma}
\begin{proof}
Conditionally on $\mathcal F_t$, every $\bar g_t^{(i)}$ is fixed. Since $\E_t[\chi_t^{(i)}]=p_i$, we have
\[
    \E_t[\nu_t]=\E_t\left[\frac1m\sum_{i=1}^m\left(\frac{\chi_t^{(i)}}{p_i}-1\right)\bar g_t^{(i)}\right]=\frac1m\sum_{i=1}^m\left(\frac{\E_t[\chi_t^{(i)}]}{p_i}-1\right)\bar g_t^{(i)}=0.
\]
Moreover, note that for any $i\in[m]$,
\[
    \E_t\left[\left(\frac{\chi_t^{(i)}}{p_i}-1\right)^2\right]=p_i\left(\frac1{p_i}-1\right)^2+(1-p_i)=\frac{1-p_i}{p_i}.
\]
For $i\neq j$, conditional independence  gives
\[
    \E_t\left[ \left(\frac{\chi_t^{(i)}}{p_i}-1\right) \left(\frac{\chi_t^{(j)}}{p_j}-1\right) \right]=0.
\]
Expanding the squared norm therefore yields
\begin{align*}
    \E_t[\|\nu_t\|^2] =\frac1{m^2}\sum_{i=1}^m\sum_{j=1}^m \E_t\left[ \left(\frac{\chi_t^{(i)}}{p_i}-1\right) \left(\frac{\chi_t^{(j)}}{p_j}-1\right) \right] \left\langle\bar g_t^{(i)},\bar g_t^{(j)}\right\rangle=\frac1{m^2}\sum_{i=1}^m \frac{1-p_i}{p_i}\|\bar g_t^{(i)}\|^2,
\end{align*}
which completes the proof.
\end{proof}
Using Lemma~\ref{lem:participation-moments}, we have from $\|\widehat g_t\|^2=\|\bar g_t+\nu_t\|^2$ and $\E_t[\nu_t]=0$ that
\begin{equation}
    \label{eq:hat_gt_bd}
    \E_t[\|\widehat g_t\|^2]=\|\bar g_t\|^2+\frac1{m^2}\sum_{i=1}^m\frac{1-p_i}{p_i}\|\bar g_t^{(i)}\|^2.
\end{equation}

As indicated in Lemma~\ref{lem:participation-moments}, partial participation creates conditional sampling noise even though every active client uses exact gradients. In particular, the sampling variance can remain positive at a stationary point of $F$ when the individual client gradients do not vanish there. The accumulator $q_t$ can therefore grow continuously with $t$, unlike in the full-participation case.
To control the magnitude of the participation variance, we introduce the following bounded-heterogeneity condition. 
\begin{assumption}[Bounded heterogeneity]\label{ass:pp-heterogeneity}
There exists a constant $\delta\geq0$ such that
\[
\frac1m\sum_{i=1}^m\|\nabla F_i(x)-\nabla F(x)\|^2\leq\delta^2,\qquad x\in\mathbb R^d.
\]
\end{assumption}
Note that Assumption~\ref{ass:pp-heterogeneity} is only required for the analysis of smooth objectives. It is not needed for the Lipschitz convex result.
For ease of notation, we define
\begin{equation*}
    A_p:=2+\frac2m\max_{i\in[m]}\frac{1-p_i}{p_i},\qquad \sigma_p^2:=\frac{2\delta^2}{m}\max_{i\in[m]}\frac{1-p_i}{p_i}.
\end{equation*}
Note that $\sigma_p=0$ under full participation or zero heterogeneity.

\begin{lemma}\label{lem:pp-displacement}
Let Assumption~\ref{ass:pp-heterogeneity} hold and assume each $F_i$ is $L$-smooth. Then there exist nonnegative constants $M_p$, $R_p$, and $B_p$ indenpendent of $t$, such that for all $t\geq1$,
\begin{equation*}
    \frac{\|g_t\|}{q_t}\leq M_p,\qquad \|r_t\|\leq R_p,\qquad \frac1m\sum_{i=1}^m\|b_t^{(i)}\|^2\leq B_p^2\tau_t^2,\qquad \|b_t\|\leq B_p\tau_t.
\end{equation*}
\end{lemma}
\begin{proof}
Note that the counterfactual path~\eqref{eq:pp-local} overlaps with the full-participation local path with exact local gradients. Hence~\eqref{eq:sqrt-xk-xt} holds clientwise with $\sigma=0$. Averaging the distance over $k$ and using smoothness gives
\begin{equation}
    \|x_{i,k}^t-x_t\|\leq\frac{k\tau_t\Theta}{q_t}\|\nabla F_i(x_t)\|,\qquad
    \|\bar g_t^{(i)}\|\leq\Theta\|\nabla F_i(x_t)\|,\qquad
    \|b_t^{(i)}\|\leq\frac{(K-1)L\Theta\tau_t}{2q_t}\|\nabla F_i(x_t)\|.
    \label{eq:pp-local-error}
\end{equation}
On the other hand, Assumption~\ref{ass:pp-heterogeneity} implies
\begin{equation}
    \|\nabla F_i(x_t)-g_t\|\leq\sqrt m\,\delta,\qquad
    \frac1m\sum_{i=1}^m\|\nabla F_i(x_t)\|^2\leq\|g_t\|^2+\delta^2.
    \label{eq:pp-gradient-heterogeneity}
\end{equation}
Consequently, ~\eqref{eq:pp-local-error} and~\eqref{eq:pp-gradient-heterogeneity} yield
\begin{equation}
    \|\widehat g_t\|\leq\frac1m\sum_{i=1}^m\frac{1}{p_i}\|\bar g_t^{(i)}\|\leq\frac{\Theta}{m}\sum_{i=1}^m\frac1{p_i}(\|g_t\|+\sqrt m\,\delta).
    \label{eq:pp-affine-pathwise}
\end{equation}

We first show that $\|g_t\|/q_t$ is bounded. Set $n_0:=\max\{2,1+\lceil(K-1)L\Theta\kappa_\tau/q_1\rceil\}$ and $\beta_t:=(K-1)L\Theta\tau_t/(2q_t)$. Since $q_t\geq q_1$ and $\tau_t\leq\kappa_\tau/(t-1)$ for $t\geq2$, we have for any $t\geq n_0$ that
\[
\beta_t\leq\frac{(K-1)L\Theta}{2q_t}\frac{\kappa_\tau}{\lceil(K-1)L\Theta\kappa_\tau/q_1\rceil}\leq1/2.
\]
We focus on nonempty rounds, since otherwise $x_{t+1}=x_t$ and hence $g_{t+1}=g_t$ and $q_{t+1}=q_t$. For a nonempty round, we obtain from~\eqref{eq:pp-local-error} and~\eqref{eq:pp-gradient-heterogeneity} that
\[
    \|\bar g_t^{(i)}-g_t\|=\|b_t^{(i)}+\nabla F_i(x_t)-g_t\|\leq\beta_t\|g_t\|+(1+\beta_t)\sqrt m\,\delta.
\]
Let $\omega_t:=m^{-1}\sum_{i\in\mathcal S_t}p_i^{-1}\geq1/m$. Then it follows from the above inequality that
\begin{align*}
    \|\widehat g_t\|
    &=\left\|\omega_tg_t+\frac1m\sum_{i\in\mathcal S_t}\frac{\bar g_t^{(i)}-g_t}{p_i}\right\|\\
    &\geq\omega_t\|g_t\|-\frac1m\sum_{i\in\mathcal S_t}\frac{\|\bar g_t^{(i)}-g_t\|}{p_i}\\
    &\geq\omega_t\bigl((1-\beta_t)\|g_t\|-(1+\beta_t)\sqrt m\,\delta\bigr).
\end{align*}
For $t\geq n_0$, using $\beta_t\leq1/2$ and $\omega_t^{-1}\leq m$ then gives
\begin{equation}
    \|g_t\|\leq\frac{2}{\omega_t}\|\widehat g_t\|+3\sqrt m\,\delta
    \leq2m\|\widehat g_t\|+3\sqrt m\,\delta.
    \label{eq:pp-gradient-from-sample}
\end{equation}
To bound $\|g_{t+1}\|$, we use smoothness,~\eqref{eq:pp-method}, and plugging in~\eqref{eq:pp-gradient-from-sample} to get 
\[
    \|g_{t+1}\|\leq\|g_t\|+L\|x_{t+1}-x_t\|=\|g_t\|+\frac{L}{q_t}\|\widehat g_t\|
    \leq\left(2m+\frac{L}{q_t}\right)\|\widehat g_t\|+3\sqrt m\,\delta.
\]
Note that $q_{t+1}\geq q_1$ and $q_{t+1}\geq\|\widehat g_t\|/\eta$, dividing both sides of the above inequality by $q_{t+1}$ yields
\begin{equation*}
    \frac{\|g_{t+1}\|}{q_{t+1}}
    \leq\frac{(2m+L/q_t)\|\widehat g_t\|+3\sqrt m\,\delta}{q_{t+1}}
    \leq\eta(2m+L/q_1)+3\sqrt m\,\delta/q_1.
\end{equation*}
For $t<n_0$, combining~\eqref{eq:pp-affine-pathwise} and $\|g_{t+1}\|\leq \|g_t\|+L\|\widehat g_t\|/q_t$ implies
\begin{equation}
    \|g_{t+1}\|+\sqrt m\,\delta\leq\left(1+\frac{L\Theta}{mq_1}\sum_{i=1}^m\frac1{p_i}\right)(\|g_t\|+\sqrt m\,\delta).
    \label{eq:pp-prefix-gradient-recursion}
\end{equation}
Iterating~\eqref{eq:pp-prefix-gradient-recursion} for at most $n_0-1$ rounds and using $q_t\geq q_1$ then bounds $\|g_{t+1}\|/q_{t+1}$ for $t<n_0$. 
Now define the nonnegative constants $M_p$, $R_p$, and $B_p$ via:
\begin{align*}
    M_p&:=\max\left\{\Big(1+\frac{L\Theta}{mq_1}\sum_{i=1}^m\frac1{p_i}\Big)^{n_0-1}\frac{\|g_1\|+\sqrt m\,\delta}{q_1},\ \eta(2m+L/q_1)+\frac{3\sqrt m\,\delta}{q_1}\right\},\\
    R_p&:=\frac{\Theta}{m}\sum_{i=1}^m\frac1{p_i}\left(M_p+\frac{\sqrt m\,\delta}{q_1}\right),\qquad
    B_p:=\frac{(K-1)L\Theta}{2}\sqrt{M_p^2+\delta^2/q_1^2}.
\end{align*}
It then follows that $\|g_t\|/q_t\leq M_p$ for all $t\geq1$.

To verify the remaining bounds, divide~\eqref{eq:pp-affine-pathwise} by $q_t$ and use $q_t\geq q_1$ to obtain
\[
    \|r_t\|=\frac{\|\widehat g_t\|}{q_t}
    \leq\frac{\Theta}{m}\sum_{i=1}^m\frac1{p_i}
       \left(\frac{\|g_t\|}{q_t}+\frac{\sqrt m\,\delta}{q_t}\right)
    \leq R_p.
\]
Moreover, squaring and averaging the last inequality in~\eqref{eq:pp-local-error}, then using~\eqref{eq:pp-gradient-heterogeneity}, yields
\begin{align*}
    \frac1m\sum_{i=1}^m\|b_t^{(i)}\|^2&\leq\frac{(K-1)^2L^2\Theta^2\tau_t^2}{4q_t^2}\frac1m\sum_{i=1}^m\|\nabla F_i(x_t)\|^2\\
    &\leq\frac{(K-1)^2L^2\Theta^2\tau_t^2}{4}\left(\frac{\|g_t\|^2}{q_t^2}+\frac{\delta^2}{q_t^2}\right)
     \leq B_p^2\tau_t^2.
\end{align*}
Convexity of $\|\cdot\|^2$ then gives $\|b_t\|^2\leq m^{-1}\sum_{i=1}^m\|b_t^{(i)}\|^2\leq B_p^2\tau_t^2$, which completes the proof.
\end{proof}

Using $\|\bar g_t^{(i)}\|^2\leq2\|\nabla F_i(x_t)\|^2+2\|b_t^{(i)}\|^2$, then it follows from Lemma~\ref{lem:participation-moments} and Lemma~\ref{lem:pp-displacement} that
\[
    \E_t[\|\nu_t\|^2]\leq\frac2m\max_{i\in[m]}\frac{1-p_i}{p_i}(\|g_t\|^2+\delta^2+B_p^2\tau_t^2),
\]
which together with $\widehat g_t=g_t+b_t+\nu_t$ implies
\begin{align}
    &\E_t[\|\widehat g_t\|^2]=\|g_t+b_t\|^2+\E_t[\|\nu_t\|^2]\notag\\
    \leq&2\|g_t\|^2+2B_p^2\tau_t^2+\frac2m\max_{i\in[m]}\frac{1-p_i}{p_i}(\|g_t\|^2+\delta^2+B_p^2\tau_t^2)=A_p\|g_t\|^2+\sigma_p^2+A_pB_p^2\tau_t^2.\label{eq:pp-oracle-bounds}
\end{align}

In the analysis, we will use the following bounds on trajectory quantities:
\begin{align}
    \sum_{t=1}^T\|\widehat g_t\|^2 &=\eta^2(q_{T+1}^2-q_1^2),
    \label{eq:pp-accumulator-identity}\\
    \sum_{t=1}^T\|r_t\|^2 &\leq (\eta^2+R_p^2)\log\frac{q_{T+1}^2}{q_1^2},
    \label{eq:pp-displacement-energy}\\
    \sum_{t=1}^T\frac{\|\widehat g_t\|^2}{q_t} &\leq\eta^2\left(1+\sqrt{1+R_p^2/\eta^2}\right)(q_{T+1}-q_1).
    \label{eq:pp-weighted-direction-energy}
\end{align}
The first identity follows directly from~\eqref{eq:pp-method}. The remaining two bounds are obtained by applying Lemma~\ref{lem:summation} with $A_t=q_t^2$, $a_t=\|\widehat g_t\|^2/\eta^2$, and $c=R_p^2/\eta^2$ due to $a_t/A_t=\|r_t\|^2/\eta^2\leq R_p^2/\eta^2$.

\subsection{Smooth nonconvex convergence}
For ease of presentation, we define
\begin{equation*}
    \mathcal P_T:=1+\log(1+\sigma_p\sqrt T).
\end{equation*}
Thus $\mathcal P_T=1$ when $\sigma_p=0$, while $\mathcal P_T\lesssim\log(1+T)$ for $\sigma_p>0$.

\begin{theorem}[Smooth nonconvex convergence]\label{thm:pp-nonconvex}
Let Assumptions~\ref{ass:lower-bounded}and~\ref{ass:pp-heterogeneity} hold and assume each $F_i$ is $L$-smooth. Let $\{x_t\}_{t\ge 1}$ be the iterates generated by Algorithm~\ref{alg:FL-adagrad-partial} with $\tau_t$ chosen as in~\eqref{eq:partial-tau} and arbitrary $\gamma,\eta>0$. Then for any $T\geq1$, we have
\begin{equation*}
    \E[\|\nabla F(\bar x_T)\|^2]=\frac1T\sum_{t=1}^T\E[\|\nabla F(x_t)\|^2]\lesssim\frac{\mathcal P_T^2}{T}+\frac{\sigma_p\mathcal P_T}{\sqrt T}.
\end{equation*}
\end{theorem}

\begin{proof}[Proof of Theorem~\ref{thm:pp-nonconvex}]
We first estimate $\E[q_{T+1}]$ and $\E[\mathcal G_T]$.
Since $q_t$ is $\mathcal F_t$-measurable, smoothness of $F$ gives
\begin{align}
    F(x_{t+1})\leq& F(x_t)-\frac{\langle g_t,\widehat g_t\rangle}{q_t}+\frac L2\|r_t\|^2\notag\\
    =& F(x_t)-\frac{\langle g_t,b_t\rangle}{q_t}-\frac{\langle g_t,\nu_t\rangle}{q_t}-\frac{\|g_t\|^2}{q_t}+\frac L2\|r_t\|^2\notag\\
    \leq& F(x_t)-\frac{\|g_t\|^2}{2q_t}+\frac{\|b_t\|^2}{2q_t}-\frac{\langle g_t,\nu_t\rangle}{q_t}+\frac L2\|r_t\|^2.
    \label{eq:pp-nonconvex-pathwise-descent}
\end{align}
Taking conditional expectations on both sides of~\eqref{eq:pp-nonconvex-pathwise-descent} and using $\E_t[\nu_t]=0$ gives
\[
    \E_t[F(x_{t+1})] \leq F(x_t)-\frac{\|g_t\|^2}{2q_t} +\frac{\|b_t\|^2}{2q_t} +\frac L2\E_t[\|r_t\|^2].
\]
Summing the above inequality, we obtain from $q_t\geq q_1$, $\|b_t\|\leq B_p\tau_t$, $\widetilde{\mathcal T}_T\lesssim 1$ under~\eqref{eq:tau_t-scheme}, and~\eqref{eq:pp-displacement-energy} that
\begin{align}
    \E[\mathcal G_T]\leq2\bigl(F(x_1)-F^\star\bigr)+\frac{B_p^2}{q_1}\widetilde{\mathcal T}_T+L(\eta^2+R_p^2)\E\left[\log\frac{q_{T+1}^2}{q_1^2}\right]\lesssim1+\log\E[q_{T+1}].
    \label{eq:pp-first-moment-descent}
\end{align}
Applying Lemma~\ref{lem:pp-predictable} with $c=q_1^2$ and $Z_t=\|\widehat g_t\|^2/\eta^2$, together with~\eqref{eq:pp-accumulator-identity},~\eqref{eq:pp-oracle-bounds}, and $\sqrt{a+b}\leq\sqrt a+\sqrt b$, we get
\begin{align}
    \E[q_{T+1}]&=\E\left[\sqrt{q_1^2+\frac1{\eta^2}\sum_{t=1}^T\|\widehat g_t\|^2}\right]\leq\sqrt2\,\E\left[\sqrt{q_1^2+\frac1{\eta^2}\sum_{t=1}^T\E_t[\|\widehat g_t\|^2]}\right]\notag\\
    &\leq\sqrt2\,\E\left[\sqrt{q_1^2+\frac{A_p}{\eta^2}\sum_{t=1}^T\|g_t\|^2+\frac{\sigma_p^2T}{\eta^2}+\frac{A_p B_p^2}{\eta^2}\widetilde{\mathcal T}_T}\right]\notag\\
    &\lesssim1+\E\left[\sqrt{\sum_{t=1}^T\|g_t\|^2}\right]+\sigma_p\sqrt T.\label{eq:partial-qT-bd1}
\end{align}
On the other hand, monotonicity of $q_t$ and then Cauchy--Schwarz yield
\[
    \E\left[\sqrt{\sum_{t=1}^T\|g_t\|^2}\right]\leq\E\left[\sqrt{q_{T+1}\sum_{t=1}^T\frac{\|g_t\|^2}{q_t}}\right]\leq\sqrt{\E[q_{T+1}]\E[\mathcal G_T]}.
\]
Substituting the above bound into~\eqref{eq:partial-qT-bd1} therefore gives
\[
    \E[q_{T+1}]\lesssim1+\sqrt{\E[q_{T+1}]\E[\mathcal G_T]}+\sigma_p\sqrt T,
\]
which further imply that for a constant $C$,
\begin{align*}
   \E[q_{T+1}]\leq&\frac{\E[q_{T+1}]}{2}+C(1+\sigma_p\sqrt T)+\frac{C^2}{2}\E[\mathcal G_T].
\end{align*}
Simplifying the above and applying~\eqref{eq:pp-first-moment-descent} then gives
\[
\E[q_{T+1}]\lesssim 1+\sigma_p\sqrt T+\log\E[q_{T+1}],
\]
solving which with Lemma~\ref{lem:log-solution} implies
\begin{equation}
    \E[q_{T+1}]\lesssim1+\sigma_p\sqrt T,
    \qquad \E[\mathcal G_T]\lesssim1+\log(1+\sigma_p\sqrt T)=\mathcal P_T.
    \label{eq:pp-first-moment-bounds}
\end{equation}

We next obtain the second moment $\E[q_{T+1}^2]$ and $\E[\mathcal G_T^2]$.
For $0\leq t\leq T$, define
\[
    \mathcal Z_t:=2\bigl(F(x_{t+1})-F^\star\bigr)+\sum_{s=1}^t\frac{\|g_s\|^2}{q_s},\quad \mathcal Z_T^\ast:=\max_{0\leq t\leq T}\mathcal Z_t,
    \quad \mathcal N_t:=-2\sum_{s=1}^t\frac{\langle g_s,\nu_s\rangle}{q_s},
\]
The process $\{\mathcal N_t\}$ is a martingale since $\E_t[\nu_t]=0$. Summing~\eqref{eq:pp-nonconvex-pathwise-descent} over $s=1$ to $t$, then maximizing over $t$ and using $\sum_{t=1}^{T}\|b_t\|^2\lesssim\widetilde{\mathcal T}_T\lesssim1$ yields
\begin{equation}
    \mathcal Z_T^\ast\lesssim1+\mathcal R_T+\max_{0\leq t\leq T}|\mathcal N_t|.
    \label{eq:pp-nonconvex-potential-bound}
\end{equation}
Squaring~\eqref{eq:pp-displacement-energy} and then plugging~\eqref{eq:pp-first-moment-bounds} imply
\begin{equation}
\begin{aligned}
    \sqrt{\E[\mathcal R_T^2]}
    &\leq2(\eta^2+R_p^2)\sqrt{\E\left[\log^2\frac{q_{T+1}}{q_1}\right]}\\
    &\leq2(\eta^2+R_p^2)\left(1+\log\frac{\E[q_{T+1}]}{q_1}\right)
     \lesssim\mathcal P_T,
\end{aligned}
\label{eq:pp-energy-second-moment}
\end{equation}
where the second inequality follows from concavity of the function $u\mapsto(1+\log u)^2$ on $[1,\infty)$. Additionally, by smoothness, $\|g_t\|^2\leq 2L\bigl(F(x_t)-F^\star\bigr)\leq L\mathcal Z_T^\ast$.
\begin{align}
    \E\left[\max_{0\leq t\leq T}|\mathcal N_t|^2\right]\leq& 4\E[|\mathcal N_T|^2]\leq16\sum_{t=1}^T\E\left[\frac{\langle g_t,\nu_t\rangle^2}{q_t^2}\right]\leq16\sum_{t=1}^T\E\left[\frac{\|g_t\|^2\|\nu_t\|^2}{q_t^2}\right]\notag\\
    \leq& 16L\E[\mathcal Z_T^\ast\mathcal R_T]\leq 16L\sqrt{\E[(\mathcal Z_T^\ast)^2]}\sqrt{\E[\mathcal R_T^2]}\lesssim\sqrt{\E[(\mathcal Z_T^\ast)^2]}\mathcal P_T,\label{eq:martingale_sq-max}
\end{align}
where the fourth inequality follows from $\E_t[\|\nu_t/q_t\|^2]\leq\E_t[\|r_t\|^2]$ due to $\nu_t/q_t=-(r_t-\E_t[r_t])$.
Taking $L^2$ norms in~\eqref{eq:pp-nonconvex-potential-bound} and using~\eqref{eq:martingale_sq-max} then gives
\[
    \sqrt{\E[(\mathcal Z_T^\ast)^2]}
    \lesssim1+\mathcal P_T+\sqrt{\mathcal P_T\sqrt{\E[(\mathcal Z_T^\ast)^2]}}
    \leq\frac12\sqrt{\E[(\mathcal Z_T^\ast)^2]}+C\mathcal P_T
\]
for some constant $C>0$. Rearranging the above yields $\sqrt{\E[(\mathcal Z_T^\ast)^2]}\lesssim\mathcal P_T$, which together with $\mathcal G_T\leq\mathcal Z_T\leq\mathcal Z_T^\ast$ implies that
\begin{equation}
    \sqrt{\E[\mathcal G_T^2]}\lesssim\mathcal P_T.
    \label{eq:pp-weighted-gradient-second-moment}
\end{equation}

It remains to bound $\E[q_{T+1}^2]$. By monotonicity of $q_t$ and~\eqref{eq:pp-weighted-gradient-second-moment}, we have
\begin{equation}\label{eq:partial-gt_sq-sum1}
    \sum_{t=1}^T\E[\|g_t\|^2]\leq\E[q_{T+1}\mathcal G_T]
    \leq\sqrt{\E[q_{T+1}^2]}\sqrt{\E[\mathcal G_T^2]}
    \lesssim\mathcal P_T\sqrt{\E[q_{T+1}^2]}.
\end{equation}
Therefore, it follows from~\eqref{eq:pp-accumulator-identity},~\eqref{eq:pp-oracle-bounds}, and $\widetilde{\mathcal T}_T\lesssim1$ that
\begin{align*}
    \E[q_{T+1}^2]
    &\leq q_1^2+\frac{A_p}{\eta^2}\sum_{t=1}^T\E[\|g_t\|^2]
       +\frac{\sigma_p^2T}{\eta^2}
       +\frac{A_pB_p^2}{\eta^2}\widetilde{\mathcal T}_T\\
    &\lesssim1+\mathcal P_T\sqrt{\E[q_{T+1}^2]}+\sigma_p^2T,
\end{align*}
which yields 
\begin{equation}\label{eq:partial-qT_sq-bd}
    \sqrt{\E[q_{T+1}^2]}\lesssim\mathcal P_T+\sigma_p\sqrt T.
\end{equation}
Therefore, we obtain from~\eqref{eq:partial-gt_sq-sum1} that
\[
    \sum_{t=1}^T\E[\|g_t\|^2]\lesssim\mathcal P_T^2+\sigma_p\mathcal P_T\sqrt T.
\]
Finally, since $\bar x_T$ is chosen uniformly from $\{x_t\}_{t=1}^T$, it follows that
\[
    \E[\|\nabla F(\bar x_T)\|^2]=\frac1T\sum_{t=1}^T\E[\|g_t\|^2]\lesssim\frac{\mathcal P_T^2}{T}+\frac{\sigma_p\mathcal P_T}{\sqrt T},
\]
which completes the proof.
\end{proof}

\subsection{Smooth convex convergence}
For convex problems, we similarly define the counterfactual local error
\begin{equation*}
    \epsilon_t:=\frac1{mK}\sum_{i=1}^m\sum_{k=0}^{K-1}\left[F_i(x_t)-F_i(x_{i,k}^t)-\langle\nabla F_i(x_{i,k}^t),x_t-x_{i,k}^t\rangle\right]\geq0.
\end{equation*}
Then the same convexity calculation as in~\eqref{eq:cvx-F-bound} gives
\begin{equation}
    F(y)\geq F(x_t)+\langle\bar g_t,y-x_t\rangle-\epsilon_t,\qquad y\in\mathbb R^d.
    \label{eq:pp-approx-subgradient}
\end{equation}

\begin{theorem}[Smooth convex convergence]\label{thm:pp-convex-radius}
Let Assumptions~\ref{ass:lower-bounded} and~\ref{ass:pp-heterogeneity} hold and assume each $F_i$ is $L$-smooth and convex. Let $\{x_t\}_{t\ge 1}$ be the iterates generated by Algorithm~\ref{alg:FL-adagrad-partial} with $\tau_t$ chosen as in~\eqref{eq:partial-tau} and arbitrary $\gamma,\eta>0$. Then for any $T\geq1$, we have
\begin{equation*}
    \E[F(\bar x_T)-F(x^\star)]
    \lesssim\frac{\mathcal P_T^2}{T}+\frac{\sigma_p\mathcal P_T}{\sqrt T}.
\end{equation*}
\end{theorem}
\begin{proof}[Proof of Theorem~\ref{thm:pp-convex-radius}]
Let $d_t=x_t-x^\star$, $D_T:=\max_{1\leq t\leq T+1}\|d_t\|^2$, and $\Delta_t:=F(x_t)-F(x^\star)$.
Since Lemma~\ref{lem:sg-smooth-convex-error} applies to the exact counterfactual paths, we apply it and then use~\eqref{eq:pp-gradient-heterogeneity} and Lemma~\ref{lem:pp-displacement} to obtain
\begin{equation}
\label{eq:pp-smooth-epsilon}
    \epsilon_t\leq\frac{c_KL\Theta^2\tau_t^2}{2mq_t^2}\sum_{i=1}^m\|\nabla F_i(x_t)\|^2\leq\frac{c_KL\Theta^2\tau_t^2}{2}\left(\frac{\|g_t\|^2}{q_t^2} +\frac{\delta^2}{q_t^2}\right)\lesssim\tau_t^2.
\end{equation}
Expand $\E_t[\|d_{t+1}\|^2]$, then it follows from $\E_t[\widehat g_t]=\bar g_t$,~\eqref{eq:pp-approx-subgradient} with $y=x^\star$ that
\begin{equation}
\begin{aligned}
    \E_t[\|d_{t+1}\|^2] &=\|d_t\|^2 -\frac2{q_t}\langle\bar g_t,d_t\rangle +\E_t[\|r_t\|^2]\\
       &\leq\|d_t\|^2 -\frac2{q_t}\Delta_t +\E_t[\|r_t\|^2]+\frac{2\epsilon_t}{q_t}.
\end{aligned}
\label{eq:pp-convex-distance-recursion}
\end{equation}
Summing~\eqref{eq:pp-convex-distance-recursion} over $t=1$ to $T$, and applying~\eqref{eq:pp-smooth-epsilon},~\eqref{eq:pp-first-moment-bounds},~\eqref{eq:pp-displacement-energy}, and $\widetilde{\mathcal T}_T\lesssim 1$ give
\begin{equation*}
    \E[\|d_{T+1}\|^2]
    +2\sum_{t=1}^T\E\left[\Delta_t/q_t\right]
    \lesssim\mathcal P_T.
\end{equation*}
To proceed, consider the martingale
\[
    \mathcal M_s:=-2\sum_{t=1}^s\frac{\langle\nu_t,d_t\rangle}{q_t},\qquad 0\leq s\leq T.
\]
Expanding $\|d_{t+1}\|^2$ and using~\eqref{eq:pp-approx-subgradient} leads to
\begin{equation}
    \|d_{t+1}\|^2+\frac{2\Delta_t}{q_t}\leq\|d_t\|^2+\|r_t\|^2+\frac{2\epsilon_t}{q_t}-\frac{2\langle\nu_t,d_t\rangle}{q_t}.
    \label{eq:pp-convex-pathwise-distance}
\end{equation}
Sum this inequality to any $s\leq T$ and maximize over $s$, then it follows from~\eqref{eq:pp-smooth-epsilon}, $\Delta_t\geq0$, and $\widetilde{\mathcal T}_T\lesssim 1$ that
\begin{equation}\label{eq:partial-DT-bd1}
    D_T\lesssim1+\mathcal R_T+\max_{0\leq s\leq T}|\mathcal M_s|.
\end{equation}
On the other hand, using similar arguments as in~\eqref{eq:martingale_sq-max}, we have
\begin{equation}
    \E\left[\max_{0\leq s\leq T}|\mathcal M_s|^2\right]\leq16\sum_{t=1}^T\E\left[\|d_t\|^2\E_t[\|r_t\|^2]\right]\leq16\E[D_T\mathcal R_T]\lesssim\mathcal P_T\sqrt{\E[D_T^2]}.
    \label{eq:pp-convex-distance-martingale}
\end{equation}
Therefore, squaring~\eqref{eq:partial-DT-bd1} and using~\eqref{eq:pp-convex-distance-martingale} and~\eqref{eq:pp-energy-second-moment} give
\[
    \sqrt{\E[D_T^2]}\lesssim1+\mathcal P_T+\sqrt{\mathcal P_T\sqrt{\E[D_T^2]}}
    \leq\frac12\sqrt{\E[D_T^2]}+C\mathcal P_T,
    \qquad \sqrt{\E[D_T^2]}\lesssim\mathcal P_T.
\]

It remains to bound $\sum_{t=1}^T\E[\Delta_t]$. Rearranging~\eqref{eq:pp-convex-pathwise-distance} and plugging in $\|r_t\|=\|\widehat g_t/q_t\|s$ gives
\begin{equation}\label{eq:partial-Delta_t-1}
    \Delta_t\leq \frac{q_t}{2}\left(\|d_t\|^2-\|d_{t+1}\|^2\right) +\frac{\|\widehat g_t\|^2}{2q_t} +\epsilon_t-\langle\nu_t,d_t\rangle.
\end{equation}
Note that the last term in~\eqref{eq:partial-Delta_t-1} has zero expectation. For the sum of the first term, we have
\[
    \sum_{t=1}^T q_t \left(\|d_t\|^2-\|d_{t+1}\|^2\right)=q_1\|d_1\|^2+\sum_{t=2}^T(q_t-q_{t-1})\|d_t\|^2-q_T\|d_{T+1}\|^2\leq q_T D_T.
\]
Consequently, sum~\eqref{eq:partial-Delta_t-1} over $t$, then use the above inequality,~\eqref{eq:pp-weighted-direction-energy}, and~\eqref{eq:partial-qT_sq-bd} from the proof of Theorem~\ref{thm:pp-nonconvex}, we derive
\begin{align*}
    \sum_{t=1}^T\E[\Delta_t]\lesssim&\E[q_T D_T]+\E[q_{T+1}]+1\lesssim\sqrt{\E[D_T^2]}\sqrt{\E[q_{T+1}^2]}+\sqrt{\E[q_{T+1}^2]}\\
    \lesssim&\mathcal P_T\sqrt{\E[q_{T+1}^2]}\lesssim \mathcal P_T^2+\sigma_p\mathcal P_T\sqrt T,
\end{align*}
which completes the proof after dividing both sides by $T$.
\end{proof}

\subsection{Convex Lipschitz convergence}
Note that in this subsection, we do not require the bounded heterogeneity assumption in Assumption~\ref{ass:pp-heterogeneity}.

\begin{theorem}[Lipschitz convex  convergence]\label{thm:pp-convex-lipschitz}
Let Assumption~\ref{ass:lower-bounded} hold. Assume each $F_i$ is $G$-Lipschitz and convex. Let $\{x_t\}_{t\ge 1}$ be the iterates generated by Algorithm~\ref{alg:FL-adagrad-partial} with $\tau_t$ chosen as in~\eqref{eq:partial-tau} and arbitrary $\gamma,\eta>0$. Then for any $T\geq1$, we have
\begin{equation*}
    \E[F(\bar x_T)-F(x^\star)]\lesssim\frac{\log(1+T)}{\sqrt T}.
\end{equation*}
\end{theorem}
\begin{proof}[Proof of Theorem~\ref{thm:pp-convex-lipschitz}]
Let $d_t=x_t-x^\star$ and $\Delta_t:=F(x_t)-F(x^\star)$. For this proof only, set
\[
    \widetilde{G}:=\frac{G}{m}\sum_{i=1}^{m}\frac1{p_i},
    \qquad
    M_G:=G^2\left(1+\frac1{m^2}\sum_{i=1}^{m}\frac{1-p_i}{p_i}\right).
\]
Then it follows from the definition of $\widehat g_t$ and~\eqref{eq:hat_gt_bd} that
\begin{equation}
    \|\widehat g_t\|\leq\frac{G}{m}\sum_{i\in\mathcal S_t}p_i^{-1}\leq \widetilde{G},\qquad
    \E_t[\|\widehat g_t\|^2]=\|\bar g_t\|^2+\frac1{m^2}\sum_{i=1}^{m}\frac{1-p_i}{p_i}\|\bar g_t^{(i)}\|^2\leq M_G.
    \label{eq:pp-lipschitz-oracle-bounds}
\end{equation}
Consequently,~\eqref{eq:pp-accumulator-identity} and~\eqref{eq:pp-lipschitz-oracle-bounds} give
\[
    q_{T+1}\leq\sqrt{q_1^2+\frac{T\widetilde{G}^2}{\eta^2}}\lesssim\sqrt T,
    \qquad
    \E[q_{T+1}^2]\leq q_1^2+\frac{TM_G}{\eta^2}.
\]
Since $\|r_t\|=\|\widehat g_t\|/q_t\leq \widetilde{G}/q_1$, we further obtain from~\eqref{eq:pp-displacement-energy} with $R_p=\widetilde{G}/q_1$ that
\begin{equation}
    \sum_{t=1}^T\E[\|r_t\|^2]\leq\left(\eta^2+\frac{\widetilde{G}^2}{q_1^2}\right)\E\left[\log\frac{q_{T+1}^2}{q_1^2}\right]\lesssim \log\left(1+T\right),
    \label{eq:pp-lipschitz-energy}
\end{equation}

On the other hand, rearranging~\eqref{eq:pp-convex-distance-recursion} and noting that $\epsilon_t\leq(K-1)G^2\tau_t/q_t$ from Lemma~\ref{lem:sg-lipschitz-local} with $\sigma=0$ yields
\[
    \E_t[\|d_{t+1}\|^2]+\frac{2\Delta_t}{q_t}
    \leq\|d_t\|^2+\E_t[\|r_t\|^2]+\frac{2\epsilon_t}{q_t}
    \leq\|d_t\|^2+\E_t[\|r_t\|^2]+\frac{2(K-1)G^2\tau_t}{q_1^2}.
\]
Summing the above inequality from $t=1$ to $T$, applying~\eqref{eq:pp-lipschitz-energy}, and using $\mathcal T_T\lesssim\log(1+T)$ give
\[
    \E[\|d_{T+1}\|^2]
    +2\sum_{t=1}^T\E\left[\frac{\Delta_t}{q_t}\right]
    \lesssim\log(1+T).
\]
Finally, it follows from $q_t\lesssim\sqrt T$ that
\[
    \E[F(\bar x_T)-F(x^\star)]=\frac1T\sum_{t=1}^T\E[\Delta_t]\lesssim\frac1{\sqrt T}\sum_{t=1}^T\E\left[\frac{\Delta_t}{q_t}\right]\lesssim\frac{\log(1+T)}{\sqrt T},
\]
which completes the proof.
\end{proof}

\section{Extension II: Adaptation to H\"older smoothness}\label{app:holder-convex}
For convex objectives, \citet{orabona2023normalized} showed that the standard AdaGrad with exact gradients can adapt to H\"older smoothness, which interpolates between smoothness and Lipschitzness.
Inspired by their work, we extend the convex analysis for Algorithm~\ref{alg:FL-adagrad} under the H\"older-smooth condition. Similar to \citep{orabona2023normalized} and for simplicity, we focus on Algorithm~\ref{alg:FL-adagrad} with exact local gradients, i.e., $g_{i,k}^t=\nabla F_i(x_{i,k}^t)$.

\subsection{Auxiliary Results}
Recall that a function $F:\mathbb R^d\to\mathbb R$ is $(L_\alpha,\alpha)$-H\"older smooth with $L_\alpha>0$ and $\alpha\in(0,1]$ if
\[
    \|\nabla F(x)-\nabla F(y)\|\leq L_\alpha\|x-y\|^\alpha,\qquad x,y\in\mathbb R^d.
\]
Note that the H\"older smoothness condition also implies~\citep[Theorem~1]{orabona2023normalized}:
\begin{equation}
    \label{eq:holder-descent}
    F(y)\leq F(x)+\langle\nabla F(x),y-x\rangle+\frac{L_\alpha}{1+\alpha}\|y-x\|^{1+\alpha},\qquad x,y\in\mathbb R^d.
\end{equation}
Then $\alpha=1$ corresponds to the standard $L$-smoothness, while $\alpha=0$ corresponds to Lipschitzness.
Since we have already established analysis for these two cases separately, in the remainder of this subsection, we mainly focus on $\alpha\in(0,1)$.

We first invoke the following lemma that converts gradient norms to function values.
\begin{lemma}[H\"older self-bounding property~\citep{orabona2023normalized}]\label{lem:sg-holder-self-bound}
Suppose $F$ is convex and $(L_\alpha,\alpha)$-H\"older smooth with $\alpha\in(0,1]$. Then for any $x\in\mathbb R^d$, we have
\begin{equation*}
    \|\nabla F(x)\|^{(1+\alpha)/\alpha}
    \leq\left(1+\frac1\alpha\right)L_\alpha^{1/\alpha}\bigl(F(x)-F^\star\bigr).
\end{equation*}
\end{lemma}

Recall $b_t:=\bar g_t-g_t$, $d_t:=x_t-x^\star$, and $\epsilon_t$ defined in~\eqref{eq:eps-def}. We next bound the local errors $\|b_t\|$ and $\epsilon_t$ in the following lemma.
\begin{lemma}\label{lem:sg-holder-local-error}
Let Assumption~\ref{ass:lower-bounded} hold. Assume each $F_i$ is convex and $(L_\alpha,\alpha)$-H\"older smooth for some $\alpha\in(0,1)$. Consider Algorithm~\ref{alg:FL-adagrad} with exact local gradients $g_{i,k}^t=\nabla F_i(x_{i,k}^t)$ and $\tau_t$ chosen as~\eqref{eq:partial-tau}. Then
\begin{equation*}
    \sum_{t=1}^T\left(\|b_t\|^{(1+\alpha)/\alpha}+\epsilon_t\right)\lesssim 1+D_T^{\alpha(1+\alpha)/2}.
\end{equation*}
\end{lemma}
\begin{proof}
We focus on $K\geq2$ since the claims for $K=1$ are immediate due to $b_t=\epsilon_t=0$. For any $k<K$, the local updates and $(L_\alpha,\alpha)$-H\"older smoothness of $F_i$ give
\[
    \|x_{i,k}^t-x_t\|\leq \frac{\tau_t}{q_t}\sum_{s=0}^{k-1}\|\nabla F_i(x_{i,s}^t)\| \leq \frac{K\tau_t}{q_t}\left(\|\nabla F_i(x_t)\|+L_\alpha\Big(\max_{0\leq s<K}\|x_{i,s}^t-x_t\|\Big)^\alpha\right),
\]
which further implies
\[
    \max_{0\leq k<K}\|x_{i,k}^t-x_t\|
    \leq\frac{K\tau_t}{q_t}\|\nabla F_i(x_t)\|+\frac{KL_\alpha\tau_t}{q_t}\Big(\max_{0\leq k<K}\|x_{i,k}^t-x_t\|\Big)^\alpha.
\]
Solving the above inequality using Lemma~\ref{lem:sublinear-power-bound} gives
\begin{equation}
    \max_{0\leq k<K}\|x_{i,k}^t-x_t\|
    \leq\frac{2K\tau_t}{q_t}\|\nabla F_i(x_t)\|+\Big(\frac{2KL_\alpha\tau_t}{q_t}\Big)^{\frac{1}{1-\alpha}}.
    \label{eq:sg-holder-local-distance}
\end{equation}
Moreover, using H\"older smoothness,~\eqref{eq:sg-holder-local-distance}, and $(a+b)^\alpha\leq a^\alpha+b^\alpha$ for $\alpha\in(0,1)$, we obtain
\begin{equation}
    \label{eq:holder-bt-bd}
    \|b_t\|\leq\frac{L_\alpha}{mK}\sum_{i=1}^m\sum_{k=0}^{K-1}\|x_{i,k}^t-x_t\|^\alpha\lesssim\left(\frac{K\tau_t}{q_t}\right)^\alpha\big(1+\|d_t\|^{\alpha^2}\big)+\left(\frac{K\tau_t}{q_t}\right)^{\frac{\alpha}{1-\alpha}}.
\end{equation}
where we used $\|\nabla F_i(x_t)\|\leq\|\nabla F_i(x^\star)\|+L_\alpha\|d_t\|^\alpha$ in the last inequality.
Similarly, we have from~\eqref{eq:holder-descent},~\eqref{eq:sg-holder-local-distance}, and $(a+b)^c\leq 2^{c-1}(a^c+b^c)$ for $c\ge 1$ that
\begin{equation}
    \label{eq:holder-eps-bd}
    \epsilon_t\leq\frac{L_\alpha}{(1+\alpha)mK}\sum_{i=1}^m\sum_{k=0}^{K-1}\|x_{i,k}^t-x_t\|^{1+\alpha}\lesssim\left(\frac{K\tau_t}{q_t}\right)^{1+\alpha}\left(1+\|d_t\|^{\alpha(1+\alpha)}\right)+\left(\frac{K\tau_t}{q_t}\right)^{\frac{1+\alpha}{1-\alpha}}.   
\end{equation}
Then raising~\eqref{eq:holder-bt-bd} to power $(1+\alpha)/\alpha$ and adding~\eqref{eq:holder-eps-bd} leads to
 \begin{equation}
    \|b_t\|^{(1+\alpha)/\alpha}+\epsilon_t\lesssim\left(\frac{K\tau_t}{q_t}\right)^{1+\alpha}\left(1+\|d_t\|^{\alpha(1+\alpha)}\right)+\left(\frac{K\tau_t}{q_t}\right)^{\frac{1+\alpha}{1-\alpha}}.
    \label{eq:sg-holder-local-error}
\end{equation}
Finally, note that $q_t\geq q_1$, and $\tau_t\lesssim 1/((K-1)t)$ satisfies
\[
    \sum_{t=1}^T(K\tau_t)^{1+\alpha}\lesssim 1,
    \qquad
    \sum_{t=1}^T(K\tau_t)^{(1+\alpha)/(1-\alpha)}\lesssim 1,\quad\forall\alpha\in(0,1).
\]
Summing~\eqref{eq:sg-holder-local-error} and using $\|d_t\|^{\alpha(1+\alpha)}\leq D_T^{\alpha(1+\alpha)/2}$ thus completes the proof.
\end{proof}

\begin{lemma}\label{lem:sg-holder-trajectory}
Consider the same setting as in Lemma~\ref{lem:sg-holder-local-error}.and recall $D_T:=\max_{1\leq t\leq T+1}\|d_t\|^2$. Then, for every $T\geq1$, we have
\begin{align}
    D_T&\lesssim 1+\log(q_{T+1}/q_1),
    \label{eq:sg-holder-distance}\\
    \sum_{t=1}^Tq_t\|r_t\|^2
    &\lesssim q_{T+1}\left(1+\log(q_{T+1}/q_1)\right).
    \label{eq:sg-holder-weighted-energy}
\end{align}
\end{lemma}
\begin{proof}
Let $\Delta_t:=F(x_t)-F^\star$ and $p:=(1+\alpha)/\alpha>2$.
Since $\bar g_t=g_t+b_t$ and $\|\bar g_t\|=q_t\|r_t\|$, we have from $(a+b)^c\leq 2^{c-1}(a^c+b^c)$ for $c\ge 1$ that
\begin{align*}
    \|g_t\|^p\geq2^{1-p}\|\bar g_t\|^p-\|b_t\|^p=2^{1-p}q_t^p\|r_t\|^p-\|b_t\|^p,
\end{align*}
which together with Lemma~\ref{lem:sg-holder-self-bound} gives
\begin{align*}
    -\frac{2\Delta_t}{q_t}\leq-\frac{2\alpha L_\alpha^{-1/\alpha}}{1+\alpha}\frac{\|g_t\|^p}{q_t}\leq-\frac{2^{2-p}\alpha L_\alpha^{-1/\alpha}}{1+\alpha}q_t^{p-1}\|r_t\|^p+\frac{2\alpha L_\alpha^{-1/\alpha}}{1+\alpha}\frac{\|b_t\|^p}{q_t}.
\end{align*}
Therefore, setting
\[
    c_\alpha:=\frac{2^{2-p}\alpha L_\alpha^{-1/\alpha}}{1+\alpha},
    \qquad
    C_\alpha:=\frac{2\alpha L_\alpha^{-1/\alpha}}{1+\alpha},
\]
we obtain from~\eqref{eq:sg-smooth-convex-one-step} with $\zeta=0$ and $q_t\ge q_1$ that
\begin{align*}
    \|d_{t+1}\|^2-\|d_t\|^2\leq\|r_t\|^2-\frac{2\Delta_t}{q_t}+\frac{2\epsilon_t}{q_t}\leq\|r_t\|^2-c_\alpha q_1^{p-1}\|r_t\|^p+\frac{C_\alpha\|b_t\|^p+2\epsilon_t}{q_t}.
\end{align*}
Applying Lemma~\ref{lem:positive-power-log-bound} to bound $\|r_t\|^2-c_\alpha q_1^{p-1}\|r_t\|^p$, we obtain
\begin{align}
    \|d_{t+1}\|^2-\|d_t\|^2&\leq\left(\eta^2+\bigl(c_\alpha q_1^{p-1}\bigr)^{-2/(p-2)}\right)\log\left(1+\frac{\|r_t\|^2}{\eta^2}\right)+\frac{C_\alpha\|b_t\|^p+2\epsilon_t}{q_t}\notag\\
    &\lesssim\log\left(1+\frac{\|r_t\|^2}{\eta^2}\right)+\frac{\|b_t\|^p+\epsilon_t}{q_t}.\label{eq:holder-dist-1}
\end{align}
The definitions of $r_t$ and $q_{t+1}$ give
\[
    \frac{\|r_t\|^2}{\eta^2}
    =\frac{\|\bar g_t\|^2}{\eta^2q_t^2}
    =\frac{q_{t+1}^2-q_t^2}{q_t^2}.
\]
Consequently, for every $1\leq s\leq T$,
\begin{align*}
    \sum_{t=1}^s\log\left(1+\frac{\|r_t\|^2}{\eta^2}\right)
    &=\sum_{t=1}^s\log\frac{q_{t+1}^2}{q_t^2}
      =2\log\frac{q_{s+1}}{q_1}
      \leq2\log\frac{q_{T+1}}{q_1}.
\end{align*}
Moreover, $q_t\geq q_1$ and Lemma~\ref{lem:sg-holder-local-error} imply
\[
    \sum_{t=1}^s\frac{\|b_t\|^p+\epsilon_t}{q_t}
    \leq\frac1{q_1}\sum_{t=1}^T\bigl(\|b_t\|^p+\epsilon_t\bigr)
    \lesssim 1+D_T^{\alpha(1+\alpha)/2}.
\]
Therefore, summing~\eqref{eq:holder-dist-1} from $t=1$ to $s$ and maximizing over $s\in\{1,\ldots,T\}$ give
\[
    D_T\lesssim 1+\log(q_{T+1}/q_1)+D_T^{\alpha(1+\alpha)/2}.
\]
Since $\alpha(1+\alpha)/2<1$, solving the above inequality yields~\eqref{eq:sg-holder-distance}.

We proceed to prove~\eqref{eq:sg-holder-weighted-energy}. Since $\nabla F(x^\star)=0$, the global H\"older condition gives
\[
    \|g_t\|^2=\|\nabla F(x_t)-\nabla F(x^\star)\|^2\leq L_\alpha^2\|d_t\|^{2\alpha}\leq L_\alpha^2D_T^\alpha.
\]
Moreover, since $K\tau_t/q_t\lesssim 1$, ~\eqref{eq:holder-bt-bd} therefore yields
\[
    \|b_t\|^2\lesssim 1+\|d_t\|^{2\alpha^2}\lesssim 1+D_T^\alpha.
\]
Thus, using $\bar g_t=g_t+b_t$ and~\eqref{eq:sg-holder-distance}, we further obtain
\begin{align*}
    \frac{\|\bar g_t\|^2/\eta^2}{q_t^2}
    \leq\frac{2}{\eta^2q_1^2}\bigl(\|g_t\|^2+\|b_t\|^2\bigr)\lesssim 1+D_T^\alpha
    \lesssim\left(1+\log(q_{T+1}/q_1)\right)^\alpha.
\end{align*}
Applying Lemma~\ref{lem:summation} to $A_t=q_t^2$, $a_t=\|\bar g_t\|^2/\eta^2$, and $c=\left(1+\log(q_{T+1}/q_1)\right)^{\alpha}$ yields
\begin{align*}
    \sum_{t=1}^Tq_t\|r_t\|^2&=\eta^2\sum_{t=1}^T\frac{\|\bar g_t\|^2/\eta^2}{q_t}\leq\eta^2(1+\sqrt{1+c})(q_{T+1} - q_1)\\
    &\lesssim q_{T+1}\left(1+\left(1+\log\frac{q_{T+1}}{q_1}\right)^{\alpha/2}\right)\lesssim q_{T+1}\left(1+\log\frac{q_{T+1}}{q_1}\right),
\end{align*}
which is~\eqref{eq:sg-holder-weighted-energy}, hence completing the proof.
\end{proof}

\subsection{Main Theorem}
\begin{theorem}[H\"older-smooth convex convergence]\label{thm:sg-holder-convex}
Let Assumption~\ref{ass:lower-bounded} hold. Assume each $F_i$ is convex and $(L_\alpha,\alpha)$-H\"older smooth for some $\alpha\in[0,1]$. 
Let $\{x_t\}_{t\geq1}$ be generated by Algorithm~\ref{alg:FL-adagrad} with exact local gradients $g_{i,k}^t=\nabla F_i(x_{i,k}^t)$, $\tau_t$ chosen as~\eqref{eq:partial-tau}, and arbitrary $\gamma,\eta>0$.
Then, for every $T\geq1$,
\begin{equation*}
    \E[F(\bar x_T)-F^\star]
    \lesssim
    \begin{cases}
        \displaystyle\frac{(\log(1+T))^{1+\alpha}}{T^{(1+\alpha)/2}}, & 0\leq \alpha<1,\\[0.8em]
        \displaystyle\frac1T, & \alpha=1.
    \end{cases}
\end{equation*}
\end{theorem}
\begin{proof}[Proof of Theorem~\ref{thm:sg-holder-convex}]
The special cases $\alpha=1$ and $\alpha=0$ are covered by Corollaries~\ref{cor:sg-smooth-convex-deterministic} and~\ref{cor:sg-lipschitz-deterministic}, respectively. It remains to consider $\alpha\in(0,1)$. Let $\Delta_t:=F(x_t)-F^\star$. We start from~\eqref{eq:sg-smooth-convex-one-step} with $\zeta=0$ to obtain
\[
    \Delta_t\leq \frac{q_t}{2}\bigl(\|d_t\|^2-\|d_{t+1}\|^2\bigr)+\frac{q_t}{2}\|r_t\|^2+\epsilon_t.
\]
Summing this inequality over $t=1,\ldots,T$, plugging in~\eqref{eq:sg-lipschitz-convex-summation-by-parts}, and applying Lemmas~\ref{lem:sg-holder-local-error} and~\ref{lem:sg-holder-trajectory} then yield
\begin{equation}
    \sum_{t=1}^T\Delta_t\leq q_TD_T+\frac12\sum_{t=1}^Tq_t\|r_t\|^2+\sum_{t=1}^T\epsilon_t
    \lesssim q_{T+1}\left(1+\log\frac{q_{T+1}}{q_1}\right).
    \label{eq:sg-holder-objective-master}
\end{equation}

It remains to bound $q_{T+1}$.
Applying H\"older's inequality to $\{1\}_{t=1}^T$ and $\{\|g_t\|^2\}_{t=1}^T$ with conjugate exponents $p:=\frac{1+\alpha}{1-\alpha}$ and $q:=\frac{1+\alpha}{2\alpha}$
and then using Lemma~\ref{lem:sg-holder-self-bound}, we obtain
\[
    \sum_{t=1}^T\|g_t\|^2
    \leq T^{\frac{1-\alpha}{1+\alpha}}\left(\sum_{t=1}^T\|g_t\|^{\frac{1+\alpha}{\alpha}}\right)^{\frac{2\alpha}{1+\alpha}}
    \lesssim T^{\frac{1-\alpha}{1+\alpha}}\left(\sum_{t=1}^T\Delta_t\right)^{\frac{2\alpha}{1+\alpha}}.
\]
Moreover, applying H\"older's inequality with the same exponents to $\{1\}_{t=1}^T$ and $\{\|b_t\|^2\}_{t=1}^T$, and then using Lemma~\ref{lem:sg-holder-local-error} and~\eqref{eq:sg-holder-distance}, gives
\[
    \sum_{t=1}^T\|b_t\|^2
    \leq T^{\frac{1-\alpha}{1+\alpha}}\left(\sum_{t=1}^T\|b_t\|^{\frac{1+\alpha}{\alpha}}\right)^{\frac{2\alpha}{1+\alpha}}\lesssim T^{\frac{1-\alpha}{1+\alpha}}\left(1+D_T^{\frac{\alpha(1+\alpha)}{2}}\right)^{\frac{2\alpha}{1+\alpha}}\lesssim T^{\frac{1-\alpha}{1+\alpha}}\left(1+\log\frac{q_{T+1}}{q_1}\right)^{\alpha^2}.
\]
Using $\|\bar g_t\|^2\leq2\|g_t\|^2+2\|b_t\|^2$, the above two estimates, and~\eqref{eq:sg-holder-objective-master}, we obtain
\begin{align*}
    q_{T+1}^2&=q_1^2+\frac1{\eta^2}\sum_{t=1}^T\|\bar g_t\|^2\lesssim \sum_{t=1}^T\|g_t\|^2+\sum_{t=1}^T\|b_t\|^2\\
    &\lesssim T^{\frac{1-\alpha}{1+\alpha}}q_{T+1}^{\frac{2\alpha}{1+\alpha}}\left(1+\log\frac{q_{T+1}}{q_1}\right)^{\frac{2\alpha}{1+\alpha}}+T^{\frac{1-\alpha}{1+\alpha}}\left(1+\log\frac{q_{T+1}}{q_1}\right)^{\alpha^2}.
\end{align*}
Taking square roots and using $\sqrt{a+b}\leq\sqrt a+\sqrt b$, then applying to the first term the Young's inequality
\[
    ab\leq\frac{\lambda^{p}}{p}a^{p}+\frac{\lambda^{-q}}{q}b^{q},\qquad \frac1{p}+\frac1{q}=1
\]
with $p:=(1+\alpha)/\alpha$, $q:=1+\alpha$, $\lambda:=(p/2)^{1/p}$ to isolate $q_{T+1}/2$, we derive for some constant $C>0$ that
\begin{align}
    q_{T+1}\leq& Cq_{T+1}^{\frac{\alpha}{1+\alpha}}T^{\frac{1-\alpha}{2(1+\alpha)}}\left(1+\log\frac{q_{T+1}}{q_1}\right)^{\frac{\alpha}{1+\alpha}}+CT^{\frac{1-\alpha}{2(1+\alpha)}}\left(1+\log\frac{q_{T+1}}{q_1}\right)^{\alpha^2/2}\notag\\
    \leq&\frac{q_{T+1}}2+\frac{\lambda^{-q}}{q}C^{1+\alpha} T^{\frac{1-\alpha}{2}}\left(1+\log\frac{q_{T+1}}{q_1}\right)^\alpha+C T^{\frac{1-\alpha}{2(1+\alpha)}}\left(1+\log\frac{q_{T+1}}{q_1}\right)^{\alpha^2/2}.\label{eq:holder-qT-bd1}
\end{align}
Note that the last term in~\eqref{eq:holder-qT-bd1} is bounded by the second term up to a constant since $\alpha\ge \alpha^2/2$ and $1+\alpha\geq 1$. Therefore, simplifying~\eqref{eq:holder-qT-bd1} gives
\[
    \frac{q_{T+1}}{q_1}\lesssim T^{\frac{1-\alpha}{2}}\left(1+\log\frac{q_{T+1}}{q_1}\right)^\alpha.
\]
Let $x:=(q_{T+1}/q_1)^{1/\alpha}$, then it follows from $1+\log (q_{T+1}/q_1)\ge 0$ that
\[
    x\lesssim T^{\frac{1-\alpha}{2\alpha}}(1+\alpha\log x),
\]
solving which using Lemma~\ref{lem:log-solution} and recovering to $q_{T+1}$ gives
\[
    q_{T+1}\lesssim T^{\frac{1-\alpha}{2}}(1+\log T)^\alpha.
\]
Substituting it into~\eqref{eq:sg-holder-objective-master} yields
\[
    \sum_{t=1}^T\Delta_t\lesssim T^{\frac{1-\alpha}{2}}(1+\log T)^{1+\alpha},
\]
which completes the proof by noting that $\E[F(\bar x_T)-F^\star]=T^{-1}\sum_{t=1}^T\Delta_t$.
\end{proof}

\newpage
\section{Additional Experiments}
\label{app:additional-experiments}

\subsection{Data construction and reference metrics}
\label{app:experiment-data}
The synthetic logistic problem has $m=32$ clients, $n_i=64$ examples per
client ($n=2048$), and dimension $d=32$. A deterministic Halton design is
mapped through the standard-normal quantile, centered, and orthogonalized.
Its singular values make the eigenvalues of $A^\top A/n$ a geometric sequence
from $1$ to $1/\kappa$, where $\kappa\in\{1,10,100\}$. Labels are
deterministic Bernoulli-logistic realizations from
$x_{\rm true}=2e_1-e_{11}+0.5e_{22}$. Examples are sorted by planted margin,
divided into blocks of 32, and dealt round-robin to clients to create fixed
heterogeneity.

For IJCNN1, we take the first 32,000 rows of the LIBSVM binary-classification
dataset, with $d=22$ features and $n_i=1000$ examples on each of 32 clients.
We do not append an intercept or rescale rows or columns. Examples are sorted
by margin under a high-accuracy centralized $\ell_1$-logistic reference
solution, divided into blocks of 50, and dealt round-robin to clients.
The three penalties and metrics are defined in Section~\ref{sec:numerical-experiments}.
We compute $F^\star$ to high precision using damped Newton for the
squared-$\ell_2$ problem and FISTA followed by active-set Newton polishing
for the $\ell_1$ problem. 

\subsection{Experimental setup and tuning}
\label{app:experiment-setup}

\paragraph{Run protocol.}
Exact local gradients use all $n_i$ examples. Stochastic gradients sample independently with
replacement for every (round, client, local-step) tuple. The primary experiment uses minibatch
size $b=1$; the additional panels use $b=8$ for synthetic clients and $b=32$ for IJCNN1 clients.
Stochastic curves are means over seeds 2027--2036. Exact-gradient trajectories are
deterministic and shared by the figure sets. The solid curve averages metric values along the trajectory, which represents expected value for a uniformly random iterate chosen from up to that iteration. The dashed last-iterate curves.

\paragraph{Step-size selection.}
We fix $\gamma=1$ in all experiments unless specified otherwise. For each dataset family, regime, oracle type, and stochastic batch size, we tune $\eta$ at $K=10$
and then hold it fixed across $K$. The grid consists of integer powers of ten. If the best value is on the boundary, we extend the grid until the selected point is interior. Synthetic tuning is performed only at
$\kappa=10$ and reused at $\kappa\in\{1,100\}$. The selected values are:
\begin{center}
\small
\begin{tabular}{llccc}
\toprule
Data & Regime & Exact & Stochastic $b=1$ & Additional stochastic \\
\midrule
Synthetic & Smooth nonconvex & 10 & 1 & 10 ($b=8$) \\
Synthetic & Smooth convex    & 10 & 1 &  1 ($b=8$) \\
Synthetic & Nonsmooth convex &  1 & 1 &  1 ($b=8$) \\
IJCNN1    & Smooth nonconvex & 10 & 10 & 10 ($b=32$) \\
IJCNN1    & Smooth convex    & 10 & 10 & 10 ($b=32$) \\
IJCNN1    & Nonsmooth convex & 10 & 10 & 10 ($b=32$) \\
\bottomrule
\end{tabular}
\end{center}
All selections are interior grid points. Because $\eta$ is retuned for each batch size, the
primary and additional panels are separately tuned demonstrations rather than a causal
batch-size ablation.

\subsection{AdaFed rate study on synthetic data}
\label{app:original-rate-study}
The synthetic rate study uses $T=2000$ reported rounds, full participation,
$\gamma=1$, and $K\in\{1,10,100\}$. The local schedule is $\tau_t=1$ for
$K=1$ and $\tau_t=1/((K-1)t)$ otherwise. Exact gradients use all local
examples; stochastic gradients use minibatches of size $b=1$ and pointwise
means over seeds 2027--2036. The corresponding IJCNN1 study in
Figure~\ref{fig:rates-ijcnn-b1-pgf} uses $T=10^4$ rounds.

With exact gradients, the two smooth metrics decay approximately as $1/T$,
while the nonsmooth convex gap exhibits a slower trend consistent with
$1/\sqrt{T}$. The exact-gradient curves nearly coincide across $K$. With
$b=1$, increasing $K$ from 1 to 100 improves the final synthetic
$\kappa=10$ uniform-random metric by factors $3.63$, $4.76$, and $3.57$
for smooth nonconvex, smooth convex, and nonsmooth convex objectives.

\begin{figure}[p]
\centering
\includegraphics[width=\linewidth]{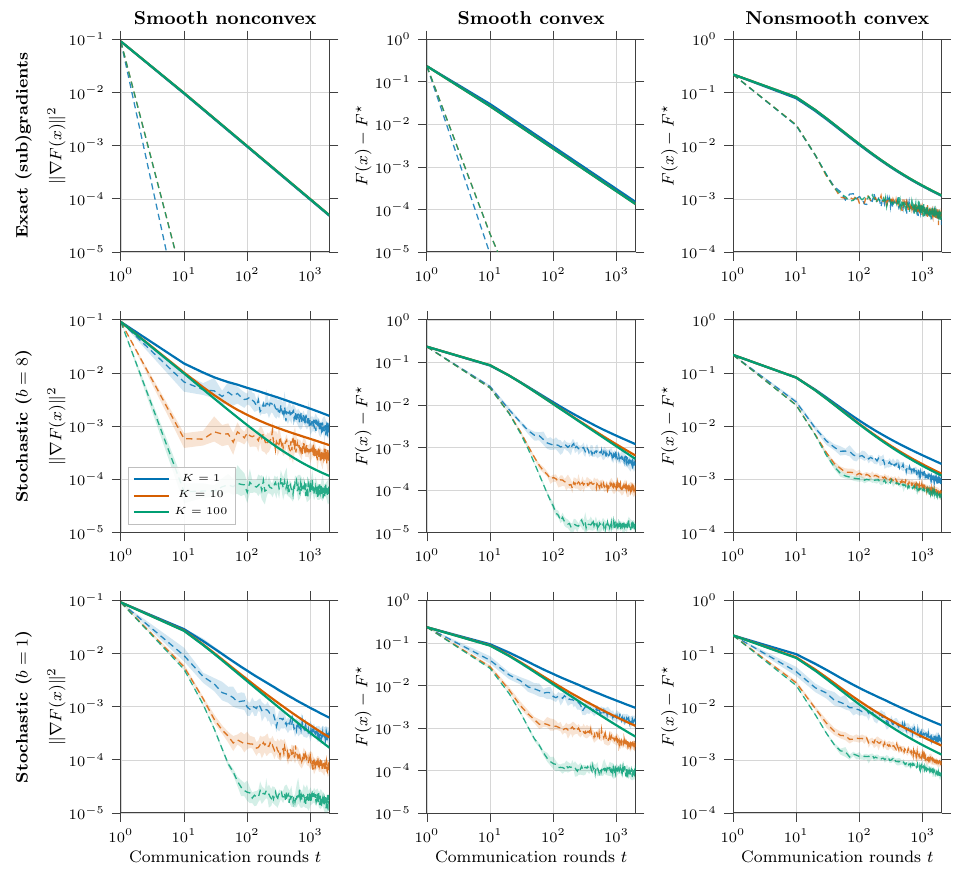}
\caption{Performance of AdaFed on synthetic logistic problems with Gram condition number $\kappa=10$ across the three problem classes and different numbers of local steps $K$. Rows show exact gradients and stochastic gradients with minibatch sizes $b=1$ and $b=8$, respectively. Solid and dashed curves show the uniform-random and last-iterate metrics, respectively. Stochastic results show the mean over ten seeds, with shaded regions indicating the upward and downward RMS deviations.}
\label{fig:rates-synthetic-b1-pgf}
\end{figure}

\begin{figure}[p]
\centering
\includegraphics[width=\linewidth]{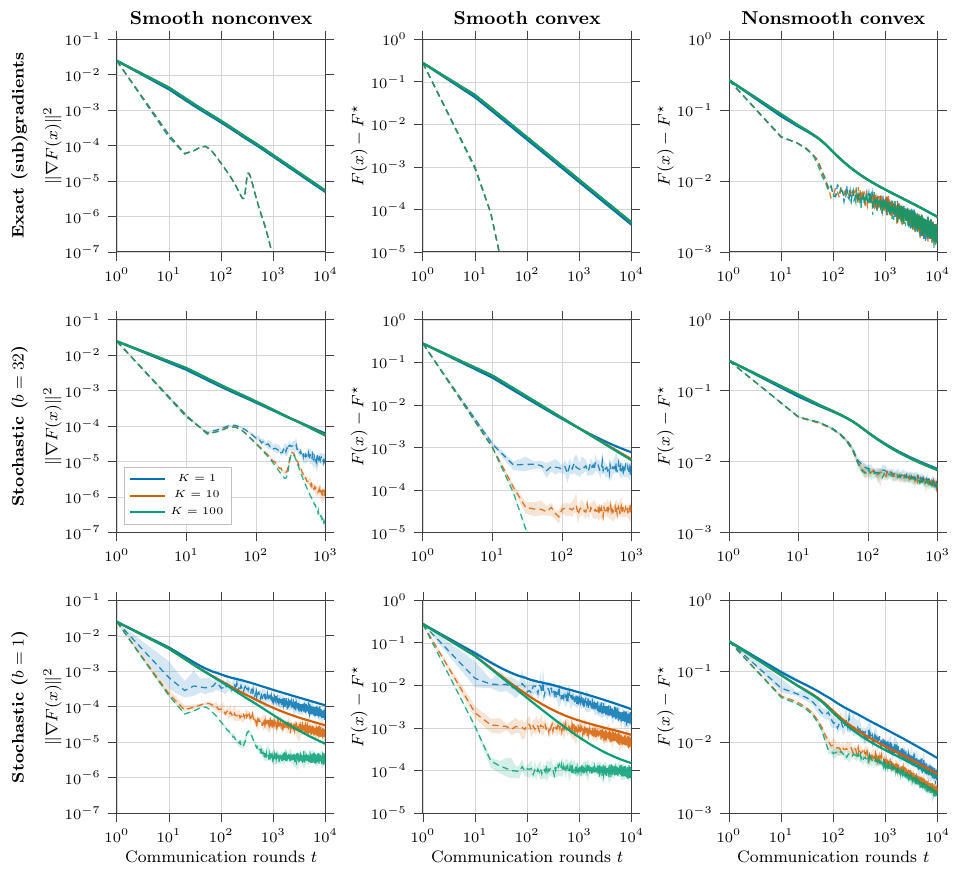}
\caption{Performance of AdaFed on IJCNN1 across the three problem classes and different numbers of local steps $K$. Rows show exact gradients and stochastic gradients with minibatch sizes $b=1$ and $b=32$, respectively. Solid and dashed curves show the uniform-random and last-iterate metrics, respectively. Stochastic results show the mean over ten seeds, with shaded regions indicating the upward and downward RMS deviations.}
\label{fig:rates-ijcnn-b32-pgf}
\end{figure}

\subsection{Sensitivity to $\eta$ and $\gamma$}

The two scalar parameters play different roles. The parameter $\eta$ sets the numerator of the
adaptive shared step, while $\gamma$ initializes its accumulator; in particular,
$\eta_1=\eta/\gamma$. We vary one parameter on an absolute power-of-ten grid while fixing the
other, with $K=10$ and full participation. Exact-gradient curves are deterministic. Stochastic
curves use $b=1$ and are pointwise means over seeds 2027--2036. 

The $\eta$ panels verify that the selected values are interior to the tested grids rather than
artifacts of a truncated search. The response to $\gamma$ is nonmonotone because changing
$\gamma$ also changes the initial step $\eta/\gamma$. On IJCNN1, $\gamma=1$ gives the smallest
final random-output metric in five of the six row--regime combinations; the exact nonconvex row
favors $0.1$. On the synthetic stochastic rows, $0.1$ has the smallest final random-output metric,
although $\gamma=1$ remains close and is used throughout the controlled rate study. Large
$\gamma$ consistently slows the initial transient.

\begin{figure}[p]
\centering
\includegraphics[width=0.87\linewidth]{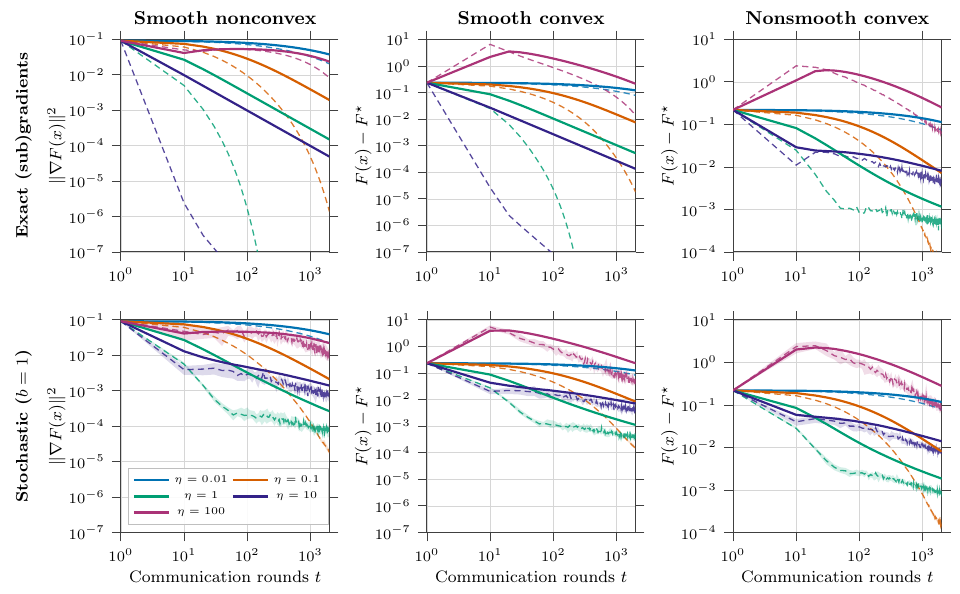}
\caption{Sensitivity of AdaFed to $\eta$ on synthetic logistic problems with Gram condition number $\kappa=10$, with $\gamma=1$ and $K=10$. Solid and dashed curves show the uniform-random and last-iterate metrics, respectively. Stochastic results show the mean over ten seeds, with shaded regions indicating the upward and downward RMS deviations.}
\label{fig:eta-synthetic-pgf}
\end{figure}

\begin{figure}[p]
\centering
\includegraphics[width=0.87\linewidth]{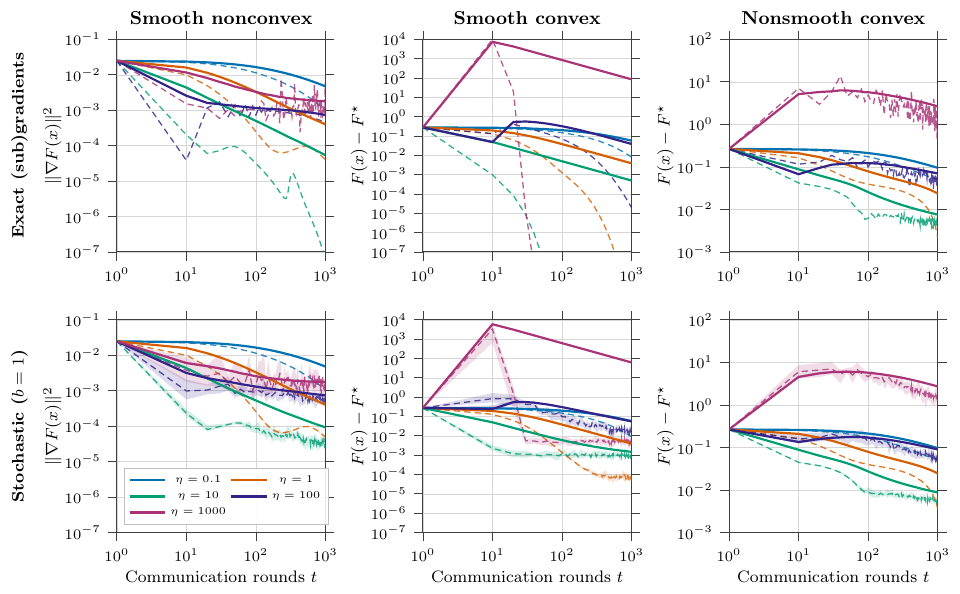}
\caption{Sensitivity of AdaFed to $\eta$ on IJCNN1, with $\gamma=1$ and $K=10$. Solid and dashed curves show the uniform-random and last-iterate metrics, respectively. Stochastic results show the mean over ten seeds, with shaded regions indicating the upward and downward RMS deviations}
\label{fig:eta-ijcnn-pgf}
\end{figure}

\begin{figure}[p]
\centering
\includegraphics[width=0.87\linewidth]{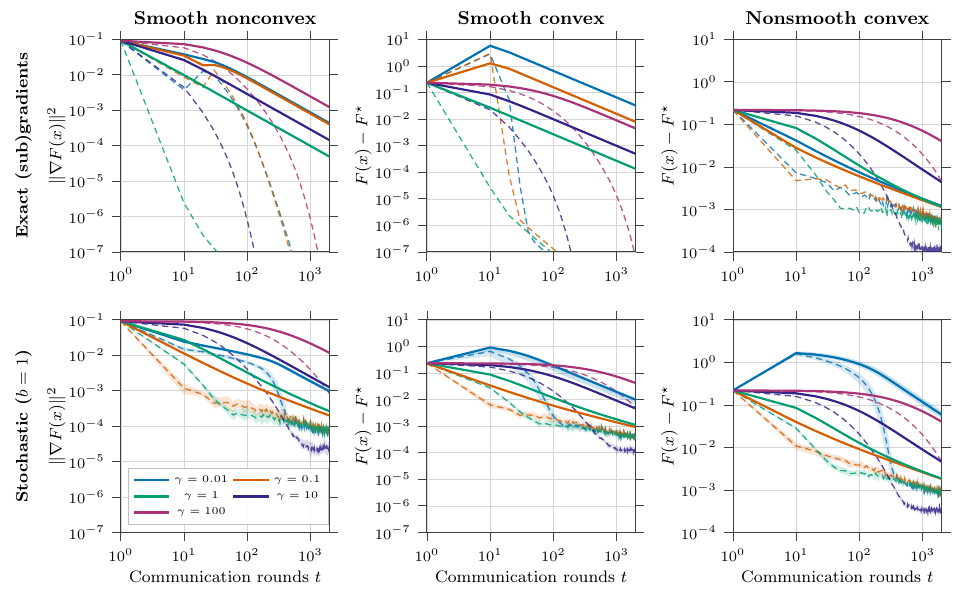}
\caption{Sensitivity of AdaFed to $\gamma$ on synthetic logistic problems with Gram condition number $\kappa=10$, with $\eta=1$ and $K=10$. Solid and dashed curves show the uniform-random and last-iterate metrics, respectively. Stochastic results show the mean over ten seeds, with shaded regions indicating the upward and downward RMS deviations.}
\label{fig:gamma-synthetic-pgf}
\end{figure}

\begin{figure}[p]
\centering
\includegraphics[width=0.87\linewidth]{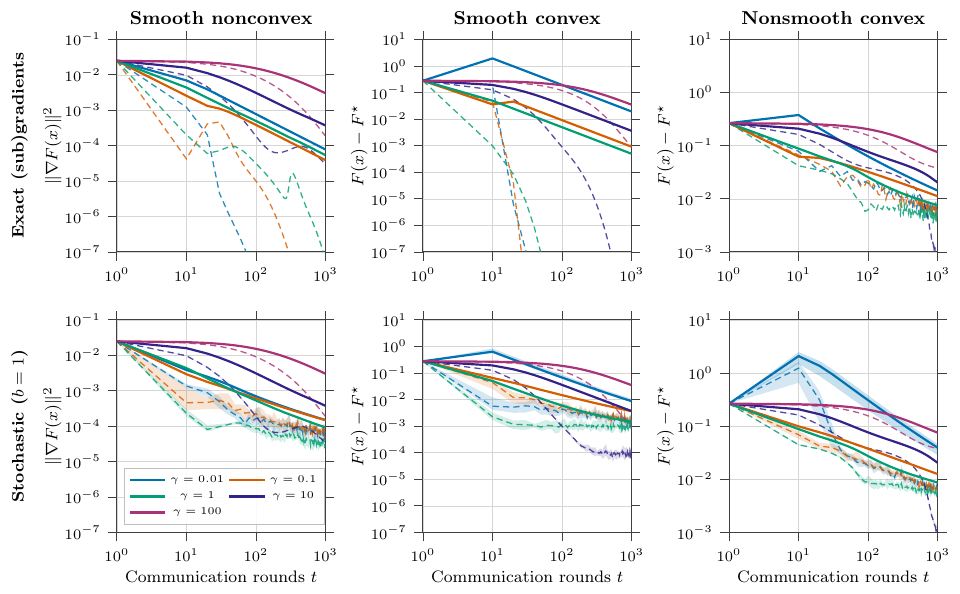}
\caption{Sensitivity of AdaFed to $\gamma$ on IJCNN1, with $\eta=1$ and $K=10$. Solid and dashed curves show the uniform-random and last-iterate metrics, respectively. Stochastic results show the mean over ten seeds, with shaded regions indicating the upward and downward RMS deviations.}
\label{fig:gamma-ijcnn-pgf}
\end{figure}

\subsection{Partial participation}

For an empirical experiment, each client participates
independently with probability $p\in\{0.25,0.5,1\}$. The server uses the same
Horvitz--Thompson-scaled update as the implementation, so its conditional mean matches the
full-participation update. We fix $K=10$, $\gamma=1$, and the corresponding tuned $\eta$.
Whenever participation is random, curves are pointwise means over seeds 2027--2036; the
full-participation exact curve is deterministic.

\begin{figure}[p]
\centering
\includegraphics[width=0.87\linewidth]{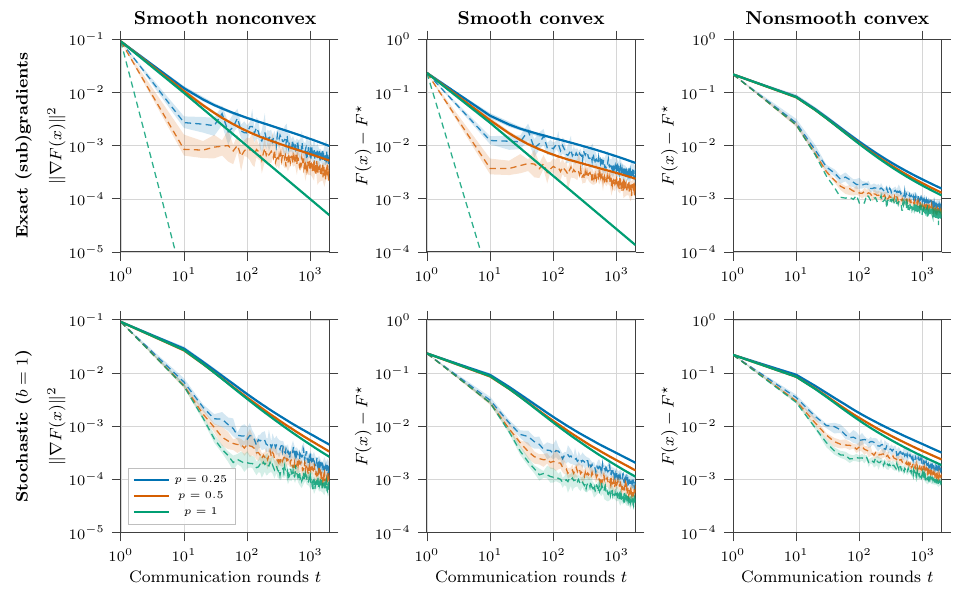}
\caption{Performance of AdaFed under partial participation on synthetic logistic problems with Gram condition number $\kappa=10$ and different participation rates $p$. Solid and dashed curves show the uniform-random and last-iterate metrics, respectively. Stochastic results show the mean over ten seeds, with shaded regions indicating the upward and downward RMS deviations.}
\label{fig:partial-synthetic-pgf}
\end{figure}

\begin{figure}[p]
\centering
\includegraphics[width=0.87\linewidth]{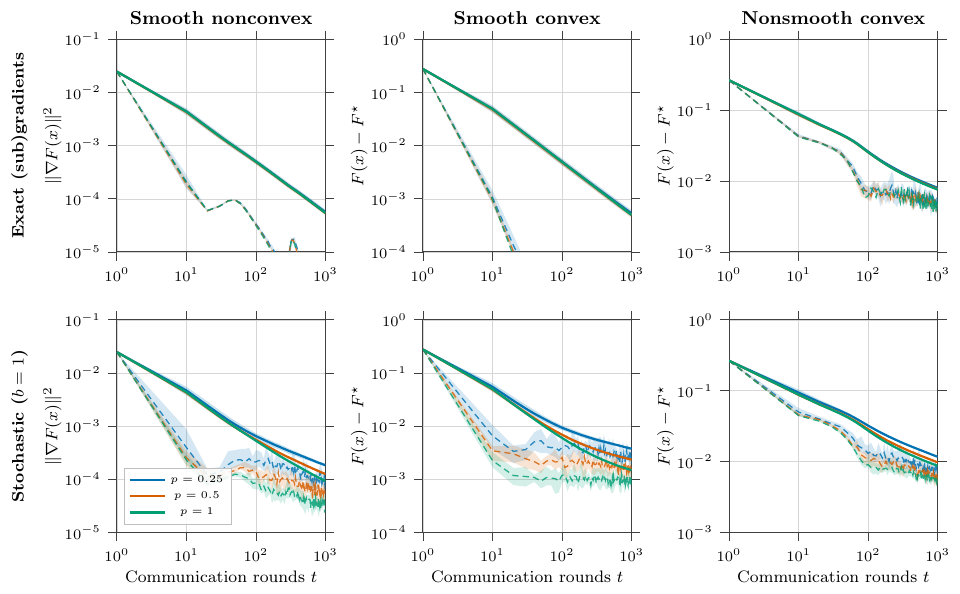}
\caption{Performance of AdaFed under partial participation on IJCNN1 with different participation rates $p$. Solid and dashed curves show the uniform-random and last-iterate metrics, respectively. Stochastic results show the mean over ten seeds, with shaded regions indicating the upward and downward RMS deviations.}
\label{fig:partial-ijcnn-pgf}
\end{figure}

Reducing participation increases the stochastic error level in every regime. At the final
checkpoint, $p=0.25$ raises the stochastic random-output metric relative to $p=1$ by factors
between $1.66$ and $1.79$ on synthetic data and between $1.35$ and $2.54$ on IJCNN1. The effect
is much larger for the exact-gradient smooth synthetic problems because client subsampling becomes
their only source of randomness.

\clearpage

\subsection{Adaptive-method comparisons}
\label{app:comparison-protocol}
These experiments use full participation of all 32 clients, $K=10$ local
steps, and $T=10^4$ rounds on synthetic data with $\kappa=1$ and
on IJCNN1. Exact local gradients use each client's
entire sample. Stochastic gradients use minibatches sampled independently
with replacement, of size $b=16$ for synthetic data and $b=32$ for IJCNN1.
The exact trajectories use seed 2027; stochastic curves are pointwise
means of seeds 2027--2029. All methods start from the zero vector and
are evaluated using the same full-objective metrics as in
Section~\ref{sec:numerical-experiments}. The solid and dashed line meanings
match Figure~\ref{fig:comparison-selected-ijcnn1-32-pgf}.

Each method is tuned separately in each dataset--objective--oracle cell.
Candidate runs use seed 1901 and 300 reported rounds on synthetic data or
150 on IJCNN1. The tuning criterion is the final cumulative-mean metric,
namely the objective gap for convex problems and squared full pooled
gradient norm for the smooth nonconvex problem. Initial Cartesian grids are
extended by a factor of ten along an axis when the selected value lies at
its edge, subject to the finite bounds in Table~\ref{tab:comparison-tuning}.
The resulting settings are held fixed throughout the 10,000-round runs;
they are short-horizon selections rather than parameters optimized for the
full horizon. The selected four-method comparison and the wider
adaptation-level comparisons reuse the same AdaFed, FedAdagrad, FedLADA,
and FAFED runs. AdaFed jointly tunes its $\eta$ and accumulator
initialization $\gamma$; its local schedule remains $\tau_t=1/((K-1)t)$.
For baselines, $c/L$ denotes a local learning rate using the pooled
logistic smoothness constant $L$.

\begin{table}[h]
\caption{Tuned quantities and finite search bounds for the adaptive-method
comparisons. $\eta_g$ is a server rate, $c$ or $c_0$ is a local scale,
$g$ is a FedSPS cap scale, and $\epsilon$ is Local AdaAlter's placeholder.}
\label{tab:comparison-tuning}
\centering\small
\begin{tabular}{@{}llp{7.1cm}@{}}
\toprule
Method & Tuned quantities & Finite search bounds \\
\midrule
AdaFed & $\eta,\gamma$ & $\eta\in[10^{-3},10^3]$; $\gamma\in[10^{-3},10^3]$ \\
FedAdagrad & $c,\eta_g$ & $c\in[1/128000,20]$, $\eta_g\in[10^{-3},100]$ \\
FedAdam & $c,\eta_g$ & $c\in[1/128000,20]$, $\eta_g\in[10^{-3},100]$ \\
FedUR & $c_0,\alpha$ & $c_0\in[10^{-4},10]$, $\alpha\in[0.1,1]$ \\
Local AdaAlter & $c,\epsilon$ & $c\in[10^{-5},100]$, $\epsilon\in[10^{-4},10]$ \\
FedSPS & $c,g$ & $c\in[10^{-3},1000]$, $g\in[10^{-4},100]$ \\
FedLADA & $c,\eta_g$ & $c\in[10^{-5},10]$, $\eta_g\in[10^{-2},10]$ \\
FAFED & $c$ & $c\in[10^{-5},0.1]$; $\eta_g=1$ fixed \\
FedAda$^2$ & $c,\eta_g$ & $c\in[10^{-6},10]$, $\eta_g\in[10^{-3},100]$ \\
\bottomrule
\end{tabular}
\end{table}

\paragraph{Initial discrete grids.}
The tested starting grids, in the order of the tuned quantities in
Table~\ref{tab:comparison-tuning}, are as follows:
\begin{itemize}[leftmargin=1.5em,topsep=2pt,itemsep=0pt]
\item AdaFed: $\eta\in\{0.01,0.1,1,10,100\}$ and
  $\gamma\in\{0.01,0.1,1,10,100\}$.
\item FedAdagrad: $c\in\{1/128,1/32,1/8,1/2,2\}$ and
  $\eta_g\in\{0.001,0.01,0.1,1,10\}$.
\item FedAdam: $c\in\{1/128,1/8,2\}$ and
  $\eta_g\in\{0.01,0.1,1,10\}$.
\item FedUR: $c_0\in\{0.001,0.01,0.1,1\}$ and
  $\alpha\in\{0.1,0.5,0.9,1\}$.
\item Local AdaAlter: $c\in\{0.001,0.01,0.1,1,10\}$ and
  $\epsilon\in\{0.01,0.1,1\}$.
\item FedSPS: $c\in\{0.1,1,10,100\}$ and
  $g\in\{0.01,0.1,1,10\}$.
\item FedLADA: $c\in\{10^{-5},10^{-4},10^{-3},10^{-2},10^{-1}\}$
  and $\eta_g\in\{0.1,1,10\}$.
\item FAFED: $c\in\{10^{-5},10^{-4},10^{-3},10^{-2},10^{-1}\}$.
\item FedAda$^2$: $c\in\{10^{-4},10^{-3},10^{-2},10^{-1},1\}$
  and $\eta_g\in\{0.01,0.1,1,10\}$.
\end{itemize}

\paragraph{Selected methods on synthetic data.}
Figure~\ref{fig:comparison-selected-synthetic-kappa1-pgf} is the synthetic
counterpart of the main-text IJCNN1 comparison. AdaFed has the smallest
final cumulative-mean metric in all six synthetic settings among this
four-method set. The nonsmooth runs of FedLADA and FAFED use the selected
$\operatorname{sign}(0)=0$ subgradient and are empirical extensions beyond
the smooth-objective assumptions in their cited analyses.

\paragraph{Server adaptation.}
Figures~\ref{fig:comparison-server-synthetic-kappa1-pgf} and
\ref{fig:comparison-server-ijcnn1-32-pgf} add FedAdam and FedUR to
FedAdagrad as server-side comparators. AdaFed has the lowest final
cumulative-mean metric in all 12 dataset--objective--oracle cells of this
group. This is an empirical comparison at equal communication rounds;
different methods may incur different local computational costs.

\paragraph{Client adaptation.}
Figures~\ref{fig:comparison-client-synthetic-kappa1-pgf} and
\ref{fig:comparison-client-ijcnn1-32-pgf} compare AdaFed with Local
AdaAlter, FedSPS, and FedLADA. AdaFed has the smallest final cumulative-mean
metric in 10 of the 12 cells. Local AdaAlter is lower in the synthetic
stochastic smooth nonconvex and smooth convex cells, with values
$3.009\times10^{-5}$ versus $5.401\times10^{-5}$ and
$9.025\times10^{-5}$ versus $1.647\times10^{-4}$, respectively.

\paragraph{Adaptation at both levels.}
Figures~\ref{fig:comparison-both-synthetic-kappa1-pgf} and
\ref{fig:comparison-both-ijcnn1-32-pgf} compare AdaFed with FAFED and
FedAda$^2$. AdaFed has the smallest final cumulative-mean metric in all
12 cells. FAFED uses about two local gradient evaluations per tracker step;
FedAda$^2$ uses coordinatewise AdaGrad at both levels with client
accumulators restarted each round. Their nonsmooth results, like those of
the other baselines, should be read as empirical tests of the implemented
subgradient extensions.

\vspace{2em}

\begin{figure}[h]
\centering
\includegraphics[width=0.8\linewidth]{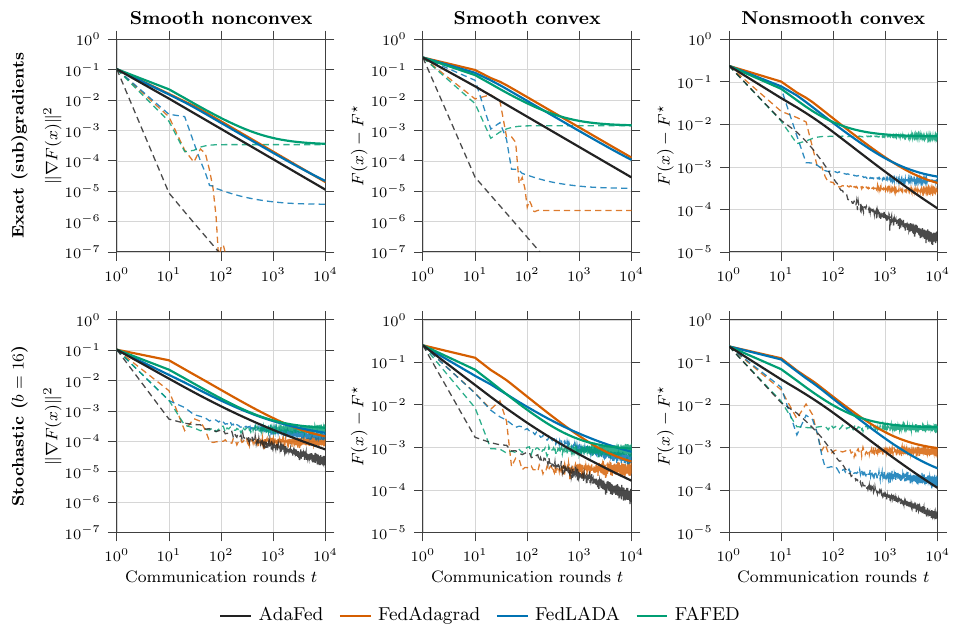}
\caption{Performance of AdaFed compared with representative adaptive FL methods on synthetic data with $\kappa=1$ and $K=10$ local steps. FedAdagrad, FedLADA, and FAFED represent server-, client-, and both-level adaptation, respectively. Solid and dashed curves show the uniform-random and last-iterate metrics, respectively. Stochastic results are shown as the mean over three seeds.}
\label{fig:comparison-selected-synthetic-kappa1-pgf}
\end{figure}

\begin{figure}[p]
\centering
\includegraphics[width=0.8\linewidth]{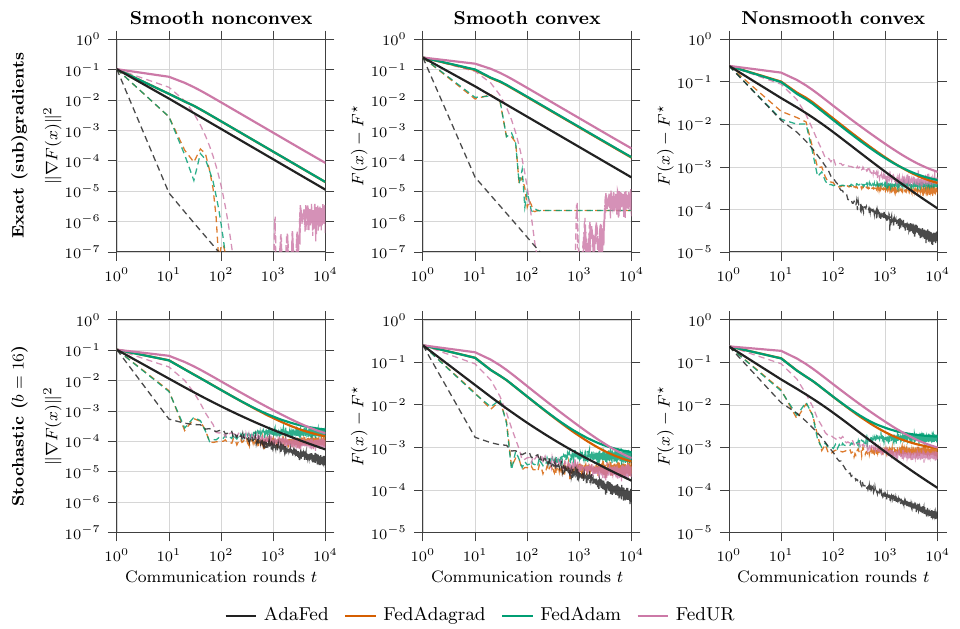}
\caption{Performance of AdaFed compared with server-adaptive FL methods on synthetic data with $\kappa=1$ and $K=10$ local steps. FedAdagrad, FedAdam, and FedUR are server-level adaptive methods. Solid and dashed curves show the uniform-random and last-iterate metrics, respectively. Stochastic results are shown as the mean over three seeds.}
\label{fig:comparison-server-synthetic-kappa1-pgf}
\end{figure}

\begin{figure}[p]
\centering
\includegraphics[width=0.8\linewidth]{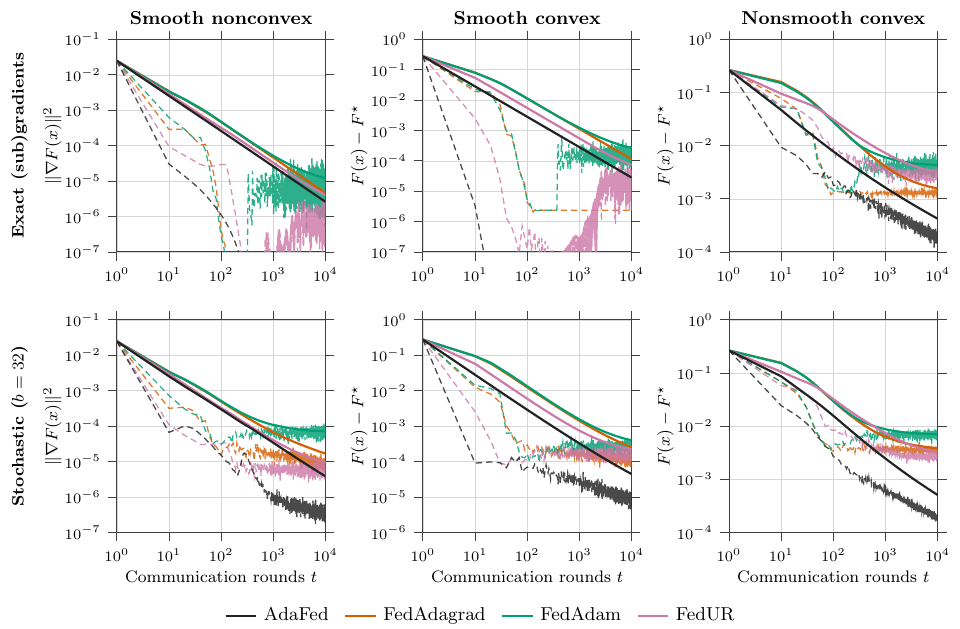}
\caption{Performance of AdaFed compared with server-adaptive FL methods on IJCNN1 with $K=10$ local steps. FedAdagrad, FedAdam, and FedUR represent server-level adaptation. Solid and dashed curves show the uniform-random and last-iterate metrics, respectively. Stochastic results are shown as the mean over three seeds.}
\label{fig:comparison-server-ijcnn1-32-pgf}
\end{figure}

\begin{figure}[p]
\centering
\includegraphics[width=0.8\linewidth]{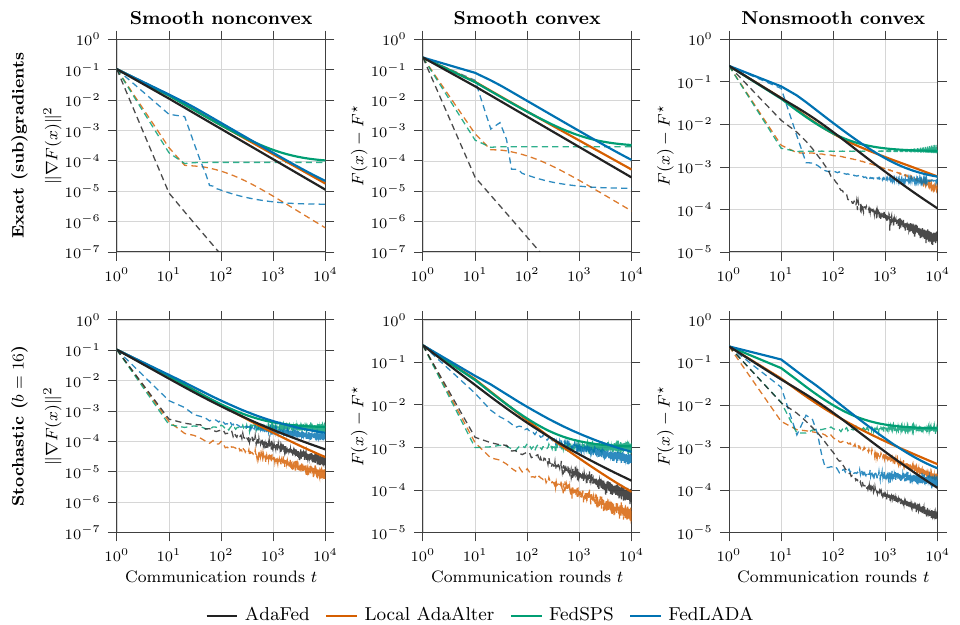}
\caption{Performance of AdaFed compared with client-adaptive FL methods on synthetic data with $\kappa=1$ and $K=10$ local steps. Local AdaAlter, FedSPS, and FedLADA represent client-level adaptation. Solid and dashed curves show the uniform-random and last-iterate metrics, respectively. Stochastic results are shown as the mean over three seeds.}
\label{fig:comparison-client-synthetic-kappa1-pgf}
\end{figure}

\begin{figure}[p]
\centering
\includegraphics[width=0.8\linewidth]{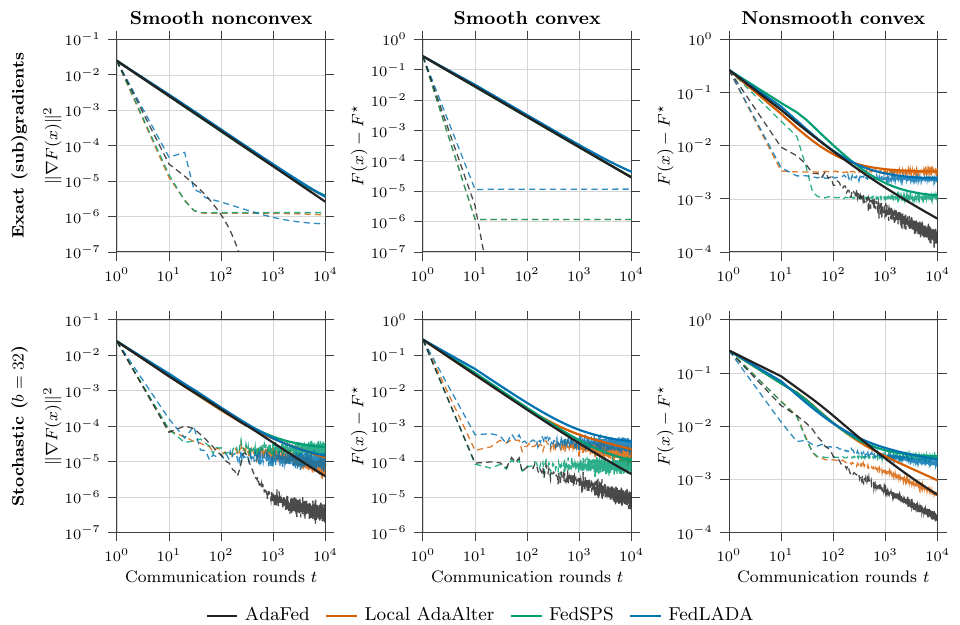}
\caption{Performance of AdaFed compared with client-adaptive FL methods on IJCNN1 with $K=10$ local steps. Local AdaAlter, FedSPS, and FedLADA represent client-level adaptation. Solid and dashed curves show the uniform-random and last-iterate metrics, respectively. Stochastic results are shown as the mean over three seeds.}
\label{fig:comparison-client-ijcnn1-32-pgf}
\end{figure}

\begin{figure}[p]
\centering
\includegraphics[width=0.8\linewidth]{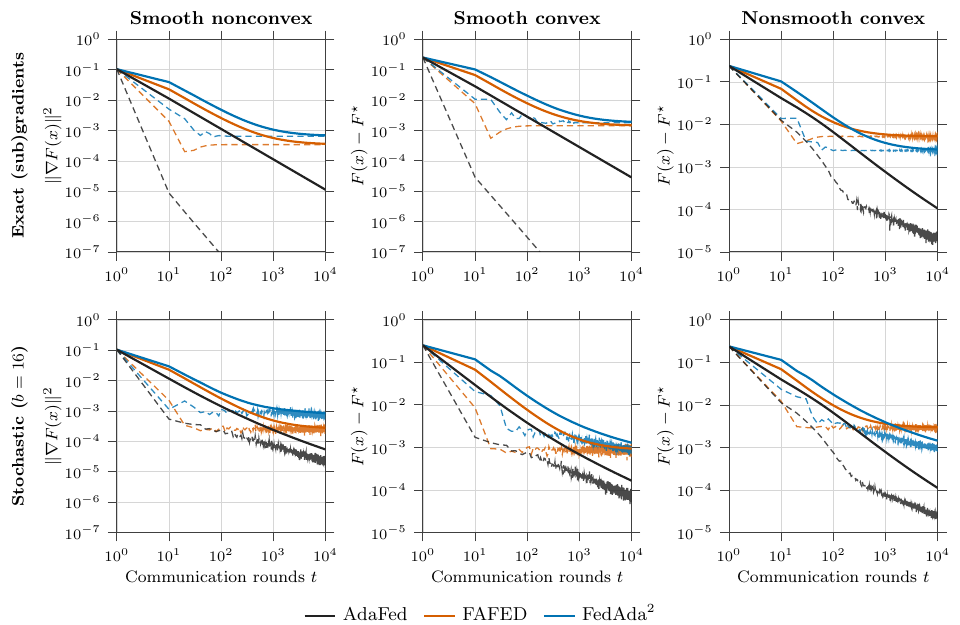}
\caption{Performance of AdaFed compared with FL methods using both server- and client-level adaptation on synthetic data with $\kappa=1$ and $K=10$ local steps. FAFED and FedAda$^2$ represent both-level adaptation. Solid and dashed curves show the uniform-random and last-iterate metrics, respectively. Stochastic results are shown as the mean over three seeds.}
\label{fig:comparison-both-synthetic-kappa1-pgf}
\end{figure}

\begin{figure}[p]
\centering
\includegraphics[width=0.8\linewidth]{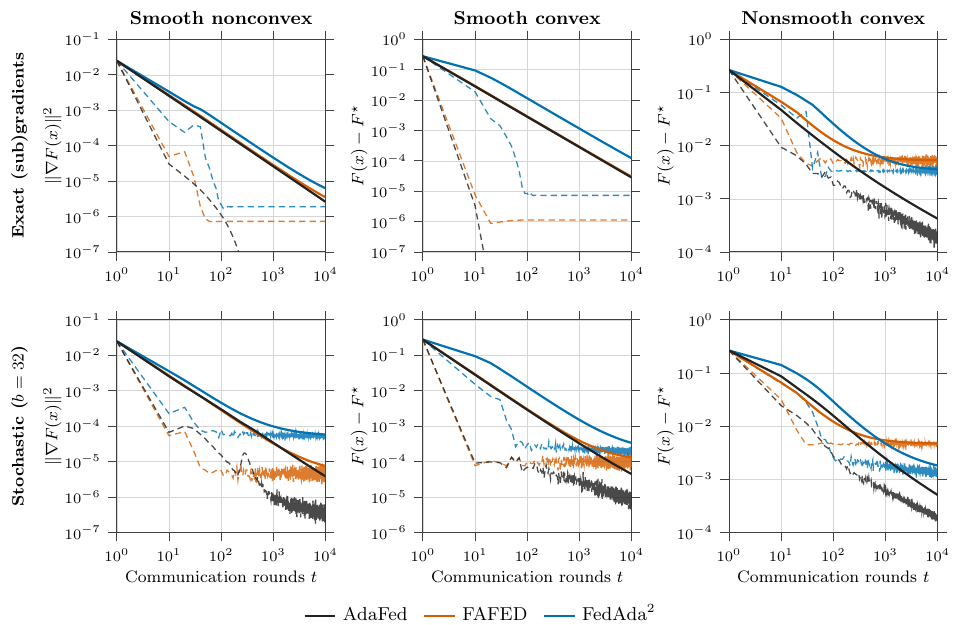}
\caption{Performance of AdaFed compared with FL methods using both server- and client-level adaptation on IJCNN1 with $K=10$ local steps. FAFED and FedAda$^2$ represent both-level adaptation. Solid and dashed curves show the uniform-random and last-iterate metrics, respectively. Stochastic results are shown as the mean over three seeds.}
\label{fig:comparison-both-ijcnn1-32-pgf}
\end{figure}

\FloatBarrier

\newpage
\NewDocumentCommand{\mildbold}{O{0.2pt} O{black} m}{%
  {\pdfrender{
    TextRenderingMode=FillStroke,
    LineWidth=#1,
    FillColor=#2,
    StrokeColor=#2
  }#3}%
}

\newcommand{\yes}{{\(\mildbold[1pt][green!45!black]{\checkmark}\)}}
\newcommand{\no}{{\(\mildbold[1pt][red!65!black]{\times}\)}}
\newcommand{\na}{\textcolor{gray!80!black}{\textbf{NA}}}
\newcommand{\hltext}[1]{\textbf{#1}}

\newcolumntype{L}[1]{>{\RaggedRight\arraybackslash}p{#1}}
\newcommand{\tblfield}[2]{\textbf{#1}\ #2\par}
\newcommand{\thickmidrule}{\specialrule{0.8pt}{3pt}{3pt}}

\section{Comparison of Adaptive FL Assumptions and Guarantees}
\label{app:adaptive-fl-assumptions}
Table~\ref{tab:adaptive-fl-literature} complements the related-work discussion with a detailed comparison of adaptive federated optimization
methods. For each result, we summarize the objective and stochastic assumptions, the participation and local-update settings, where step-size adaptation is performed, and the main theoretical guarantees. 

\begingroup
\scriptsize
\renewcommand{\arraystretch}{1.08}
\setlength{\tabcolsep}{2pt}
\setlength{\LTleft}{0pt}
\setlength{\LTright}{0pt}
\setlength{\extrarowheight}{1pt}

\begin{longtable}{@{}L{0.124\linewidth} L{0.227\linewidth} L{0.201\linewidth} L{0.187\linewidth} L{0.227\linewidth}@{}}
\caption{Detailed comparison of assumptions, adaptive mechanisms, and theoretical guarantees in related federated optimization methods.}
\label{tab:adaptive-fl-literature}\\
\toprule
\textbf{Paper / result}
& \textbf{Objective assumptions}
& \textbf{Stochastic / heterogeneity}
& \textbf{Method setting}
& \textbf{Main result / notes}
\\
\thickmidrule
\endfirsthead

\multicolumn{5}{l}{\footnotesize\tablename\ \thetable{} -- continued from previous page}\\
\toprule
\textbf{Paper / result}
& \textbf{Objective assumptions}
& \textbf{Stochastic / heterogeneity}
& \textbf{Method setting}
& \textbf{Main result / notes}
\\
\thickmidrule
\endhead

\midrule
\multicolumn{5}{r}{\footnotesize Continued on next page}\\
\endfoot

\bottomrule
\endlastfoot
\citep{wang2025adaptive}\par{\scriptsize Thm. 1}
& \tblfield{Class:}{nonconvex}\tblfield{Smoothness:}{$L$-smooth in $\|\cdot\|_\alpha$}\tblfield{Strong conv.:}{\no}\tblfield{Grad. bound:}{$\|\nabla f_n(w)\|_\infty \le C$}
& \tblfield{Variance/noise:}{$\mathbb{E}\|\xi_t\|_\alpha^\alpha \le G$}\tblfield{Heterogeneity:}{\no}
& \tblfield{Participation:}{full participation}\tblfield{Adaptation:}{\hltext{In server}}\tblfield{Local updates:}{one local step}
& Thm. 1: AdaGrad-OTA; rate $O(\log(T)/T^{1-1/\alpha})$
\\
\midrule
\citep{wang2025adaptive}\par{\scriptsize Thm. 2}
& \tblfield{Class:}{nonconvex}\tblfield{Smoothness:}{$L$-smooth in $\|\cdot\|_\alpha$}\tblfield{Strong conv.:}{\no}\tblfield{Grad. bound:}{$\|\nabla f_n(w)\|_\infty \le C$}
& \tblfield{Variance/noise:}{$\mathbb{E}\|\xi_t\|_\alpha^\alpha \le G$}\tblfield{Heterogeneity:}{\no}
& \tblfield{Participation:}{full participation}\tblfield{Adaptation:}{\hltext{In server}}\tblfield{Local updates:}{one local step}
& Thm. 2: Adam-OTA; rate $O(1/T)$
\\
\thickmidrule
\citep{chen2025gradient}\par{\scriptsize Thm. 5.2}
& \tblfield{Class:}{nonconvex}\tblfield{Smoothness:}{each $f_i$ is $L$-smooth}\tblfield{Strong conv.:}{\no}\tblfield{Grad. bound:}{\no}
& \tblfield{Variance/noise:}{$\mathbb{E}\|n_i^{(t,k)}\|^2 \le \sigma^2$}\tblfield{Heterogeneity:}{\no}
& \tblfield{Participation:}{client sampling $S^t$; minibatch}\tblfield{Adaptation:}{\hltext{In client}}\tblfield{Local updates:}{$K$ local steps}
& Thm. 5.2: FAdamGC, $\beta_2=\epsilon=0$; rate $O\!\left(\sqrt{\frac{LF\sigma^2}{nKT}}+\frac{LF}{T}\allowbreak+\frac{LK}{T}+\frac{K\sigma}{T}\right)$
\\
\midrule
\citep{chen2025gradient}\par{\scriptsize Thm. 5.4}
& \tblfield{Class:}{nonconvex}\tblfield{Smoothness:}{each $f_i$ is $L$-smooth}\tblfield{Strong conv.:}{\no}\tblfield{Grad. bound:}{$\|g_i^{(t)}\|\le G$}
& \tblfield{Variance/noise:}{$\mathbb{E}\|n_i^{(t,k)}\|^2 \le \sigma^2$}\tblfield{Heterogeneity:}{\no}
& \tblfield{Participation:}{client sampling $S^t$; minibatch}\tblfield{Adaptation:}{\hltext{In client}}\tblfield{Local updates:}{$K$ local steps}
& Thm. 5.4: FAdamGC, general Adam case; rate $O\!\Big(\sqrt{\frac{LF\sigma^2}{nKT}}+\frac{LF}{T}+\frac{KG^6}{\epsilon^2T} \allowbreak +\frac{K(\sigma^2+(1+\epsilon^2)G^2)}{\epsilon^2T}\Big)$
\\
\midrule
\citep{chen2025gradient}\par{\scriptsize Thm. C.2}
& \tblfield{Class:}{nonconvex}\tblfield{Smoothness:}{each $f_i$ is $L$-smooth}\tblfield{Strong conv.:}{\no}\tblfield{Grad. bound:}{$\|g_i^{(t)}\|\le G$}
& \tblfield{Variance/noise:}{$\mathbb{E}\|n_i^{(t,k)}\|^2 \le \sigma^2$}\tblfield{Heterogeneity:}{\no}
& \tblfield{Participation:}{client sampling $S^t$; minibatch}\tblfield{Adaptation:}{\hltext{In client}}\tblfield{Local updates:}{$K$ local steps}
& Thm. C.2: FA-NT, general Adam case; rate has worse $K^2$ final term
\\
\midrule
\citep{chen2025gradient}\par{\scriptsize Thm. C.3}
& \tblfield{Class:}{nonconvex}\tblfield{Smoothness:}{each $f_i$ is $L$-smooth}\tblfield{Strong conv.:}{\no}\tblfield{Grad. bound:}{\no}
& \tblfield{Variance/noise:}{$\mathbb{E}\|n_i^{(t,k)}\|^2 \le \sigma^2$}\tblfield{Heterogeneity:}{$\|\nabla f_i(x)-\nabla f(x)\|\le B$}
& \tblfield{Participation:}{client sampling $S^t$; minibatch}\tblfield{Adaptation:}{\hltext{In client}}\tblfield{Local updates:}{$K$ local steps}
& Thm. C.3: FA-NT, $\beta_2=\epsilon=0$; requires bounded data heterogeneity
\\
\thickmidrule
\citep{chezhegov2024local}\par{\scriptsize Thm. 1 / Cor. 2}
& \tblfield{Class:}{strongly convex}\tblfield{Smoothness:}{each $f_m$ is $L$-smooth}\tblfield{Strong conv.:}{$\mu$-strongly convex, $\mu>0$}\tblfield{Grad. bound:}{\no}
& \tblfield{Variance/noise:}{unbiased stochastic gradients; variance $\le \sigma^2$}\tblfield{Heterogeneity:}{\no}
& \tblfield{Participation:}{full participation}\tblfield{Adaptation:}{\hltext{Server-like} All clients use the same $\widehat{D}^{t_p}$.}\tblfield{Local updates:}{local steps; max gap $H$}
& Thm. 1 / Cor. 2: SAVIC for identical data; Local SGD with adaptive scaling/preconditioning
\\
\midrule
\citep{chezhegov2024local}\par{\scriptsize Thm. 2 / Cor. 3}
& \tblfield{Class:}{strongly convex}\tblfield{Smoothness:}{each $f_m(\cdot,z)$ is $L$-smooth}\tblfield{Strong conv.:}{$\mu$-strongly convex}\tblfield{Grad. bound:}{\no}
& \tblfield{Variance/noise:}{via $\sigma_{\mathrm{dif}}^2= \allowbreak \frac1M\sum_m\mathbb{E}\|\nabla f_m(x^\star,z_m)\|^2$}\tblfield{Heterogeneity:}{heterogeneity measured at $x^\star$}
& \tblfield{Participation:}{full participation}\tblfield{Adaptation:}{\hltext{Server-like} All clients use the same $\widehat{D}^{t_p}$.}\tblfield{Local updates:}{local steps; max gap $H$}
& Thm. 2 / Cor. 3: SAVIC for heterogeneous data; Local SGD with adaptive scaling/preconditioning
\\
\thickmidrule
\citep{wang2021local}\par{\scriptsize Thm. 1}
& \tblfield{Class:}{strongly convex}\tblfield{Smoothness:}{each $F_i$ is $L_i$-smooth}\tblfield{Strong conv.:}{each $F_i$ is $\mu_i$-strongly convex}\tblfield{Grad. bound:}{\no}
& \tblfield{Variance/noise:}{bounded cumulative variance $q_i(k)\sigma^2$}\tblfield{Heterogeneity:}{\no}
& \tblfield{Participation:}{full participation}\tblfield{Adaptation:}{\hltext{In client}}\tblfield{Local updates:}{$\tau_i$ local steps}
& Thm. 1: FEDOPT with local adaptive CLIENTOPT; convergence to fixed point $\tilde x$, not necessarily global minimizer.
\\
\midrule
\citep{wang2021local}\par{\scriptsize Thm. 2}
& \tblfield{Class:}{quadratic convex}\tblfield{Smoothness:}{quadratic smooth}\tblfield{Strong conv.:}{$H_i \succ 0$}\tblfield{Grad. bound:}{\no}
& \tblfield{Variance/noise:}{\no\; deterministic}\tblfield{Heterogeneity:}{\no}
& \tblfield{Participation:}{full participation}\tblfield{Adaptation:}{\hltext{In client}}\tblfield{Local updates:}{$\tau_i$ local steps}
& Thm. 2: preconditioned local GD; non-vanishing bias if $\eta_i,\tau_i,P_i$ differ across clients. The fixed point $\tilde x$ of the algo is not necessarily the global minimizer.
\\
\midrule
\citep{wang2021local}\par{\scriptsize Thm. 3}
& \tblfield{Class:}{nonconvex}\tblfield{Smoothness:}{each $F_i$ is $L$-smooth}\tblfield{Strong conv.:}{\no}\tblfield{Grad. bound:}{preconditioned gradients bounded: $\|B_i^{(t,k)}\nabla F_i(x_i^{(t,k)})\|\le G$}
& \tblfield{Variance/noise:}{\no\; deterministic}\tblfield{Heterogeneity:}{\no}
& \tblfield{Participation:}{full participation}\tblfield{Adaptation:}{\hltext{In client}}\tblfield{Local updates:}{same $\tau$ local steps}
& Thm. 3: local correction; removes non-vanishing bias; rate $O\!\left(\frac{1}{T}+\frac{\Lambda^{1/3}G^{2/3}}{T^{2/3}}\right)$
\\
\thickmidrule
\citep{kim2024adaptive}\par{\scriptsize Thm. 1}
& \tblfield{Class:}{nonconvex}\tblfield{Smoothness:}{local smoothness $\widetilde L$}\tblfield{Strong conv.:}{\no}\tblfield{Grad. bound:}{$\|\nabla f_i(x)\|\le G$}
& \tblfield{Variance/noise:}{$\mathbb{E}\|\nabla F_i(x,z)-\nabla f_i(x)\|^2\le \sigma^2$}\tblfield{Heterogeneity:}{strong-growth dissimilarity}
& \tblfield{Participation:}{partial participation}\tblfield{Adaptation:}{\hltext{In client}}\tblfield{Local updates:}{$K$ local steps}
& Thm. 1: $\Delta$-SGD; client-specific locality-adaptive step sizes; rate $O(T^{-1/2})$.
\\
\midrule
\citep{kim2024adaptive}\par{\scriptsize Thm. 5}
& \tblfield{Class:}{convex}\tblfield{Smoothness:}{locally Lipschitz gradients}\tblfield{Strong conv.:}{\no}\tblfield{Grad. bound:}{\no}
& \tblfield{Variance/noise:}{\no\; deterministic}\tblfield{Heterogeneity:}{common minimizer}
& \tblfield{Participation:}{full participation}\tblfield{Adaptation:}{\hltext{In client}}\tblfield{Local updates:}{$K=1$}
& Thm. 5: convex full-batch $\Delta$-SGD; Lyapunov convergence.
\\
\thickmidrule
\citep{lee2026efficient}\par{\scriptsize Thm. 4.1 / Cor. 4.3}
& \tblfield{Class:}{nonconvex}\tblfield{Smoothness:}{$L$-smooth local objectives}\tblfield{Strong conv.:}{\no}\tblfield{Grad. bound:}{$|[\nabla F_i(x)]_j|\le G$}
& \tblfield{Variance/noise:}{\no\; deterministic}\tblfield{Heterogeneity:}{\no}
& \tblfield{Participation:}{partial participation}\tblfield{Adaptation:}{\hltext{In server and client}}\tblfield{Local updates:}{$K$ local steps}
& Thm. 4.1 / Cor. 4.3: FedAda$^2$ and FedAda$^{2++}$; efficient joint adaptivity without transmitting preconditioners; FedAda$^{2++}$ uses memory-efficient client optimizer; rate $O(T^{-1/2})$.
\\
\midrule
\citep{lee2026efficient}\par{\scriptsize Thm. E.1}
& \tblfield{Class:}{nonconvex}\tblfield{Smoothness:}{$L$-smooth local objectives}\tblfield{Strong conv.:}{\no}\tblfield{Grad. bound:}{$|[\nabla F_i(x)]_j|\le G$}
& \tblfield{Variance/noise:}{\no\; deterministic}\tblfield{Heterogeneity:}{\no}
& \tblfield{Participation:}{partial participation}\tblfield{Adaptation:}{\hltext{In server and client}}\tblfield{Local updates:}{variable $K(O_i^\ell)$ local steps}
& Thm. E.1: federated blended optimization; allows heterogeneous client optimizer strategies such as SGD, AdaGrad, Adam, and SM3 variants, with bounded optimizer updates.
\\
\thickmidrule
\citep{li2024fedda}\par{\scriptsize Thm. 5.6 / Cor. 5.7}
& \tblfield{Class:}{nonconvex}\tblfield{Smoothness:}{sample mean-squared $L$-smoothness}\tblfield{Strong conv.:}{\no}\tblfield{Grad. bound:}{\no}
& \tblfield{Variance/noise:}{unbiased stochastic gradients; variance $\le \sigma^2$}\tblfield{Heterogeneity:}{$\|\nabla f^{(k)}(x)-\nabla f^{(\ell)}(x)\|^2\le \zeta^2$}
& \tblfield{Participation:}{full participation}\tblfield{Adaptation:}{\hltext{In server/client}}\tblfield{Local updates:}{$I$ local dual-averaging steps}
& Thm. 5.6 / Cor. 5.7: FedDA-MVR for \hltext{federated constrained optimization}; adaptive matrix fixed during local steps and updated at server; sample complexity $\widetilde O(K^{-1}\epsilon^{-1.5})$, communication complexity $\widetilde O(K^{-0.25}\epsilon^{-1.25})$. $\frac1T \sum_1^T \mathbb{E} [\mathcal{G}_t] = \tilde{\mathcal{O}} \left( T^{-2/3} K^{-2/3}\right)$
\\
\thickmidrule
\citep{liu2026lodadac}\par{\scriptsize Thm. 3.2}
& \tblfield{Class:}{nonconvex}\tblfield{Smoothness:}{$L$-smooth local objectives}\tblfield{Strong conv.:}{\no}\tblfield{Grad. bound:}{$\|g_i^t\|\le B$, $\|g_i^t\|_\infty\le B_\infty$}
& \tblfield{Variance/noise:}{unbiased stochastic gradients; bounded gradients}\tblfield{Heterogeneity:}{\no}
& \tblfield{Participation:}{all agents; decentralized neighbor communication}\tblfield{Adaptation:}{\hltext{In local agents}}\tblfield{Local updates:}{$K$ local steps per communication}
& Thm. 3.2: LoDAdaC; decentralized adaptive local training with compressed communication. Covers SGD, momentum SGD, Adam, AMSGrad, AdaGrad, and Adam-Mini; rate $O(1/\sqrt{nTK})$ under suitable conditions, with communication cost $O((1-\eta)/(nK\epsilon^4))$ for an $\epsilon$-stationary point.
\\
\thickmidrule
\citep{mukherjee2023locally}\par{\scriptsize Thm. 3(a)}
& \tblfield{Class:}{convex}\tblfield{Smoothness:}{$F_i(\cdot,\xi)$ is $L$-smooth}\tblfield{Strong conv.:}{\no}\tblfield{Grad. bound:}{\no}
& \tblfield{Variance/noise:}{stochastic gradients}\tblfield{Heterogeneity:}{finite $\sigma_f^2 = f^\star - \frac{1}{n}\sum_i \ell_i^\star$}
& \tblfield{Participation:}{full participation}\tblfield{Adaptation:}{\hltext{In client}}\tblfield{Local updates:}{$\tau$ local steps}
& Thm. 3(a): FedSPS with fully local stochastic Polyak stepsizes; sublinear convergence to a neighborhood for convex objectives.
\\
\midrule
\citep{mukherjee2023locally}\par{\scriptsize Thm. 3(b) / Cor. 5}
& \tblfield{Class:}{strongly convex}\tblfield{Smoothness:}{$F_i(\cdot,\xi)$ is $L$-smooth}\tblfield{Strong conv.:}{$\mu$-strongly convex}\tblfield{Grad. bound:}{\no}
& \tblfield{Variance/noise:}{stochastic gradients; finite $\sigma_f^2$}\tblfield{Heterogeneity:}{$\sigma_f^2 = f^\star - \frac{1}{n}\sum_i \ell_i^\star$}
& \tblfield{Participation:}{full participation}\tblfield{Adaptation:}{\hltext{In client}}\tblfield{Local updates:}{$\tau$ local steps}
& Thm. 3(b) / Cor. 5: FedSPS; linear convergence under interpolation, i.e., $\sigma_f^2=0$.
\\
\midrule
\citep{mukherjee2023locally}\par{\scriptsize Thm. 18}
& \tblfield{Class:}{\no\; nonconvex}\tblfield{Smoothness:}{$F_i(\cdot,\xi)$ is $L$-smooth}\tblfield{Strong conv.:}{\no}\tblfield{Grad. bound:}{\no}
& \tblfield{Variance/noise:}{$\mathbb{E}\|\nabla F_i(x,\xi)-\nabla f_i(x)\|^2\le \sigma^2$}\tblfield{Heterogeneity:}{$\frac1n\sum_i\|\nabla f_i(x)-\nabla f(x)\|^2\le \zeta^2$}
& \tblfield{Participation:}{full participation}\tblfield{Adaptation:}{\hltext{In client}}\tblfield{Local updates:}{$\tau$ local steps}
& Thm. 18: small-stepsize FedSPS/FedAvg for nonconvex objectives; rate for $\min_t \mathbb{E}\|\nabla f(\bar x_t)\|^2$.
\\
\thickmidrule
\citep{reddi2021adaptive}\par{\scriptsize Thm. 1 / Cor. 1}
& \tblfield{Class:}{nonconvex}\tblfield{Smoothness:}{each $F_i$ is $L$-smooth}\tblfield{Strong conv.:}{\no}\tblfield{Grad. bound:}{$|[\nabla f_i(x,z)]_j|\le G$}
& \tblfield{Variance/noise:}{local variance $\sigma_{l,j}^2$}\tblfield{Heterogeneity:}{$\frac{1}{m}\sum_i |[\nabla F_i(x)]_j-[\nabla f(x)]_j|^2\le \sigma_{g,j}^2$}
& \tblfield{Participation:}{full participation}\tblfield{Adaptation:}{\hltext{In server}}\tblfield{Local updates:}{$K$ local SGD steps}
& Thm. 1 / Cor. 1: FedAdagrad with SGD as client optimizer and adaptive server optimizer; rate essentially $O(1/\sqrt{mKT})$ for suitable $\eta_l,\eta,\tau$.
\\
\midrule
\citep{reddi2021adaptive}\par{\scriptsize Thm. 2 / Cor. 2}
& \tblfield{Class:}{nonconvex}\tblfield{Smoothness:}{each $F_i$ is $L$-smooth}\tblfield{Strong conv.:}{\no}\tblfield{Grad. bound:}{$|[\nabla f_i(x,z)]_j|\le G$}
& \tblfield{Variance/noise:}{local variance $\sigma_{l,j}^2$}\tblfield{Heterogeneity:}{$\frac{1}{m}\sum_i |[\nabla F_i(x)]_j-[\nabla f(x)]_j|^2\le \sigma_{g,j}^2$}
& \tblfield{Participation:}{full participation}\tblfield{Adaptation:}{\hltext{In server}}\tblfield{Local updates:}{$K$ local SGD steps}
& Thm. 2 / Cor. 2: FedAdam; FedYogi follows similarly. Server uses Adam/Yogi-style adaptive preconditioning on the pseudo-gradient from client model differences.
\\
\thickmidrule
\citep{sun2023efficient}\par{\scriptsize Thm. 3.5 / Cor. 3.6}
& \tblfield{Class:}{nonconvex}\tblfield{Smoothness:}{$F_i$ is $L$-smooth}\tblfield{Strong conv.:}{\no}\tblfield{Grad. bound:}{$\|g_i^t\|_\infty \le G_g$}
& \tblfield{Variance/noise:}{unbiased stochastic gradients; variance $\le \sigma_l^2$}\tblfield{Heterogeneity:}{$\|\nabla F_i(x)-\nabla F(x)\|^2 \le \sigma_g^2$}
& \tblfield{Participation:}{partial participation; $S$ clients}\tblfield{Adaptation:}{\hltext{In client}}\tblfield{Local updates:}{$K$ local adaptive steps}
& Thm. 3.5 / Cor. 3.6: FedLADA; local adaptive amended optimizer with global offset correction $g_a^t$; rate $O\!\left(\frac{1}{\sqrt{SKT}}+\frac{1}{KT}+\frac{1}{T}+\frac{1}{T^{3/2}}\right)$.
\\
\thickmidrule
\citep{takakura2025feddua}\par{\scriptsize Thm. 4.3 / Cor. 4.4}
& \tblfield{Class:}{convex}\tblfield{Smoothness:}{$F_i$ is $L$-smooth}\tblfield{Strong conv.:}{\no}\tblfield{Grad. bound:}{\no}
& \tblfield{Variance/noise:}{\no\; deterministic}\tblfield{Heterogeneity:}{$\frac{1}{M}\sum_{i} \|\nabla F_i(w^\star)\|^2 \le \sigma_\star^2$}
& \tblfield{Participation:}{full participation}\tblfield{Adaptation:}{\hltext{In server}}\tblfield{Local updates:}{$\tau$ local SGD steps}
& Thm. 4.3 / Cor. 4.4: FedDuA / FedDuAdagrad; server-side doubly adaptive global step-size based on both inter-client and coordinate-wise heterogeneity; convergence for convex objectives with full-batch local gradient steps.
\\
\thickmidrule
\citep{wu2023faster}\par{\scriptsize Thm. 2 / Rem. 1}
& \tblfield{Class:}{nonconvex}\tblfield{Smoothness:}{$f_i(\cdot;\xi)$ is $L$-smooth}\tblfield{Strong conv.:}{\no}\tblfield{Grad. bound:}{$\|\nabla f_i(x)\|\le G$}
& \tblfield{Variance/noise:}{unbiased stochastic gradients; variance $\le \sigma^2$}\tblfield{Heterogeneity:}{$\|\nabla f_i(x)-\nabla f_j(x)\|^2\le \zeta^2$}
& \tblfield{Participation:}{full participation; cross-silo}\tblfield{Adaptation:}{\hltext{In client/server}}\tblfield{Local updates:}{periodic local adaptive steps; sync every $q$ steps}
& Thm. 2 / Rem. 1: FAFED with shared adaptive learning rates and momentum-based variance reduction; adaptive matrix $A_t=\mathrm{diag}(\sqrt{\bar v_t}+\rho)$; sample complexity $\widetilde O(\epsilon^{-3})$ and communication complexity $\widetilde O(\epsilon^{-2})$.
\\
\thickmidrule
\citep{xie2020local}\par{\scriptsize Thm. 1}
& \tblfield{Class:}{nonconvex}\tblfield{Smoothness:}{$F$ and $F_i$ are $L$-smooth}\tblfield{Strong conv.:}{\no}\tblfield{Grad. bound:}{$\|G_{i,t}\|_\infty \le \rho$}
& \tblfield{Variance/noise:}{\no\; no explicit variance bound}\tblfield{Heterogeneity:}{\no}
& \tblfield{Participation:}{full participation}\tblfield{Adaptation:}{\hltext{In client}}\tblfield{Local updates:}{$H$ local adaptive steps}
& Thm. 1: Local AdaAlter; AdaGrad-style local SGD with lazy denominator updates and synchronization every $H$ iterations; convergence rate $O(1/(\eta\sqrt{T})) + O(\eta^2H^2\log(T)/\sqrt{T}) + O(\eta\log(T)/(n\sqrt{T}))$.
\\
\thickmidrule
\citep{wang2022communication}\par{\scriptsize Thm. 4.4 / 4.9}
& \tblfield{Class:}{nonconvex}\tblfield{Smoothness:}{$F_i$ is $L$-smooth}\tblfield{Strong conv.:}{\no}\tblfield{Grad. bound:}{$\|\nabla f_i(x,\xi)\|\le G$}
& \tblfield{Variance/noise:}{local variance $\le \sigma_l^2$}\tblfield{Heterogeneity:}{$\frac{1}{m}\sum_{i=1}^m\|\nabla F_i(x)-\nabla f(x)\|^2\le \sigma_g^2$}
& \tblfield{Participation:}{full/partial participation}\tblfield{Adaptation:}{\hltext{In server}}\tblfield{Local updates:}{$K$ local SGD steps}
& Thm. 4.4 / 4.9: FedAMS; AMSGrad-type server adaptive optimizer with max stabilization; rate $O(1/\sqrt{TKm})$ under full participation and partial-participation extension.
\\
\midrule
\citep{wang2022communication}\par{\scriptsize Thm. 4.17 / Thm. B.1}
& \tblfield{Class:}{nonconvex}\tblfield{Smoothness:}{$F_i$ is $L$-smooth}\tblfield{Strong conv.:}{\no}\tblfield{Grad. bound:}{$\|\nabla f_i(x,\xi)\|\le G$}
& \tblfield{Variance/noise:}{local variance $\le \sigma_l^2$}\tblfield{Heterogeneity:}{$\frac{1}{m}\sum_{i=1}^m\|\nabla F_i(x)-\nabla f(x)\|^2\le \sigma_g^2$}
& \tblfield{Participation:}{full/partial participation}\tblfield{Adaptation:}{\hltext{In server}}\tblfield{Local updates:}{$K$ local SGD steps}
& Thm. 4.17 / Thm. B.1: FedCAMS; communication-compressed FedAMS with error feedback and biased compressor $\mathbb{E}\|\mathcal{C}(x)-x\|\le q\|x\|$; same full-participation rate $O(1/\sqrt{TKm})$ as FedAMS.
\\
\thickmidrule
\citep{wang2024fadas}\par{\scriptsize Thm. 5.6 / Cor. 5.7}
& \tblfield{Class:}{nonconvex}\tblfield{Smoothness:}{$F_i$ is $L$-smooth}\tblfield{Strong conv.:}{\no}\tblfield{Grad. bound:}{$\|\nabla F_i(x;\xi)\|\le G$}
& \tblfield{Variance/noise:}{unbiased stochastic gradients; variance $\le \sigma^2$}\tblfield{Heterogeneity:}{$\frac1N\sum_{i=1}^N\|\nabla F_i(x)-\nabla f(x)\|^2\le \sigma_g^2$}
& \tblfield{Participation:}{asynchronous; buffer size $M$; uniform arrivals}\tblfield{Adaptation:}{\hltext{In server}}\tblfield{Local updates:}{$K$ local SGD steps}
& Thm. 5.6 / Cor. 5.7: FADAS; asynchronous adaptive FL with AMSGrad-type server update; rate $O\!\left(\frac{1}{\sqrt{TM}}+\frac{\tau_{\max}\tau_{\mathrm{avg}}}{T}\right)$.
\\
\midrule
\citep{wang2024fadas}\par{\scriptsize Thm. 5.9 / Cor. 5.10}
& \tblfield{Class:}{nonconvex}\tblfield{Smoothness:}{$F_i$ is $L$-smooth}\tblfield{Strong conv.:}{\no}\tblfield{Grad. bound:}{$\|\nabla F_i(x;\xi)\|\le G$}
& \tblfield{Variance/noise:}{unbiased stochastic gradients; variance $\le \sigma^2$}\tblfield{Heterogeneity:}{$\frac1N\sum_{i=1}^N\|\nabla F_i(x)-\nabla f(x)\|^2\le \sigma_g^2$}
& \tblfield{Participation:}{asynchronous; buffer size $M$; uniform arrivals}\tblfield{Adaptation:}{\hltext{In server; delay-adaptive}}\tblfield{Local updates:}{$K$ local SGD steps}
& Thm. 5.9 / Cor. 5.10: delay-adaptive FADAS; server learning rate adjusted by maximum delay $\tau_t^{\max}$; reduces dependence on worst-case delay $\tau_{\max}$.
\\
\thickmidrule
\citep{yuan2025aggregated}\par{\scriptsize Thm. 1--3}
& \tblfield{Class:}{nonconvex}\tblfield{Smoothness:}{$\beta$-Lipschitz smooth}\tblfield{Strong conv.:}{\no}\tblfield{Grad. bound:}{$\|\nabla F_i\|_2 \le P$}
& \tblfield{Variance/noise:}{\no\; deterministic}\tblfield{Heterogeneity:}{\no}
& \tblfield{Participation:}{full participation}\tblfield{Adaptation:}{\hltext{In client}}\tblfield{Local updates:}{\no\; one local update in theory}
& Thm. 1--3: FedAgile; aggregated-gradient adaptive learning rate with mean-field terms and JS-distance refinement; convergence rate $\widetilde O(T^{-1})$.
\\
\thickmidrule
\citep{zhang2023fedur}\par{\scriptsize Thm. IV.1--IV.3}
& \tblfield{Class:}{convex}\tblfield{Smoothness:}{$\beta$-smooth}\tblfield{Strong conv.:}{\no}\tblfield{Grad. bound:}{$F$ is $\rho$-Lipschitz}
& \tblfield{Variance/noise:}{\no\; no explicit variance bound}\tblfield{Heterogeneity:}{$\|G_k(w)-G(w)\|\le \delta_k$}
& \tblfield{Participation:}{partial participation; $K$ selected devices}\tblfield{Adaptation:}{\hltext{In server}}\tblfield{Local updates:}{$\tau$ local iterations}
& Thm. IV.1--IV.3: FedUR; adaptive centralized-learning optimizer for FL with client-side momentum and piecewise adaptive learning rate in global aggregation; analyzes communication-frequency and convergence-rate bounds.
\\
\end{longtable}
\endgroup

\end{document}